\documentclass[11pt]{amsart}
\usepackage{amssymb,amsfonts,amsmath}
\usepackage{enumerate} 
\usepackage{psfrag}
\usepackage[dvips]{graphicx} 
\usepackage{color}
\newtheorem{theo}{Theorem}[section]
\newtheorem{prop}[theo]{Proposition}
\newtheorem{lemma}[theo]{Lemma}
\newtheorem{coro}[theo]{Corollary}
\newtheorem{claim}[theo]{Claim}
\newtheorem{construct}[theo]{Construction}

\newcommand{\sm}{\setminus}
\newcommand{\eps}{\varepsilon}

\newcommand{\Z}{{\mathbb Z}}

\newcommand{\Part}{{\mathcal P}}
\newcommand{\Qart}{{\mathcal Q}}
\newcommand{\ab}{{\bf a}}
\newcommand{\ib}{{\bf i}} 
\newcommand{\vb}{{\bf v}}
\newcommand{\ub}{{\bf u}}

\newcommand{\CONFIG}{\mathcal{C}}
\newcommand{\mc}[1]{\mathcal{#1}} 

\newcommand{\nib}[1]{\noindent {\bf #1}}

\newcommand{\bsize}[1]{\left| #1 \right|}
\newcommand{\bfl}[1]{\left\lfloor #1 \right\rfloor}
\newcommand{\bcl}[1]{\left\lceil #1 \right\rceil}
\newcommand{\bgen}[1]{\left\langle #1 \right\rangle}
\newcommand{\sub}{\subseteq}

\newcommand{\es}{\emptyset}

\newcommand{\unit}{{\bf u}}

\def\noproof{{\unskip\nobreak\hfill\penalty50\hskip2em\hbox{}\nobreak\hfill%
       $\square$\parfillskip=0pt\finalhyphendemerits=0\par}\goodbreak}
\def\endproof{\noproof\bigskip}
\def\proof{\removelastskip\penalty55\medskip\noindent{\bf Proof. }}
\def\COMMENT#1{}
\let\COMMENT=\footnote
\newdimen\margin   % needed for macros \textdisplay & \ltextdisplay
\def\textno#1&#2\par{%
   \margin=\hsize
   \advance\margin by -4\parindent
          \setbox1=\hbox{\sl#1}%
   \ifdim\wd1 < \margin
      $$\box1\eqno#2$$%
   \else
      \bigbreak
      \hbox to \hsize{\indent$\vcenter{\advance\hsize by -3\parindent
      \sl\noindent#1}\hfil#2$}%
      \bigbreak
   \fi}
\title{A multipartite Hajnal-Szemer\'edi theorem}
\author{Peter Keevash and Richard Mycroft}
\thanks{Research supported in part by ERC grant 239696 and EPSRC grant EP/G056730/1.}
\date{\today}

\begin{document}

\vspace*{-0.8cm}
\begin{abstract}
The celebrated Hajnal-Szemer\'edi theorem gives the precise minimum degree threshold that forces a graph to contain a perfect $K_k$-packing. Fischer's conjecture states that the analogous result holds for all multipartite graphs except for those formed by a single construction. Recently, we deduced an approximate version of this conjecture from new results on perfect matchings in hypergraphs. In this paper, we apply a stability analysis to the extremal cases of this argument, thus showing that the exact conjecture holds for any sufficiently large graph.
\end{abstract}
\maketitle
\vspace*{-0.6cm}

\section{Introduction} \label{sec:intro}
A fundamental result of Extremal Graph Theory is the Hajnal-Szemer\'edi theorem, which states that if $k$ divides $n$ then any graph $G$ on~$n$ vertices with minimum degree $\delta(G) \geq (k-1)n/k$ contains a perfect $K_k$-packing, i.e.~a spanning collection of vertex-disjoint $k$-cliques. This paper considers a conjecture of Fischer~\cite{F} on a multipartite analogue 
of this theorem. Suppose $V_1, \dots, V_k$ are disjoint sets of $n$ vertices each, and $G$ is a $k$-partite graph on vertex classes $V_1, \dots, V_k$ (that is, $G$ is a graph on the vertex set $V_1 \cup \dots \cup V_k$ such that no edge of $G$ has both endvertices in the same $V_j$). We define the \emph{partite minimum degree} of $G$, denoted $\delta^*(G)$, to be the largest $m$
such that every vertex has at least $m$ neighbours in each part other than its own, i.e.\
$$\delta^*(G) := \min_{i \in [k]} \min_{v \in V_i} \min_{j \in [k]\sm\{i\}} |N(v) \cap V_j|,$$
where $N(v)$ denotes the neighbourhood of $v$. Fischer conjectured that if $\delta^*(G) \ge (k-1)n/k$ then $G$ has a perfect $K_k$-packing. This conjecture is straightforward for $k=2$, as it is not hard to see that any maximal matching must be perfect in this case, but for odd $k \geq 3$ the conjecture does not hold, as can be seen from constructions provided by Catlin~\cite{C} (these counterexamples are presented in Construction~\ref{fischereg} as $\Gamma_{n, k, k}$). For $k=3$, Magyar and Martin~\cite{MM} proved that, for large $n$, Catlin's construction is in fact the only counterexample to this conjecture. More precisely, they showed that if $n$ is sufficiently large, $G$ is a $3$-partite graph with vertex classes each of size~$n$ and $\delta^*(G) \geq 2n/3$, then either~$G$ contains a perfect $K_3$-packing, or $n$ is odd and divisible by $3$, and $G$ is isomorphic to the graph $\Gamma_{n, 3, 3}$ defined in Construction~\ref{fischereg}.

The implicit conjecture behind this result (stated explicitly by K\"uhn and Osthus~\cite{KO}) is that the only counterexamples to Fischer's original conjecture are those constructed by Catlin, that is, the graphs $\Gamma_{n, k, k}$ defined in Construction~\ref{fischereg} when $n$ is odd and divisible by~$k$. We refer to this as the modified Fischer conjecture. If $k$ is even then $n$ cannot be both odd and divisible by~$k$, so the modified Fischer conjecture is the same as the original conjecture in this case. Martin and Szemer\'edi~\cite{MSz} proved that (the modified) Fischer's conjecture holds for $k=4$. Another partial result was obtained by Csaba and Mydlarz~\cite{Cs}, who gave a function $f(k)$ with $f(k) \to 0$ as $k \to \infty$ such that the conjecture holds for large $n$ if one strengthens the degree assumption to  $\delta^*(G) \geq (k-1)n/k + f(k)n$. Recently, an approximate version of the conjecture assuming the degree condition $\delta^*(G) \geq (k-1)n/k + o(n)$ was proved independently and simultaneously by Keevash and Mycroft~\cite{KM}, and by Lo and Markstr\"om~\cite{LM}. The proof in \cite{KM} was a quick application of the geometric theory of hypergraph matchings developed in the same paper; this will be formally introduced in the next section. By a careful analysis of the extremal cases of this result, we will obtain the following theorem, the case $r=k$ of which shows that (the modified) Fischer's conjecture holds for any sufficiently large graph. Note that the graph $\Gamma_{n, r, k}$ in the statement is defined in Construction~\ref{fischereg}.

\begin{theo} \label{partitehajnalszem}
For any $r \geq k$ there exists $n_0$ such that for any $n \geq n_0$ with $k \mid rn$ the following statement holds.
Let $G$ be an $r$-partite graph whose vertex classes each have size~$n$ such that  $\delta^*(G) \geq (k-1)n/k$. 
Then $G$ contains a perfect $K_k$-packing, unless $rn/k$ is odd, $k \mid n$, and $G \cong \Gamma_{n, r, k}$.
\end{theo} 

We now give the generalised version of the construction of Catlin~\cite{C} which shows Fischer's original conjecture to be false.

\begin{construct} \label{fischereg} Suppose $rn/k$ is odd and $k$ divides $n$. Let $V$ be a vertex set partitioned into parts $V_1,\dots,V_r$ of size $n$. Partition each $V_i$, $i \in [r]$ into subparts $V^j_i$, $j \in [k]$ of size $n/k$. Define a graph $\Gamma_{n, r, k}$, where for each $i,i' \in [r]$ with $i \ne i'$ and $j \in [k]$, if $j \geq 3$ then any vertex in $V^j_i$ is adjacent to all vertices in $V^{j'}_{i'}$ with $j' \in [k] \sm \{j\}$, and if $j=1$ or $j=2$ then any vertex in $V^j_i$ is adjacent to all vertices in $V^{j'}_{i'}$ with $j' \in [k] \sm \{3-j\}$.
\end{construct}

\begin{figure}[ht] 
\centering
\psfrag{1}{$V_1$}
\psfrag{2}{$V_2$}
\psfrag{3}{$V_3$}
\psfrag{4}{$\Big\}V^1$}
\psfrag{5}{$\Big\}V^2$}
\psfrag{6}{$\Big\}V^3$}
\includegraphics[width=6cm]{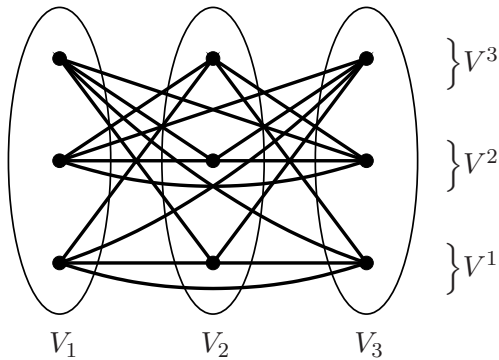}
\caption{Construction~\ref{fischereg} for the case $k=r=3$.}
\label{fig:construct}
\end{figure}

Figure~\ref{fig:construct} shows Construction~\ref{fischereg} for the case $k=r=3$. For $n=k$ this is the exact graph of the construction; for larger $n$ we `blow up' the graph above, replacing each vertex by a set of size $n/k$, and each edge by a complete bipartite graph between the corresponding sets. In general, it is helpful to picture the construction as an $r$ by $k$ grid, with columns corresponding to parts $V_i$, $i \in [r]$ and rows $V^j = \bigcup_{i \in [r]} V^j_i$, $j \in [k]$ corresponding to subparts of the same superscript. Vertices have neighbours in other rows and columns to their own, except in rows $V^1$ and $V^2$, where vertices have neighbours in other columns in their own row and other rows besides rows $V^1$ and $V^2$. Thus $\delta^*(G) = (k-1)n/k$. We claim that there is no perfect $K_k$-packing. For any $K_k$ has at most one vertex in any $V^j$ with $j \geq 3$, so at most $k-2$ vertices in $\bigcup_{j \geq 3} V^j$. Also $\bsize{\bigcup_{j \geq 3} V^j}=(k-2)rn/k$, and there are $rn/k$ copies of $K_k$ in a perfect packing. Thus each $K_k$ must have $k-2$ vertices in $\bigcup_{j \geq 3} V^j$, and so $2$ vertices in $V^1 \cup V^2$, which must either both lie in $V^1$ or both lie in $V^2$. However, $|V^1| = rn/k$ is odd, so $V^1$ cannot be perfectly covered by pairs. Thus $G$ contains no perfect $K_k$-packing.

This paper is organised as follows. In the next section we introduce ideas and results from~\cite{KM} on perfect matchings in $k$-graphs. Section~\ref{sec:outline} gives an outline of the proof of Theorem~\ref{partitehajnalszem}. In Sections~\ref{sec:blocks} to~\ref{sec:proof2} we prove several preliminary lemmas, before combining these lemmas in Section~\ref{proof} to prove Theorem~\ref{partitehajnalszem}.

\medskip

\nib{Notation.} The following notation is used throughout the paper: $[k]=\{1, \dots, k\}$; if $X$ is a set then $\binom{X}{k}$ is
the set of subsets of $X$ of size $k$; $x \ll y$ means that for every $y > 0$ there exists some $x_0 > 0$ such that the subsequent
statement holds for any $x < x_0$ (such statements with more variables are defined similarly); if $x$ is a vertex in a graph then $N(x)$ is the neighbourhood of $x$.

\section{Perfect matchings in hypergraphs} \label{sec:theory}

In this section we describe the parts of the geometric theory of perfect matchings in hypergraphs from~\cite{KM} that we will use in the proof of Theorem~\ref{partitehajnalszem}. We start with some definitions.
A \emph{hypergraph} $G$ consists of a vertex set $V$ and an edge set $E$, where each edge $e \in E$ is a subset of $V$. We say that $G$ is a \emph{$k$-graph} if every edge has size $k$.
A \emph{matching} $M$ in $G$ is a set of vertex-disjoint edges in $G$. 
We call $M$ \emph{perfect} if it covers all of $V$. 
We identify a hypergraph $H$ with its edge set,
writing $e \in H$ for $e \in E(H)$, and $|H|$ for $|E(H)|$.
A \emph{$k$-system} is a hypergraph $J$ in which every edge of $J$ has size at most $k$
and $\es \in J$. We refer to the edges of size $r$ in $J$
as \emph{$r$-edges of~$J$}, and write $J_r$ to denote the $r$-graph on $V(J)$ formed by these edges.
A \emph{$k$-complex} $J$ is a $k$-system whose edge set is closed under inclusion,
i.e.\ if $e \in J$ and $e' \sub e$ then $e' \in J$.
For any non-empty $k$-graph $G$, we may generate a $k$-complex $G^\le$
whose edges are any $e \sub V(G)$ such that $e \sub e'$ for some edge $e' \in G$.

Let $V$ be a set of vertices, and let $\Part$ partition $V$ into parts $V_1, \dots, V_r$ of size $n$. Then we say that a hypergraph $G$ with vertex set $V$ is $\Part$-partite if $|e \cap V_i| \leq 1$ for every $i \in [r]$ and $e \in G$. We say that $G$ is $r$-partite if it is $\Part$-partite for some partition $\Part$ of $V$ into $r$ parts.

Let $J$ be a $\Part$-partite $k$-system on $V$. For each $0 \leq j \leq k-1$ we define the
\emph{partite minimum $j$-degree} $\delta^*_j(J)$ as the largest $m$ such that any $j$-edge $e$ has
at least $m$ extensions to a $(j+1)$-edge in any part not intersected by $e$, i.e.\
$$\delta^*_j(J) := \min_{e \in J_j} \min_{i:e \cap V_i = \es} |\{v \in V_i : e \cup \{v\} \in J\}|.$$
The \emph{partite degree sequence} is $\delta^*(J) = (\delta_0^*(J), \dots, \delta_{k-1}^*(J))$.
Note that we suppress the dependence on $\Part$ in our notation: this will be clear from the context. Note also that this is {\em not} the standard notion of degree used in $k$-graphs, in which the degree of a set is the number of edges containing it.
Our minimum degree assumptions will always be of the form $\delta(J) \ge \ab$ pointwise for some vector
$\ab = (a_0,\dots,a_{k-1})$, i.e.\ $\delta_i(J) \ge a_i$ for $0 \le i \le k-1$. It is helpful to interpret
this `dynamically' as follows: when constructing an edge of $J_k$ by greedily choosing one vertex at
a time, there are at least $a_i$ choices for the $(i+1)$st vertex (this is the reason for the requirement that $\emptyset \in J$, which we need for the first choice in the process). 

The following key definition relates our theorems on hypergraphs to graphs. Fix $r \geq k$ and a partition $\Part$ of a vertex set $V$ into $r$ parts $V_1, \dots, V_r$ of size $n$. Let $G$ be a $\Part$-partite graph on $V$. Then the {\em clique $k$-complex} $J(G)$ of $G$ is the $k$-complex whose edges of size $i$ are precisely the copies of $K_i$ in $G$ for $0 \leq i \leq k$. Note that $J(G)$ must be $\Part$-partite. Furthermore, if $\delta^*(G) \geq (k-1)n/k -\alpha n$ and $0 \leq i \leq k-1$, then the vertices of any copy of $K_i$ in $G$ have at least $n - in/k - i \alpha n$ common neighbours in each $V_j$ which they do not intersect. That is, if $G$
satisfies $\delta^*(G) \geq (k-1)n/k -\alpha n$, then the clique $k$-complex $J(G)$ satisfies
\begin{equation} \label{eq:mindeg}
\delta^*(J(G)) \geq \left(n, \left(\frac{k-1}{k} - \alpha\right)n, \left(\frac{k-2}{k} - 2\alpha\right) n, \dots, \left(\frac{1}{k} - (k-1)\alpha\right) n\right).
\end{equation} 
Note also that any perfect matching in the $k$-graph $J(G)_k$ corresponds to a perfect $K_k$-packing in $G$. So if we could prove that any $\Part$-partite $k$-complex $J$ on $V$ which satisfies~$(\ref{eq:mindeg})$ must have a perfect matching in the $k$-graph $J_k$, then we would have already proved Theorem~\ref{partitehajnalszem}! Along these lines, Theorem~2.4 in~\cite{KM} shows that any such $J$ must have a matching in $J_k$ which covers all but a small proportion of the vertices of $J$. (Here we assume $1/n \ll \alpha \ll 1/r, 1/k$). However, two different families of constructions show that this condition does not guarantee a perfect matching in $J_k$; we refer to these as \emph{space barriers} and \emph{divisibility barriers}. We will describe these families in some detail, since the results of~\cite{KM} show that these are essentially the only $k$-complexes $J$ on $V$ which satisfy (\ref{eq:mindeg}) but do not have a perfect matching in $J_k$. Firstly, space barriers are characterised by a bound on the size of the intersection of every edge with some fixed set $S \sub V(J)$. If $S$ is too large, then $J_k$ cannot contain a perfect matching. The following construction gives the precise formulation.

\begin{construct}\label{spacebarpartite} (Space barriers)
Suppose $\Part$ partitions a set $V$ into $r$ parts $V_1,\dots,V_r$ of size $n$. Fix $j \in [k-1]$ and a set $S \sub V$ containing $s = \bfl{(j/k+\alpha) n}$ vertices in each part~$V_j$. Then we denote by $J = J_r(S,j)$ the $k$-complex in which $J_i$ (for $0 \le i \le k$) consists of all $\Part$-partite sets $e \subseteq V$ of size $i$ that contain at most $j$ vertices of $S$. Observe that $\delta^*_i(J) = n$ for $0 \le i \le j-1$ and $\delta^*_i(J)=n-s$ for $j \le i \le k-1$, so (\ref{eq:mindeg}) is satisfied. However, any matching in $J_k$ has size at most $\bfl{\frac{|V \sm S|}{k-j}}$ and so leaves at least $r(\alpha n - k)$ vertices uncovered.
\end{construct}

Having described the general form of space barriers, we now turn our attention to divisibility barriers. These are characterised by every edge satisfying an arithmetic condition with respect to some partition $\Qart$ of $V$. To be more precise, we need the following definition. Fix any partition $\Qart$ of a vertex set $V$ into $d$ parts $V_1, \dots, V_d$. For any $\Qart$-partite set $S \subseteq V$ (that is, $S$ has at most one vertex in each part of $\Qart$), the \emph{index set of $S$ with respect to $\Qart$} is $i_\Qart(S) := \{i \in [d] : |S \cap V_i| = 1\}$. For general sets $S \subseteq V$, we have the similar notion of the \emph{index vector of $S$ with respect to $\Qart$}; this is the vector $\ib_\Qart(S) := (|S \cap V_1|, \dots, |S \cap V_d|)$ in $\Z^d$. So $\ib_\Qart(S)$ records how many vertices of $S$ are in each part of $\Qart$. Observe that if $S$ is $\Qart$-partite then $\ib(S)$ is the characteristic vector of the index set $i(S)$. When $\Qart$ is clear from the context, we write simply $i(S)$ and $\ib(S)$ for $i_\Qart(S)$ and $\ib_\Qart(S)$ respectively, and refer to $i(S)$ simply as the \emph{index} of $S$. We will consider the partition $\Qart$ to define the order of its parts so that $\ib_\Qart(S)$ is well-defined.

\begin{construct} \label{divbar} (Divisibility barriers)
Suppose $\Qart$ partitions a set $V$ into $d$ parts, and $L$ is a lattice in $\Z^d$ (i.e.\ an additive subgroup) with $\ib(V) \notin L$. Fix any $k \ge 2$, and let $G$ be the $k$-graph on $V$ whose edges are all $k$-tuples $e$ with $\ib(e) \in L$.
For any matching $M$ in $G$ with vertex set $S = \bigcup_{e \in M} e$ we have $\ib(S) = \sum_{e \in M} \ib(e) \in L$.
Since we assumed that $\ib(V) \notin L$ it follows that $G$ does not have a perfect matching.
\end{construct}

For the simplest example of a divisibility barrier take $d=2$ and $L = \bgen{(-2,2),(0,1)}$. So $(x,y) \in L$ precisely when $x$ is even. Then the construction described has $|V_1|$ odd, and the edges of $G$ are all $k$-tuples $e \sub V$ such that $|e \cap V_1|$ is even. If $|V| = n$ and $|V_1| \sim n/2$, then any set of $k-1$ vertices of $G$ is contained in around $n/2$ edges of $G$, but $G$ contains no perfect matching.

We now consider the multipartite setting. Let $\Part$ partition a vertex set $V$ into parts $V_1, \dots, V_r$ of size $n$, and let
$\Qart$ be a partition of $V$ into $d$ parts $U_1, \dots, U_d$ which refines $\Part$. Then we say that a lattice $L \sub \Z^d$ is
{\em complete with respect to $\Part$} if $L$ contains every difference of basis vectors $\ub_i - \ub_j$ for which $U_i$ and $U_j$
are contained in the same part $V_\ell$ of $\Part$, otherwise we say that $L$ is {\em incomplete with respect to $\Part$}. The
idea behind this definition is that if $L$ is incomplete with respect to $\Part$, then it is possible that $\ib_\Qart(V) \notin
L$, in which case we would have a divisibility barrier to a perfect matching, whilst if $L$ is complete with respect to $\Part$
then this is not possible. 

There is a natural notion of minimality for an incomplete lattice $L$ with respect to $\Part$. We say that $\Qart$ is \emph{transferral-free} if $L$ does not contain any difference of basis vectors $\unit_i-\unit_j$ for which $U_i, U_j$ are contained in the same part $V_\ell$ of $\Part$. For suppose $L$ does contain some such difference $\unit_i-\unit_j$ and form a partition $\Qart'$ from $\Qart$ by merging parts $U_i$ and $U_j$ of $\Qart$. Let $L' \subseteq \Z^{d-1}$ be the lattice formed by this merging (that is, by replacing the $i$th and $j$th coordinates with a single coordinate equal to their sum). Then $L'$ is also incomplete with respect to $\Part$, so we have a smaller divisibility barrier.

Let $J$ be an $r$-partite $k$-complex whose vertex classes $V_1, \dots, V_r$ each have size $n$. The next theorem, Theorem~2.9 from~\cite{KM}, states that if $J$ satisfies (\ref{eq:mindeg}) and $J_k$ is not `close' to either a space barrier or a divisibility barrier, then $J_k$ contains a perfect matching. Moreover, we can find a perfect matching in $J_k$ which has roughly the same number of edges of each index. More precisely, for a perfect matching $M$ in $J_k$ and a set $A \in \binom{[r]}{k}$ let $N_A(M)$ be the number of edges $e \in M$ with index $i(e) = A$. We say that $M$ is \emph{balanced} if $N_A(M)$ is constant over all $A \in \binom{[r]}{k}$, that is, if there are equally many edges of each index. Similarly, we say that $M$ is \emph{$\gamma$-balanced} if $N_A(M) = (1 \pm \gamma) N_B(M)$ for any $A, B \in \binom{[r]}{k}$. Finally, we formalise the notion of closeness to a space or divisibility barrier as follows. Let $G$ and $H$ be $k$-graphs on a common vertex set $V$ of size~$n$. We say $G$ is {\em $\beta$-contained} in $H$ if all but at most $\beta n^k$ edges of $G$ are edges of $H$. Also, given a partition $\Part$ of $V$ into $d$ parts, we define the {\em $\mu$-robust edge lattice} $L^\mu_\Part(G) \sub \Z^d$ to be the lattice generated by all vectors $\vb \in \Z^d$ such that there are at least $\mu n^k$ edges $e \in G$ with $\ib_\Part(e) = \vb$. 

\begin{theo} \label{partitematching}
Suppose that $1/n \ll \gamma \ll \alpha \ll \mu, \beta \ll 1/r$, $r \geq k$ and $k \mid rn$. Let~$\Part'$ partition a set $V$ into parts $V_1, \dots, V_r$ each of size~$n$. Suppose that $J$ is a $\Part'$-partite $k$-complex with $$\delta^*(J) \geq \left(n, \left(\frac{k-1}{k} - \alpha\right)n, \left(\frac{k-2}{k} - \alpha\right) n, \dots, \left(\frac{1}{k} - \alpha\right) n\right).$$
Then $J$ has at least one of the following properties.
\begin{description}
\item[1 (Matching)] $J_k$ contains a $\gamma$-balanced perfect matching.
\item[2 (Space barrier)] $J_k$ is $\beta$-contained in $J_r(S,j)_k$ for some $j \in [k-1]$ and $S \sub V$ with $\bfl{jn/k}$ vertices in each $V_i$, $i \in [r]$.
\item[3 (Divisibility barrier)] There is some partition $\Part$ of $V(J)$ into $d \le kr$
parts of size at least  $\delta^*_{k-1}(J) - \mu n$ such that $\Part$ refines $\Part'$
and $L^\mu_\Part(J_k)$ is incomplete with respect to $\Part'$.
\end{description}
\end{theo}

Note that the fact that the perfect matching in $J_k$ is $\gamma$-balanced in the first property is not stated in the statement of the theorem in~\cite{KM}. However, examining the short derivation of this theorem from Theorem~7.11 in~\cite{KM} shows this to be the case.

\section{Outline of the proof} \label{sec:outline}

\begin{figure}[ht] 
\centering
\psfrag{1}{$V_1$}
\psfrag{2}{$V_2$}
\psfrag{3}{$V_3$}
\psfrag{4}{$V_4$}
\psfrag{5}{$\Big\}X^2$}
\psfrag{6}{$\Bigg\}X^1$}
\psfrag{7}{$X^1_1$}
\psfrag{8}{$X^1_2$}
\psfrag{9}{$X^1_3$}
\psfrag{10}{$X^1_4$}
\psfrag{11}{$X^2_1$}
\psfrag{12}{$X^2_2$}
\psfrag{13}{$X^2_3$}
\psfrag{14}{$X^2_4$}
\psfrag{a}{\hspace{.5cm}$p_1 = 2$}
\psfrag{b}{\hspace{.5cm}$p_2 = 1$}
\includegraphics[width=10cm]{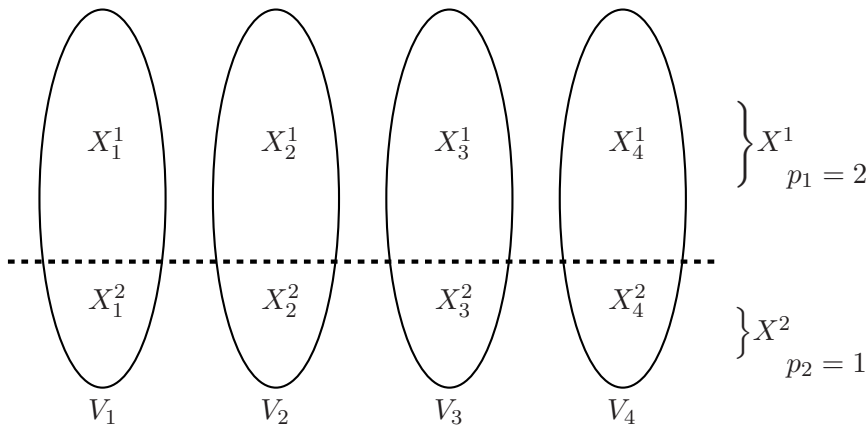} 
\caption{A row-decomposition of a $4$-partite graph $G$ into $2$ rows.}
\label{fig:decomp}
\end{figure}

In this section we outline the proof of Theorem~\ref{partitehajnalszem}. For ease of explanation we restrict to the case when $G$ is an $r$-partite graph whose vertex classes each have size $kn$ and $\delta^*(G) \geq (k-1)n$. Our strategy consists of the following three steps:
\begin{enumerate}[(i)]
\item Impose a row structure on $G$.
\item Find balanced perfect clique-packings in each row.
\item Glue together the row clique-packings to form a $K_k$-packing of $G$.
\end{enumerate} 

For step (i) we partition $V(G)$ into \emph{blocks} $X^i_j$, so that each vertex class $V_j$ is partitioned into $s$ blocks $X^1_j, \dots, X^s_j$. This partition is best thought of as an $s \times r$ grid, with rows $X^i := \bigcup_{j \in [r]} X^i_j$ and columns the vertex classes $V_j = \bigcup_{i \in [s]} X^i_j$. We insist that all the blocks in a given row $X^i$ have equal size $p_i n$, where $\sum_{i \in [s]} p_i = k$. We call a partition of $V(G)$ which satisfies these conditions an \emph{$s$-row-decomposition} of $G$. We also require that $G$ has density close to $1$ between any two blocks which do not lie in the same row or column (we refer to the smallest such density as the \emph{minimum diagonal density}). Figure~\ref{fig:decomp} illustrates this structure.
We begin with the trivial 1-row-decomposition of $G$ with a single row (so the blocks are the vertex classes $V_j$). If it is possible to split this row into two rows to obtain a row-decomposition with minimum diagonal density at least $1-d$ (where $d$ will be small), then we say that $G$ is \emph{$d$-splittable}. If so, we partition $G$ in this manner, and then examine in turn whether either of the two rows obtained is splittable (for some larger value of $d$). By repeating this process, we obtain a row-decomposition of $G$ with high minimum diagonal density in which no row is splittable; this argument is formalised in Lemma~\ref{iterate}.

For step (ii) we require a balanced perfect $K_{p_i}$-packing in each row $X^i$. We first use the results of Section~\ref{sec:theory} to obtain a near-balanced perfect $K_{p_i}$-packing in $G[X^i]$. Fix $i$ and take $J$ to be the clique $p_i$-complex of $G[X^i]$. So we regard the row $X^i$ as an $r$-partite vertex set whose parts are the blocks $X^i_1, \dots, X^i_r$, and the edges of $J_j$ are the $j$-cliques in $G[X^i]$ for $j \leq p_i$. The assumption $\delta^*(G) \geq (k-1)n$ implies that 
\begin{equation} \label{eq:mindega}
\nonumber \delta^*(J) \geq \left(p_i n, (p_i-1)n, (p_i-2)n, \dots, n\right).
\end{equation} 
Then Theorem~\ref{partitematching} (with $p_i$ playing the role of $k$) implies that $J_{p_i}$ contains a near-balanced perfect matching, unless $J_{p_i}$ is close to a space or divisibility barrier. In Section~\ref{sec:blocks} we consider a space barrier, showing in Lemma~\ref{splittable} that since $G[X^i]$ is not $d$-splittable, $J_{p_i}$ cannot be close to a space barrier. We then consider a divisibility barrier in Section~\ref{sec:div}. For $p_i \geq 3$, Lemma~\ref{split} shows that since $G[X^i]$ is not $d$-splittable, $J_{p_i}$ also cannot be close to a divisibility barrier. However, the analogous statement for $p_i = 2$ is false, for the following reason. 

We say that $G[X^i]$ is `pair-complete' if it has a structure close to that which appears in rows $V^1$ and $V^2$ of Construction \ref{fischereg}. That is, there is a partition of $X^i$ into `halves' $S$ and $X^i \sm S$, such that each vertex class $V_j$ is partitioned into two equal parts, and both $G[S]$ and $G[X^i \sm S]$ are almost complete $r$-partite graphs. Such a row is not $d$-splittable if $r$ is odd, but $J_{2}$ is close to a divisibility barrier.  However, Lemma~\ref{2splitorpc} shows that this is essentially the only such example, that is, that if $p_i = 2$ and $G[X^i]$ is neither $d$-splittable nor pair-complete then $J_2$ is not close to a divisibility barrier. So unless $p_i = 2$ and $G[X^i]$ is pair-complete, Theorem~\ref{partitematching} implies that $J_{p_i}$ contains a near-balanced perfect $K_{p_i}$-packing.
In Section~\ref{sec:packing} we then show that we can actually obtain a balanced perfect matching in $J_{p_i}$. Indeed, in Lemma~\ref{theoremmatching} we first delete some `configurations' from $G[X^i]$; these are subgraphs of $G[X^i]$ that can be expressed as two disjoint copies of $K_{p_i}$ in $G[X^i]$ in two different ways (with different index sets). After these deletions we proceed as just described to find a near-balanced perfect $K_{p_i}$-packing in $G[X^i]$. Then by carefully choosing which pair of disjoint edges to add to the matching from each `configuration', we obtain a balanced perfect $K_{p_i}$-packing in $G[X^i]$, as required. This leaves only the case where $p_i = 2$ and $G[X^i]$ is pair-complete; in this case Lemma~\ref{paircompletematching} gives a balanced perfect $K_2$-packing in $G[X^i]$, provided that each half has even size.

For step (iii), we construct auxiliary hypergraphs, perfect matchings in which describe how to glue together the perfect $K_{p_i}$-packings in the rows into a perfect $K_k$-packing of $G$. Recall that the row-decomposition of $G$ was chosen to have large minimum diagonal density, so almost every vertex of any block $X^i_j$ has few non-neighbours in any block $X^{i'}_{j'}$ in a different row and column. Assume for now that this row-decomposition of $G$ has the stronger condition of large \emph{minimum diagonal degree}, i.e.\ that we can delete `almost' from the previous statement. For each row $i$, we partition its perfect $K_{p_i}$-packing into sets $E_{\sigma,i}$, one for each injective function $\sigma : [k] \to [r]$. For each $\sigma$ we then form an auxiliary $s$-partite $s$-graph $H_\sigma$, where for each $i \in [s]$ the $i$-th vertex class of $H_\sigma$ is the set $E_{\sigma,i}$ (so a copy of $K_{p_i}$ in $G[X^i]$ is a vertex of $H_\sigma$). Edges in $H_\sigma$ are those $s$-tuples of vertices for which the corresponding copies of $K_{p_i}$ together form a copy of $K_k$ in $G$. We defer the details of the partition to the final section of this paper; the crucial point is that the large minimum diagonal degree of $G$ ensures that each $H_\sigma$ has sufficiently large vertex degree to guarantee a perfect matching. Taking the copies of $K_k$ in $G$ corresponding to the union of these matchings gives a perfect $K_k$-packing in $G$, completing the proof.

The above sketch glosses over the use of the precise minimum degree condition in Theorem~\ref{partitehajnalszem}. Indeed, to replace our minimum diagonal density condition with a minimum diagonal degree condition, we must remove all `bad' vertices, namely those which have many non-neighbours in some block in a different row and column. To achieve this, before step (ii) we delete some vertex-disjoint copies of $K_k$ from $G$ which cover all bad vertices. We must ensure that the number of vertices deleted from row $X^i$ is a constant multiple of $p_i$ for each $i$, so that we will be able to join together the $K_{p_i}$-packings of the undeleted vertices of each $X^i$ to form a $K_k$-packing of $G$. We also need to ensure that each half has even size in pair-complete rows. This is accomplished in Section~\ref{sec:proof2}, which is the most lengthy and technical part of the paper. After this, it is fairly quick to complete the proof as outlined above in Section~\ref{proof}.

\section{Row decompositions and space barriers} \label{sec:blocks} 
In this section we formalise our description of row-decompositions and the iterative process of splitting rows described in Section~\ref{sec:outline}. We then show that the clique $p$-complex of any row obtained at the end of this process is not close to a space barrier. Note that many definitions and results of this section (and later sections) require that the size of each vertex class should be a multiple of $k$, which is not assumed in the statement of Theorem~\ref{partitehajnalszem}. However, our first step in the proof of Theorem~\ref{partitehajnalszem} will be to remove vertices so that this condition is satisfied, allowing these definitions and results to be used.

\subsection{Row-decompositions.}

Fix $r \geq 2$, and let $G$ be an $r$-partite graph on vertex classes $V_1, \dots, V_r$ each of size $kn$. Suppose $s \in [k]$ and $p_i$, $i \in [s]$ are positive integers with $\sum_{i \in [s]} p_i = k$. Write $p = (p_i:i \in [s])$. An \emph{$s$-row-decomposition} $X = (X^i_j)_{i \in [s], j \in [r]}$ of $G$ (of \emph{type} $p$), consists of subsets $X^i_j \sub V_j$ with $|X^i_j| = p_i n$ for each $i \in [s]$ and $j \in [r]$ such that each $V_j$ is partitioned by the sets $X^i_j$ with $i \in [s]$. We refer to the sets $X^i_j$ as the \emph{blocks}, and the sets $X^i := X^i_1 \cup \dots \cup X^i_r$ for $i \in [s]$ as the \emph{rows}. We call the parts $X_j := V_j = X^1_j \cup \dots \cup X^s_j$ for $j \in [r]$ the \emph{columns}, so $G$ has $s$ rows and $r$ columns. Given subsets $A,B$ of different vertex classes of $G$, let $G[A,B]$ denote the bipartite subgraph of $G$ induced by $A \cup B$. We write $e_G(A,B) = |G[A,B]|$ and define the \emph{density} of $G$ between $A$ and $B$ as $d_G(A,B) = \frac{e_G(A,B)}{|A||B|}$. We usually write $e(A,B)=e_G(A,B)$ and $d(A,B)=d_G(A,B)$, as $G$ is clear from the context. The \emph{minimum diagonal density} of $G$ is defined to be the minimum of $d(X^i_j, X^{i'}_{j'})$ over all $i \neq i'$ and $j \neq j'$. If $G$ has only one row then for convenience we define the minimum diagonal density of $G$ to be $1$. Note that all this terminology depends on the choice of row-decomposition of $G$, but this will be clear from the context.

For any $i \in [s]$ with $p_i \geq 2$ we may obtain an $(s+1)$-row-decomposition of $G$ by partitioning the row $X^i$ of $G$. Indeed, choose positive integers $y$ and $z$ with $y + z = p_i$. For each $j \in [r]$ partition $X^i_j$ into sets $Y^i_j$ and $Z^i_j$ with $|Y^i_j| = yn$ and $|Z^i_j| = zn$. Take $p'_i := y$, $p'_{s+1} := z$ and $p'_\ell := p_\ell$ for each $\ell \in [s] \sm \{i\}$, and for each $j \in [r]$ let $\hat{X}^i_j := Y^i_j$, $\hat{X}^{s+1}_j := Z^i_j$ and $\hat{X}^\ell_j := X^{\ell}_j$ for each $\ell \in [s] \sm \{i\}$. Then the blocks $\hat{X}^\ell_j$ form an $(s+1)$-row-decomposition of $G$ of type $p' = (p'_\ell: \ell \in [s+1])$.

Bearing in mind the proof strategy sketched above, we are happy to split rows provided that we keep the minimum diagonal density close to $1$. Thus we make the following definition. Let $G$ be an $r$-partite graph on vertex classes $V_1, \dots, V_r$ each of size $pn$. We say $G$ is \emph{$d$-splittable} if for some $p' \in [p-1]$ we may choose sets $S_i \sub V_i$, $i \in [r]$ with $|S_i| = p'n$, such that for any $i, i' \in [r]$ with $i \neq i'$ we have $d(S_i, V_{i'} \sm S_{i'}) \geq 1-d$. It is helpful to think of $G$ as being a row-decomposition with just one row; then $G$ is $d$-splittable if it is possible to partition this row into two rows as described above so that the minimum diagonal density is at least $1-d$. Note that this definition depends on $p$, however this will always be clear from the context. Note also that $G$ can never be $d$-splittable if $p=1$. The next proposition shows that we can iteratively split $G$ until we reach a row-decomposition which has high minimum diagonal density and does not have any splittable row.

\begin{prop} \label{iterate}
Suppose that $1/n \ll d_0 \ll \dots \ll d_k \ll 1/r$ and $r \geq 2$. Let $G$ be an $r$-partite graph on vertex classes $V_1, \dots, V_r$ each of size $kn$. Then for some $s \in [k]$ there exists an $s$-row-decomposition $X$ of $G$ with minimum diagonal density at least $1-k^2d_{s-1}$ such that each row $G[X^i]$ of $G$ is not $d_s$-splittable.
\end{prop}

\proof 
Initially we take the trivial $1$-row-decomposition of $G$ with one row whose blocks are the vertex classes $V_1, \dots, V_r$ of $G$. We now repeat the following step. Given an $s$-row-decomposition of $G$, if every row $G[X^i]$ is not $d_s$-splittable, then terminate. Alternatively, if $G[X^i]$ is $d_s$-splittable for some $i \in [s]$, according to some sets $S_j \sub X^i_j$, $j \in [r]$, then partition each block $X^i_j$ into two blocks $S_j$ and $X^i_j \sm S_j$ to obtain an $(s+1)$-row-decomposition of $G$.

Since $G[X^i]$ can only be $d_s$-splittable if $p_i \geq 2$, this process must terminate with $s \le k$. Then we have an $s$-row-decomposition of $G$ all of whose rows are not $d_s$-splittable, so it remains only to show that $G$ has minimum diagonal density at least $1-k^2d_{s-1}$. If $s=1$ then this is true by definition, so we may assume $s \geq 2$. Consider any rows $i \neq i'$ and columns $j \neq j'$. Since $X^i_j$ and $X^{i'}_{j'}$ do not lie in the same row of $G$, at some point in the process we must have partitioned blocks $Y^\ell_j$ and $Y^\ell_{j'}$ into $S_j$, $Y^\ell_j \sm S_j$ and $S_{j'}$, $Y^\ell_{j'} \sm S_{j'}$ with $X^i_j \sub S_j$ and $X^{i'}_{j'} \sub Y^\ell_{j'} \sm S_{j'}$ respectively. Since $G[Y^\ell]$ was $d_t$-splittable for some $t \le s-1$, we have $d(S_j, Y^\ell_{j'} \sm S_{j'}) \geq 1-d_{s-1}$. Then, since $|X^i_j| \geq |S_j|/k$ and $|X^{i'}_{j'}| \geq |Y^\ell_{j'}|/k$, we have $d(X^i_j, X^{i'}_{j'}) \geq 1-k^2d_{s-1}$, as required.
\endproof

\subsection{Avoiding space barriers} \label{sec:space}

Let $G$ be an $r$-partite graph whose vertex classes have size $pn$ with $\delta^*(G) \geq (p-1)n - \alpha n$, and let $J=J(G)$ be the clique $p$-complex of $G$. In this section we show that if $G$ is not $d$-splittable then there is no space barrier to a perfect matching in $J_p$. We shall use this result in combination with the results of the next section to find a perfect clique packing in each row. We also prove that if $p < r$ then $G$ contains many copies of $K_{p+1}$; this result will play an important role in the proof of Lemma~\ref{diagonalmindeg}.

\begin{lemma} \label{splittable}
Suppose that $1/n \ll \alpha \ll \beta \ll d \ll 1/r$ and $2 \le p \leq r$. Let $G$ be an $r$-partite graph on vertex classes $V_1, \dots, V_r$ each of size $pn$ with $\delta^*(G) \geq (p-1)n - \alpha n$. Suppose that $G$ is not $d$-splittable. Then
\begin{enumerate}[(i)]
\item for any $p' \in [p-1]$ and sets $S_i \sub V_i$, $i \in [r]$ of size $p'n$ there are at least $\beta n^p$ copies of $K_{p}$ in $G$ with more than $p'$ vertices in $G[S]$, where $S = \bigcup_{i \in [r]} S_i$, and
\item if $p<r$ then there are at least $\beta n^{p+1}$ copies of $K_{p+1}$ in $G$.
\end{enumerate}
\end{lemma}

\proof
For (i), since $G$ is not $d$-splittable, we may suppose that $d(S_1, V_{p'+1} \sm S_{p'+1}) < 1-d$. Let $A$ be the set of vertices in $S_1$ with fewer than $(1-d/2)(p-p')n$ neighbours in $V_{p'+1} \sm S_{p'+1}$. Write $|A|=ap'n$. Then $(1-d/2)(1-a) < d(S_1, V_{p'+1} \sm S_{p'+1}) < 1-d$, so $a>d/2$. We now greedily form a copy of $K_{p'+1}$ in $G[S]$ by choosing a vertex $v_i \in S_i$ for each $i \in [p' +1]$ in turn (in increasing order). We do this so that $v_1 \in A$ and $v_i \in N(v_j)$ for any $j < i$. There are $|A| \geq dn/2$ suitable choices for~$v_1$. For each $i \in \{2, \dots, p'\}$ we have chosen $i-1$ vertices prior to choosing $v_i$, so there are at least $|S_i| -  (i-1)(pn-\delta^*(G)) \geq p'n - (p'-1)(n+\alpha n) \geq (1-p\alpha)n$ suitable choices for $v_i$. Finally, since $v_1 \in A$ has at least $(p-p')nd/2 \geq nd/2$ non-neighbours in $V_{p'+1} \sm S_{p'+1}$, and at most $|V_{p'+1}| - \delta^*(G) \leq n + \alpha n$ non-neighbours in $V_{p'+1}$ in total, it has fewer than $(1-d/2+\alpha)n \le (1-d/3)n$ non-neighbours in $S_{p'+1}$. This means that there are at least $|S_{p'+1}| - (1-d/3)n - (p'-1)(pn-\delta^*(G)) \geq dn/4$ suitable choices for $v_{p'+1}$. Together we conclude that there are at least $(dn/4)(dn/2)((1-p\alpha) n)^{p'-1} \geq 2\beta n^{p'+1}$ copies of $K_{p'+1}$ in $G[S]$. Each such copy can be extended to a copy of $K_p$ in $G$ with more than $p'$ vertices in $S$ by choosing $v_i \in V_i$ for each $p'+2 \leq i \leq p$ in turn, so that each $v_i$ chosen is a neighbour of every $v_j$ with $j \leq i$. For each $p'+2 \leq i \leq p$ there are at least $pn - (i-1)(pn - \delta^*(G)) \geq pn - (p-1)(n+\alpha n) \geq (1-p\alpha)n$ suitable choices for $v_i$, so we obtain at least $2\beta n^{p'+1} ((1-p\alpha)n)^{p - p' - 1} \geq \beta n^p$ such copies of $K_p$.

For (ii), introduce new constants with $\beta \ll \gamma \ll \beta' \ll d_1 \ll d_2 \ll d$, and suppose for a contradiction that there are fewer than $\beta n^{p+1}$ copies of $K_{p+1}$ in $G$. Say that a vertex $x \in V(G)$ is \emph{bad} if it lies in at least $\sqrt{\beta} n^p$ copies of $K_{p+1}$ in $G$, and let $X$ be the set of all bad vertices. Then $\sqrt{\beta} n^p |X| \leq r\beta n^{p+1}$, so $|X| \leq r\sqrt{\beta} n$. We now show that for any $i \in [r]$, any vertex $v \in V_i \sm X$ has at most $(p-1)n + \gamma n$ neighbours in $V_{j}$ for any $j \neq i$. Without loss of generality we consider the case $i=1$, i.e.\ $v \in V_1 \sm X$. Suppose for a contradiction that $|N(v) \cap V_j| > (p-1)n + \gamma n$ for some $j$, say $j = p+1$. Then we may greedily form a copy of $K_{p+1}$ in $G$ containing $v$ by choosing $x_2, \dots, x_{p+1}$ with $x_i \in V_i$ for each $i$ so that each $x_i$ is a neighbour of $v, x_2, \dots, x_{i-1}$. We have at least $pn - (i-1)(pn - \delta^*(G)) \geq (p-i+1)n - (i-1)\alpha n \geq n/2$ choices for each $x_i$ with $i \in \{2, \dots, p\}$, and at least $|N(v) \cap V_{p+1}| - (p-1) (pn - \delta^*(G)) \geq (p-1)n + \gamma n - (p-1)(n + \alpha n) \geq \gamma n/2$ choices for $x_{p+1}$. Thus there are at least $\gamma n^p/2^p \geq \sqrt{\beta} n^p$ copies of $K_{p+1}$ in $G$ containing $v$, a contradiction to $v \notin X$.

Now we fix some $v \in V_1 \sm X$ and use the neighbourhood of $v$ to impose structure on the rest of the graph. We choose a set $S_j \subseteq V_j$ of size $(p-1)n$ which contains or is contained in $N(v) \cap V_j$ for each $j \geq 2$. If $d(S_i, V_j \sm S_j) < 1-d_1$ for some $i,j \geq 2$ with $i \neq j$, then as in part (i) we can find at least $2\beta' n^p$ copies of $K_p$ in $\bigcup_{i \geq 2} S_i$. At least $\beta' n^p$ of these are contained in $N(v)$, and so form copies of $K_{p+1}$ with $v$, another contradiction. So we may suppose that $d(S_i, V_j \sm S_j) \geq 1-d_1$ for any $i, j \geq 2$ with $i \neq j$. We now partition $V_1$ into sets $A, B, C$ as follows. Let $A$ consist of all vertices $u \in V_1$ with $|N(u) \cap (V_j \sm S_j)| \leq d_2n$ for every $2 \leq j \leq r$. Let $B$ consist of all vertices $u \in V_1$ with $|N(u) \cap (V_j \sm S_j)| \geq (1-d_2)n$ for every $2 \leq j \leq r$. Let $C = V_1 \sm (A \cup B)$ consist of all remaining vertices of $V_1$. Next we bound the sizes of each of these sets. By definition of $A$ we have $e(A, V_2 \sm S_2) \leq d_2 n |A|$, so some vertex in $V_2 \sm S_2$ has at most $d_2 |A|$ neighbours in $A$. So $pn - |A| + d_2|A| \geq \delta^*(G) \geq (p-1)n - \alpha n$, from which we obtain $|A| \leq (1+2d_2)n$. Next note that by definition of~$B$ we have $e(B, V_2 \sm S_2) \geq (1-d_2) n |B|$. So at least $n/2$ vertices of $V_2 \sm S_2$ have at least $(1-2d_2)|B|$ neighbours in $B$. At least one of these vertices is not bad, so by our earlier observation has at most $(p-1)n + \gamma n$ neighbours in $V_1$. Then $(1-2d_2)|B| \leq (p-1)n +\gamma n$, so $|B| \leq (1+3d_2)(p-1) n$. 

To bound $|C|$ we show that $C \sub X$. Consider any vertex $w \in C$. Without loss of generality $|N(w) \cap (V_{p+1} \sm S_{p+1})| > d_2n$ and $|N(w) \cap (V_p \sm S_p)| < (1-d_2)n$. Choose greedily a vertex $x_j \in S_j$ for each $2 \leq j \leq p$ so that $x_j$ is a neighbour of $w, x_2, \dots, x_{j-1}$ and satisfies $|N(x_j) \cap (V_{p+1} \sm S_{p+1})| \geq (1-\sqrt{d_1})n$. To see that this is possible for each $2 \leq j \leq p$, note that since $d(S_j, V_{p+1} \sm S_{p+1}) \geq 1-d_1$, at most $(p-1)\sqrt{d_1} n$ vertices $x_j \in S_j$ fail the latter condition. Note also that at least $|S_j| - |S_j \sm N(w)| - \sum_{2 \leq i < j} |S_j \sm N(x_i)| \geq (p-1)n - |S_j \sm N(w)| - (j-2) (n + \alpha n)$ vertices $x_j \in S_j$ satisfy the neighbourhood condition. For $j < p$ this gives at least $n/2$ suitable choices for $x_j$. On the other hand, for $j=p$ we have $|N(w) \cap (V_p \sm S_p)| < (1-d_2)n$, which implies that $|S_j \sm N(w)| \leq n - d_2 n + \alpha n$, so we have at least $d_2 n/2$ suitable choices for $x_j$. So we may form at least $d_2 (n/2)^{p-1}$ copies of $K_p$ containing $w$ in this manner. By construction, each $x_j$ in any such copy has at most $\sqrt{d_1}n$ non-neighbours in $V_{p+1} \sm S_{p+1}$. Since $w$ has at least $d_2 n$ neighbours in $V_{p+1} \sm S_{p+1}$, we find a total of at least $(d_2 n - p\sqrt{d_1}n) d_2 (n/2)^{p-1} \geq \sqrt{\beta} n^p$ copies of $K_{p+1}$ in $G$ containing $w$, so $w \in X$. We deduce that $C \subseteq X$, so $|C| \leq |X| \leq r\sqrt{\beta} n$. 

We therefore have $|B| \geq pn - |A| - |C| \geq (p-1)n - 3d_2 n$, so $|B| = (1 \pm 3d_2)(p-1) n$. Let $S_1$ be a set of size $(p-1) n$ which either contains or is contained in $B$. Then for any $2 \leq j \leq r$ we have $e(S_1, V_j \sm S_j) \geq \min\{|B|, |S_1|\} (1-d_2)n \geq (1-4d_2)(p-1)n^2$, so $d(S_1, V_j \sm S_j) \geq 1 - 4d_2$. 
Also, at most $4 d_2 (p-1) n $ vertices of $V_1 \sm S_1$ lie in $B \cup C$, so for any $2 \leq j \leq r$ we have $e(V_1 \sm S_1, V_j \sm S_j) \leq d_2n^2 + 4d_2(p-1) n^2 \leq 4pd_2n^2$. But $e(V_1 \sm S_1, V_j) \geq \delta^*(G)n \geq (p-1-\alpha)n^2$, so we obtain $e(V_1 \sm S_1, S_j) \geq (p-1-\alpha)n^2 - 4pd_2n^2$, and so $d(V_1 \sm S_1, S_j) \geq 1-9d_2$. Recall also that $d(S_i, V_j \sm S_j) \geq 1-d_1$ for any $i, j \geq 2$ with $i \neq j$. Since $d_1, d_2 \ll d$ we conclude that $G$ is $d$-splittable with respect to the sets $S_j$ for $j \in [r]$. This is a contradiction, so (ii) holds.
\endproof

\section{Avoiding divisibility barriers} \label{sec:div} 

Let $G$ be an $r$-partite graph with vertex classes of size $pn$ such that $\delta^*(G) \geq (p-1)n - \alpha n$, and let $J = J(G)$ be the clique $p$-complex of $G$. In the previous section we saw that if $G$ is not $d$-splittable (for small $d$), then there is no space barrier to a perfect matching in $J_p$. In this section we instead consider divisibility barriers. Indeed, we shall see in the second subsection that if $p \geq 3$ and $G$ is not $d$-splittable, then $J_p$ cannot be close to a divisibility barrier. However, for $p=2$ there is another possibility, namely that $G$ has the structure of $V^1 \cup V^2$ in Construction \ref{fischereg}. There we described both $V^1$ and $V^2$ as rows, but with the terminology of the previous section they should be considered as a single row. We consider this case in the first subsection. Note that here we have $J_p = J_2 = G$.

\subsection{Pair-complete rows} \label{subsec:paircomplete}
Let $G$ be an $r$-partite graph with vertex classes $V_1,\dots,V_r$ each of size $2n$. We say that $G$ is \emph{$d$-pair-complete} (with respect to $S = \bigcup_{j \in [r]} S_j$) if there exist sets $S_j \sub V_j$, $j \in [r]$ each of size $n$ such that $d(S_i, S_j) \ge 1-d$, $d(V_i \sm S_i, V_j \sm S_j) \geq 1-d$ and $d(S_i, V_j \sm S_j) \le d$ for any $i,j \in [r]$ with $i \neq j$. That is, $G$ consists of two \emph{halves} $S$ and $V \sm S$, where each half is an almost-complete $r$-partite graph, and there are few edges between halves. We will show that if $G$ is close to a divisibility barrier, then $G$ is either $d$-splittable or $d$-pair-complete. For this we need the following proposition.

\begin{prop} \label{auxgraph}
Let $r \geq 2$ and $H$ be an $r$-partite graph whose parts $V_1,\dots,V_r$ each have size $2$. Suppose that $\delta^*(H) \geq 1$, and for any $A \sub V$ such that $|A \cap V_j| = 1$ for every $j \in [r]$ there is
\begin{itemize}
\item[(i)] an edge $ab$ with $a,b \in A$ or $a,b \notin A$, and
\item[(ii)] an edge $ab$ with $a \in A$ and $b \notin A$.
\end{itemize}
Then for some $V_j = \{x,y\}$ there is a path of even length between $x$ and $y$.
\end{prop}

\proof 
Suppose for a contradiction that there is no such path. Then for any $i, j \in [r]$ with $i \neq j$, any vertex $v \in V_i$ must have precisely one neighbour in $V_j$. Indeed, $v$ must have at least one neighbour in $V_j$ since $\delta^*(H) \geq 1$, but cannot have two since then we obtain a path of length two between these two neighbours in $V_j$. So for any $i \neq j$ the graph $H[V_i, V_j]$ consists of two disjoint edges. 
Write $V_1 = \{x_1,y_1\}$, and for each $2 \le i \le r$ let $x_i \in V_i$ be adjacent to $x_{i-1}$ and $y_i \in V_i$ be adjacent to $y_{i-1}$. 
There are then two possibilities for $H[V_r, V_1]$.

The first case is that $x_r$ is adjacent to $x_1$ and $y_r$ to $y_1$. By property (ii) of $A = \{x_1,\dots,x_r\}$, there must be some edge $x_iy_j$. Fix such an $i$ and $j$, and consider the paths $x_1\dots x_i y_j \dots y_1$ and $x_1\dots x_iy_j \dots y_ry_1$ between $x_1$ and $y_1$. They have lengths $i+j-1$ and $i+r-j+1$, which must both be odd, so $i+j$ and $r$ are both even. This argument shows that $i'+j'$ must be even for any edge $x_{i'}y_{j'}$. So by property (i) of $A = \{x_{i'}: i' \text{ even}\} \cup \{y_{i'}: i' \text{ odd}\}$ there must be an edge $x_{i'}x_{j'}$ or $y_{i'}y_{j'}$ such that $i'+j'$ is even. Without loss of generality we may assume the former, and that $j' \neq i$. If $i' = i$ then $x_{j'}x_iy_jy_{j+1}\dots y_{j'}$ is a path whose length is congruent to $j' - j + 2 \equiv j' + i' - j - i + 2 \equiv 0$ modulo $2$, giving a contradiction. On the other hand, if $i' \neq i$ then let $P$ be a path from $x_{j'}$ to $x_{i}$ which does not contain $x_{i'}$ and whose length has the same parity as $i-j'$  (this must exist since $r$ is even and $x_i, x_{i'}$ and $x_{j'}$ all lie on the cycle $x_1x_2\dots x_rx_1$). Then $x_{i'}x_{j'}Px_iy_{j}y_{j+1}\dots y_{i'}$ is a path whose length is congruent to $i - j' + i' - j + 2 \equiv 0$ modulo $2$, again giving a contradiction.

The second case is that $x_r$ is adjacent to $y_1$ and $y_r$ to $x_1$. Then $x_1 \dots x_ry_1$ must have odd length, so $r$ is odd. By property (i) of $A = \{x_i: i \text{ even}\} \cup \{y_i: i \text{ odd}\}$ we must have either an edge $x_iy_j$ with $i+j$ odd, or an edge $x_ix_j$ with $i+j$ even, or an edge $y_iy_j$ with $i+j$ even. In the first case $x_1\dots x_i y_j \dots y_1$ is a path of even length $i+j-1$. On the other hand, for the second case we may assume that $i < j$, whereupon $x_1 \dots x_i x_j \dots x_r y_1$ is a path of even length $r - j + i + 1$, and the third case is similar. Thus we have a contradiction in all cases, so the required path exists.
\endproof

Now we can deduce the required structure for divisibility barriers in $G$ when $p=2$.

\begin{lemma} \label{2splitorpc}
Suppose that $1/n \ll  \mu, \alpha \ll d \ll 1/r$ and $r \geq 2$. Let $\Part'$ partition a set $V$ into parts $V_1, \dots, V_r$ each of size $2n$. Suppose $G$ is a $\Part'$-partite graph with $\delta^*(G) \geq n - \alpha n$, and that there exists a partition $\Part$ refining $\Part'$ into parts each of size at least $n - \mu n$ such that $L^\mu_\Part(G)$ is incomplete with respect to $\Part'$. Then $G$ is $d$-splittable or $d$-pair-complete.
\end{lemma}
 
\proof
We can assume that $L^\mu_\Part(G)$ is transferral-free (recall from Section~\ref{sec:theory} that this means that $L^\mu_\Part(G)$ does not contain any difference of basis vectors $\unit_i-\unit_j$, for some $i \neq j$ which index subparts of the same part of $\Part'$). Thus for any $i,j \in [r]$ with $i \ne j$, distinct parts $A$, $B$ of $\Part$ contained in $V_i$, and part $C$ of $\Part$ contained in $V_j$, we cannot have both $e(A,C) \ge \mu (2rn)^2$ and $e(B,C) \ge \mu (2rn)^2$. Since all parts of $\Part$ have size at least $n - \mu n$, each part of $\Part'$ is partitioned into at most two parts of $\Part$. We can assume that $\Part$ is a strict refinement of $\Part'$, so without loss of generality $\Part$ partitions $V_1$ into two parts $V_1^1$ and $V_1^2$. Next we note that there cannot be any part $V_j$ that is not partitioned into two parts by $\Part$. For otherwise, letting $A = V^1_1$, $B = V^2_1$ and $C=V_j$, we have $e(A,C) \ge |A| \delta^*(G) \ge \mu (2rn)^2$ and $e(B,C) \ge |B| \delta^*(G) \ge \mu (2rn)^2$, which contradicts $L^\mu_\Part(G)$ being transferral-free, as described above. Thus each $V_j$ is partitioned into two parts $V^1_j$ and $V^2_j$ by $\Part$. Form an auxiliary graph $H$ on $2r$ vertices, where for each $i \in \{1,2\}$ and $j \in [r]$ we have a vertex $x^i_j$ of $H$ corresponding to the part $V^i_j$ of $\Part$, and we have an edge $x^i_jx^{i'}_{j'}$ if and only if $e(V^i_j, V^{i'}_{j'}) \ge \mu (2rn)^2$. Since $\delta^*(G) \geq n - \alpha n$ and each part of $\Part$ has size at least $n - \mu n$, we have $\delta^*(H) \geq 1$. 
Furthermore, for any edge $ab$ of $H$, the at least $\mu (2rn)^2$ edges of $G$ corresponding to this edge ensure that $\ub_a + \ub_b \in L_\Part^\mu(G)$, where we consider $a$ and $b$ to index the parts of $\Part$ to which they correspond. It follows that $H$ cannot contain a path of even length between $x^1_j$ and $x^2_j$ for any $j \in [r]$. Indeed, if $x^1_j = y_0, y_1, \dots, y_{2m} = x^2_j$ are the vertices of such a path then, since $\ub_{y_{\ell-1}} + \ub_{y_{\ell}} \in L_\Part^\mu(G)$ for each $\ell \in [2m]$, we have $$\ub_{x^1_j} - \ub_{x^2_j} = \ub_{y_0} - \ub_{y_{2m}} = \sum_{\ell \in [m]} \big((\ub_{y_{2\ell-2}} +  \ub_{y_{2\ell-1}}) - (\ub_{y_{2\ell-1}} + \ub_{y_{2\ell}})\big) \in L_\Part^\mu(G),$$
a contradiction to $L^\mu_\Part(G)$ being transferral-free.

We may therefore apply Proposition~\ref{auxgraph} to deduce that there exists some $A \subseteq V(H)$ with $|A \cap \{x^1_j, x^2_j\}| = 1$ for each $j \in [r]$ such that either 
\begin{itemize}
\item[(i)] $H$ contains no edges $ab$ with $a,b \in A$ or $a,b \notin A$, or
\item[(ii)] $H$ contains no edges $ab$ with $a \in A$ and $b \notin A$.
\end{itemize}
For each $j \in [r]$ let $S'_j = V^1_j$ if $x^1_j \in A$, and $V^2_j$ otherwise. For $j \in [r]$, let $S_j \sub V_j$ be a set of size $n$ that contains or is contained by $S'_j$. Note that for any $j \neq j'$, we have $e(S_j, V_{j'}) \geq n\delta^*(G) \geq (1-\alpha)n^2$, and similarly $e(V_j \sm S_j, V_{j'}) \geq (1-\alpha)n^2$. In case (i), for any $j \neq j'$ we have $e(S'_j, S'_{j'}) \leq \mu (2rn)^2$. This implies $e(S_j, V_{j'} \sm S_{j'}) \geq (1-\alpha - 8r^2\mu) n^2$, and so $G[S_j, V_{j'} \sm S_{j'}]$ has density at least $1-d$. Since $j$ and $j'$ were arbitrary, we may conclude that $G$ is $d$-splittable. On the other hand, in case (ii), for any $j \neq j'$ we have $e(S'_j, V_{j'} \sm S'_{j'}) \leq \mu (2rn)^2$. This implies $e(S_j, S_{j'}) \geq (1-\alpha - 8r^2\mu) n^2$, and so $G[S_j, S_{j'} ]$ has density at least $1-d$. Similarly, $G[V_j \sm S_j, V_{j'} \sm S_{j'}]$ has density at least $1-d$, and $G[S_j, V_{j'} \sm S_{j'}]$ has density at most $d$ for any $j \neq j'$, so $G$ is $d$-pair-complete.
\endproof

\subsection{Avoiding divisibility barriers for $p > 2$.} \label{subsec:div}
We next show that for $p > 2$, if $G$ is not $d$-splittable then $J_p$ is not close to a divisibility barrier.

\begin{lemma} \label{split}
Suppose that $1/n \ll  \mu, \alpha \ll d \ll 1/r$ and $3 \le p \le r$. Let $\Part'$ partition a set $V$ into vertex classes $V_1, \dots, V_r$ each of size $pn$. Suppose $G$ is a $\Part'$-partite graph with $\delta^*(G) \geq (p-1)n - \alpha n$ and let $J =J(G)$ be the clique $p$-complex of $G$. Suppose $\Part$ is a partition refining $\Part'$ into parts each of size at least $n - \mu n$ such that $L^\mu_\Part(J_p)$ is incomplete with respect to $\Part'$. Then $G$ is $d$-splittable.
\end{lemma}

\proof
We introduce new constants with $\mu, \alpha \ll \mu' \ll c \ll \gamma \ll \gamma' \ll \gamma'' \ll d$. We can assume that $\Part$ is a strict refinement of $\Part'$, so without loss of generality $\Part$ partitions $V_1$ into parts $V^i_1$, $i \in [m]$ with $2 \le m \le p$. As in the proof of Lemma \ref{2splitorpc}, we can also assume that $L^\mu_\Part(J_p)$ is transferral-free, meaning that it does not contain any difference of basis vectors $\unit_i-\unit_j$, for some $i \neq j$ which index subparts of the same part of $\Part'$. Thus for any $p$ vertex classes $V_{i_1},\dots,V_{i_p}$, distinct parts $U_{i_1}$, $U'_{i_1}$ of $\Part$ contained in $V_{i_1}$, and parts $U_{i_j} \sub V_{i_j}$ for $2 \le j \le p$, it cannot be that $\bigcup_{j=1}^p U_{i_j}$ and $U'_1 \cup \bigcup_{j=2}^p U_{i_j}$ both have at least $\mu |V(G)|^p = \mu (rpn)^p$ edges of $J_p$ (i.e.\ copies of $K_p$). We use this to deduce the following properties, which control the typical behaviour of neighbourhoods and certain pairwise intersections of neighbourhoods. Note that the bound in (c) is close to the lower bound on $\delta^*(G)$, so it says that $G$ is mostly approximately regular from the point of view of $V_1$.

\begin{claim} \label{splitc1} \
\begin{itemize}
\item[(a)] There are at most $\mu' n^2$ pairs $(x,y)$ for which there exist $i \in [m]$ and $j \in [r]$ such that $x \in V_1^i$, $y \in V_1 \sm V_1^i$ and $|V_j \cap N(x) \cap N(y)| \geq (p-2) n + \mu' n$.
\item[(b)] There are at most $\mu' n^3$ triples $(x,y,z)$ for which there exist $i \in [m]$ and $j \in [r]$ such that $x \in V_1^i$, $y \in V_1 \sm V_1^i$, $z \in V \sm V_1$, $xz, yz \in G$ and $|V_j \cap N(x) \cap N(z)| \geq (p-2) n + \mu' n$.
\item[(c)] For any $2 \leq j \leq r$ there are at most $2\mu' n$ vertices $x \in V_1$ such that $|N(x) \cap V_{j}| \geq (p-1)n + 2\mu' n$.
\end{itemize}
\end{claim}

For (a), suppose for a contradiction that there are more than $\mu' n^2$ such pairs. Without loss of generality there are at least
$\mu' n^2/rp^2$ pairs $(x,y)$ with $x \in V_1^1$ and $y \in V_1^2$ such that $|V_p \cap N(x) \cap N(y)| \geq (p-2) n + \mu'
n$. For each such pair, we consider greedily choosing $w_j \in V_j$, $2 \le j \le p$ such that $xw_2\dots w_p$ and $yw_2\dots w_p$
are copies of $K_p$. The number of choices for $w_j$ is $N_j \ge |V_j \cap N(x) \cap N(y)| - \sum_{i=2}^{j-1} |V_j \sm
N(w_i)|$. For $2 \le j \le p-1$, we have $N_j \ge pn - (p-1)(pn-\delta^*(G)) \ge (1-(p-1)\alpha)n > n/2$. Also, $N_p \ge (p-2) n +
\mu' n - (p-2)(pn-\delta^*(G))\geq \mu' n - (p-2)\alpha n \geq \mu'n/2$. Considering all such pairs $(x,y)$, we obtain at least
$(\mu')^2 n^{p+1}/2^{p-1}rp^2$ such $(p+1)$-tuples $(x,y,w_2, \dots, w_p)$. There are at most $p^{r+1}$ possible indices for such
a $(p+1)$-tuple, so we may choose $\mu (rpn)^{p+1}$ such $(p+1)$-tuples which all have the same index; let $(x,y,w_2, \dots, w_p)$
be a representative of this collection. Then there must be at least $\mu (rpn)^p$ edges of $J_p$ with index $\ib(\{x, w_2, \dots,
w_p\})$, and at least $\mu (rpn)^p$ edges of $J_p$ with index $\ib(\{y, w_2, \dots, w_p\})$. But this contradicts
$L^\mu_\Part(J_p)$ being transferral-free. A very similar argument applies for (b). Indeed, suppose for a contradiction that there are more than $\mu' n^3$ such triples. Say there are at least $\mu' n^3/r^2p^2$ triples $(x,y,z)$ with $x \in V_1^1$, $y \in V_1^2$, $z \in V_2$ such that $xz, yz \in G$ and $|V_p \cap N(x) \cap N(z)| \geq (p-2) n + \mu' n$. For each such triple we consider greedily choosing $w_j \in V_j$, $3 \le j \le p$ such that $xzw_3\dots w_p$ and $yzw_3\dots w_p$ are copies of $K_p$. The number of choices for $w_j$ is $N_j \ge |V_j \cap N(x) \cap N(z)| - \sum_{i=3}^{j-1} |V_j \sm N(w_i)| - |V_j \sm N(y)|$. Thus the same calculation as in (a) gives a contradiction. For (c), suppose for a contradiction that there are at least $2\mu' n$ such vertices $x$. For each such $x$, and each choice of $y \in V_1$ in a different part of $\Part$ to $x$, we have $|V_j \cap N(x) \cap N(y)| \geq |N(x) \cap V_{j}| - (pn-\delta^*(G)) \ge (p-2) n + \mu' n$. There are at least $n - \mu n$ choices of $y$ for each $x$, so this contradicts (a). Thus we have proved Claim \ref{splitc1}. \medskip

Now for each $i \in [m]$ and $2 \leq j \leq r$ let $X_j^i$ consist of all vertices of $V_j$ which have at most $|V_1^i| - \gamma n$ neighbours in $V_1^i$. Bearing in mind the row structure we are aiming for, the intuition is that $X_j^i$ should approximate the $j$th part of the $i$th row. We show the following properties that agree with this intuition: the size of $X^i_j$ is roughly correct, and the intended diagonal densities are close to $1$.

\begin{claim} \label{splitc2} \
\begin{itemize}
\item[(d)] For any $i \in [m]$ and $2 \leq j \leq r$ we have $|X^i_j| \geq |V_1^i| - \gamma' n > n/2$.
\item[(e)] For any $i, i' \in [m]$ with $i \neq i'$ and any $2 \leq j \leq r$ we have $d(V_1^i, X_j^{i'}) \geq 1 - c$.
\item[(f)] For any $2 \leq j \leq r$ at most $\gamma n$ vertices $v \in V_j$ lie in more than one of the sets~$X^i_j$. Thus $|X^i_j| \le |V_1^i| + p\gamma' n$ for any $i \in [m]$, and all but $2p\gamma' n$ vertices of $V_j$ lie in $\bigcup_{i \in [m]} X^i_j$.
\item[(g)] For any $i, i' \in [m]$ with $i \neq i'$ and any $2 \leq j < j' \leq r$ we have $d(X_j^{i}, X_{j'}^{i'}) \geq 1 - d/2$.
\end{itemize}
\end{claim}

For (d), note that $e(V_1^i,V_j) \le ((p-1)n + 2\mu' n)|V_1^i| + 2 \mu' pn^2 \leq (p-1 + 5p\mu')n|V_1^i|$ by (c). Also,
\begin{align*}
e(V_1^i,V_j) & \ge |X^i_j|(|V_1^i| - (pn - \delta^*(G))) + (pn - |X_j^i|)(|V_1^i| - \gamma n) \\
& \ge pn(|V_1^i|-\gamma n) - |X^i_j|(1 + \alpha - \gamma)n.
\end{align*}
Therefore $|X^i_j|(1 + \alpha - \gamma)n \ge (1 - 5p\mu')|V_1^i|n - p\gamma n^2$, which gives (d).
Next suppose for a contradiction that (e) is false, say that $d(V_1^1, X_2^2) < 1 - c$. Let $A$ be the set of vertices in $X_2^2$ with fewer than $(1-c/2) |V_1^1|$ neighbours in $V_1^1$. Write $|A|=a|X_2^2|$. Then $(1-a)(1-c/2) \le d(V_1^1, X_2^2) < 1 - c$, so $a \ge c/2$. Thus $|A| \ge cn/4$ by (d). Each vertex in $A$ has fewer than $|V_1^1| - cn/4$ neighbours in $V_1^1$, and also fewer than $|V_1^2| - \gamma n$ neighbours in $V_1^2$ by definition of $X_2^2$. This gives at least $(cn/4)^2 \gamma n$ triples $(x,y,z)$ with $x \in V_1^1$, $y \in V_1^2$, and $z \in X_2^2$ such that $xz, yz \notin G$. At least $\mu' n^2$ pairs $(x,y)$ therefore lie in at least $2 \mu' n$ such triples. For each of these pairs we have $|V_2 \cap N(x) \cap N(y)| \geq pn - 2(pn-\delta^*(G)) + 2\mu'n \geq (p-2)n + \mu'n$. However, this contradicts (a), so (e) holds. For (f), suppose for a contradiction that $A := X_2^i \cap X_2^{i'}$ has size at least $\gamma n/p^2$, for some $i, i' \in [m]$ with $i \neq i'$. By definition, each vertex of $A$ has at most $|V_1^i| - \gamma n$ neighbours in $|V_1^i|$, so $d(A, V_1^i) \leq 1 - \gamma/p$. But then $d(X_2^{i'}, V_1^i) \leq 1 - \gamma^2/p^4$, contradicting (e). So $|A| < \gamma n/p^2$, and summing over all possible values of $i$ and $i'$ we obtain the first statement of (f). This implies that
$\sum_{i \in [m]} |X^i_j| \le pn + m\gamma n$, so $\sum_{i \in [m]} (|X^i_j| - |V_1^i|) \le m\gamma n$. In combination with (d) this implies $|X^i_j| \le |V_1^i| + p\gamma' n$ for any $i \in [m]$. Also, (d) gives $\bsize{\bigcup_{i \in [m]} X^i_j} \ge \sum_{i \in [m]} (|V_1^i| - \gamma' n) - m\gamma n \ge pn - 2p\gamma' n$, so (f) holds.

Finally, suppose for a contradiction that (g) is false, say $d(X_2^1, X_3^2) < 1 - d/2$. Without loss of generality we have $|V_1^2| \leq pn/2$. Let $A$ be the set of vertices in $X_3^2$ with fewer than $(1-d/4) |X_2^1|$ neighbours in $X_2^1$. (We re-use $A$ to avoid excessive notation.) Write $|A|=a|X_3^2|$. Then  $(1-a)(1-d/4) \le d(X_2^1, X_2^3) < 1 - d/2$, so $a \ge d/4$. Thus $|A| \ge dn/8$ by (d). Each vertex in $A$ has fewer than $|X_2^1| - dn/8$ neighbours in $X_2^1$, and also fewer than $|V_1^2| - \gamma n$ neighbours in $V_1^2$ by definition of $X_3^2$. This gives a set $T$ of at least $(dn/8)^2 \gamma n$ triples $(x,z,w)$ with $x \in V_1^2$, $z \in X_2^1$ and $w \in X_3^2$ such that $xw, zw \notin G$. Furthermore, since by (e) we have $d(V_1^2, X_2^1) \geq 1-c$, all but at most $c(pn)^3$ triples in $T$ have the additional property that $xz \in G$. Let $P$ be the set of pairs $(x,z)$ with $xz \in G$, $x \in V_1^2$, $z \in X_2^1$ that lie in at least $2 \mu' n$ triples of $T$. Then $|T| - c(pn)^3 \le |P|pn + (pn)^2 2\mu' n$, so $|P| \ge 3 \mu' n^2$, say. Since $|V_1^2| \leq pn/2$ and $p \geq 3$, for each $(x,z) \in P$ there are more than $pn/2 - (pn-\delta^*(G)) > n/3$ vertices $y$ such that $y \in V_1 \sm V_1^2$ and $yz \in G$. Note that this is a key use of the assumption $p \ge 3$, so we had to deal with the case $p=2$ separately in the previous subsection. There are therefore more than $\mu' n^3$ triples $(x,y,z)$ with $x \in V_1^2$, $y \in V_1 \sm V_1^2$ and $z \in X_2^1$ such that $xz, yz \in G$ and $(x,z) \in P$. However, for any $(x,z) \in P$ we have $|V_3 \cap N(x) \cap N(z)| \geq pn - 2(pn-\delta^*(G)) + 2\mu'n \geq (p-2)n + \mu'n$, which contradicts (b). Thus (g) holds, proving Claim \ref{splitc2}.\medskip

To complete the proof, we also need to show that the size of each part of $V_1$ is close to an integer multiple of $n$. Since each part of $V_1$ has size at least $n - \mu n$, this will be true if $V_1$ has a part of size close to $(p-1)n$. So for the final claim we assume that $V_1$ does not have such a large part; in this case we extend (c) by showing that the bipartite graph induced by any pair of vertex classes is mostly approximately regular. We then show that most vertices in $V_1$ have sparse non-neighbourhoods, before finally deducing the required statement on the sizes of the parts of~$V_1$.

\begin{claim} \label{splitc3} Suppose $V_1$ does not have a part of size at least $(p-1)n - \gamma n$. Then
\begin{itemize}
\item[(h)] for any $j, j' \in [r]$ there are at most $cn$ vertices $z \in V_j$ such that $|N(z) \cap V_{j'}| \geq (p-1)n + 2\mu' n$,
\item[(j)] for a set $V'_1$ of all but at most $\gamma n$ vertices $x \in V_1$ we have $|V_j \sm N(x)| = n \pm 2\mu'n$ and $d(V_j \sm N(x), V_{j'} \sm N(x)) \leq \gamma$ for any $2 \leq j, j' \leq r$ with $j \neq j'$, and
\item[(k)] for each $i$ there is an integer $p_i$ such that $|V_1^i| = p_in \pm \gamma'' n$.
\end{itemize}
\end{claim}

For (h), note that by our assumption on the part sizes of $V_1$, any such $z$ lies in at least $\gamma^2 n^2/2$ triples $(x,y,z)$ such that $xz, yz \in G$ and $x$ and $y$ lie in different parts of $V_1$. Any such triple is counted by (b), as $|V_{j'} \cap N(x) \cap N(z)| \ge |N(z) \cap V_{j'}| - (pn - \delta^*(G)) \ge (p-2)n + \mu' n$, so there can be at most $cn$ such vertices $z$, as claimed. For (j) we introduce the following notation: $N_j(x) := N(x) \cap V_j$ is the set of neighbours of $x$ in $V_j$, and $N_j^c(x) := V_j \sm N(x)$ is the set of non-neighbours of $x$ in $V_j$. Fix some $j$ and $j'$, and suppose $x \in V_1$ is such that $d(N^c_{j}(x), N^c_{j'}(x)) \geq \gamma$ and $|N_j(x)|, |N_{j'}(x)| \leq (p-1)n + 2 \mu'n$.  We can estimate $d(N_j(x), N^c_{j'}(x))$ as follows. Write $e(N_j(x), N^c_{j'}(x)) = e(V_j, N^c_{j'}(x)) - e(N^c_{j}(x), N^c_{j'}(x))$. Then by (h) we have $e(V_j, N^c_{j'}(x)) = \sum_{v \in N^c_{j'}(x)} |N(v) \cap V_j| \le ((p-1)n + 2\mu' n)|N^c_{j'}(x)| + cn \cdot pn$. Also, $e(N^c_{j}(x), N^c_{j'}(x)) \ge \gamma |N^c_{j}(x)| |N^c_{j'}(x)| \ge \gamma (n-2\mu' n)|N^c_{j'}(x)|$ by choice of $x$. This gives \[e(N_j(x),N^c_{j'}(x)) \le (p-1 + 2\mu'-(1-2\mu')\gamma )n|N^c_{j'}(x)| + cpn^2.\]
Since $|N_j(x)| \ge \delta^*(G) \ge (p-1)n - \alpha n$, and $|N_{j'}^c(x)| \ge n - 2 \mu' n$ by choice of $x$, we deduce that $d(N_j(x), N^c_{j'}(x)) \leq 1- \gamma/2$. Let $A$ be the set of vertices in $N_j(x)$ with fewer than $(1-\gamma/4)|N^c_{j'}(x)|$ neighbours in $N^c_{j'}(x)$. Write $|A|=a|N_j(x)|$. Then $(1-a)(1-\gamma/4) \le d(N_j(x), N^c_{j'}(x)) \leq 1- \gamma/2$, so $a \ge \gamma/4$, and $|A| \ge \gamma n/4$. 
Moreover, by definition of $A$ any $z \in A$ has at least $|N_{j'}(z)| - (1-\gamma/4)|N^c_{j'}(x)|$  neighbours in $N_{j'}(x)$. Since $|N_{j'}(z)| \geq \delta^*(G) \geq (p-1)n - \alpha n$ and $|N^c_{j'}(x)| \leq pn - \delta^*(G) \leq n + \alpha n$, it follows that for any $z \in A$ we have
\begin{align*}
|V_{j'} \cap N(x) \cap N(z)| & \ge (p-1)n-\alpha n - (1-\tfrac{\gamma}{4})(n + \alpha n) \\
&\geq (p-2) n + \gamma n/4 - 2 \alpha n \geq (p-2) n + \mu' n.
\end{align*}
Furthermore, since $V_1$ does not have a part of size at least $(p-1)n - \gamma n$, there must be at least $\gamma n/2$ neighbours $y$ of $z$ which lie in a different part of $V_1$ to $x$. There are at least $\gamma n/4$ choices for $z \in A$, so $x$ lies in at least $\gamma^2 n^2/8$ triples $(x,y,z)$ counted in (b). Thus there are at most $\frac{\mu' n^3}{\gamma^2 n^2/8} < \gamma n/2r^2$ such vertices $x \in V_1$ with $d(N_j^c(x), N^c_{j'}(x)) \geq \gamma$ and $|N_j(x)|, |N_{j'}(x)| \leq (p-1)n + 2 \mu'n$. Since by (c) at most $4\mu'n$ vertices do not satisfy the latter condition, summing over all $j,j' \in [r]$ gives (j). (Every vertex $x \in V_1$ satisfies $|N_j(x)| \geq \delta^*(x) \geq (p-1)n - \alpha n$.)

For (k), consider any $x,y \in V'_1$ as defined in (j), and let $I^{xy}_j = N^c_j(x) \cap N^c_j(y)$ for $j=2,3$. We will show that either $|I^{xy}_2| \le 3\sqrt{\gamma}n$ or $|I^{xy}_2| \ge (1-3\sqrt{\gamma})n$. For suppose that $I^{xy}_2 > 3\sqrt{\gamma}n$. Let $B = N^c_3(x) \cup N^c_3(y)$. By definition of $V'_1$ we have $e(I^{xy}_2,B) \le e(N^c_2(x),N^c_3(x)) + e(N^c_2(y),N^c_3(y)) \le 3\gamma n |B| \le \sqrt{\gamma} |I^{xy}_2||B|$, so there is a vertex $z \in I^{xy}_2$ with $|N(z) \cap B| \le \sqrt{\gamma} |B|$. Then $(1-\sqrt{\gamma})|B| \le |B \sm N(z)| \le |V_3 \sm N(z)| \le n + \alpha n$. This gives $|B| \le (1+2\sqrt{\gamma})n$, so $|I^{xy}_3| = |V_3 \sm N(x)| + |V_3 \sm N(y)| - |B| \ge (1-3\sqrt{\gamma})n$. Now the same argument interchanging $I^{xy}_2$ and $I^{xy}_3$ shows that $|I^{xy}_2| \ge (1-3\sqrt{\gamma})n$, as required.

Next we define a relation $\sim$ on $V'_1$ by $x \sim y$ if $|I^{xy}_2| \ge (1-3\sqrt{\gamma})n$. This is an equivalence relation, since if $|I^{xy}_2| \ge (1-3\sqrt{\gamma})n$ and $|I^{yz}_2| \ge (1-3\sqrt{\gamma})n$, then $|I^{xz}_2| \ge |I^{xy}_2| - |N_2^c(y) \sm N_2^c(z)| \ge (1-3\sqrt{\gamma})n - ((n+\alpha n) - (1-3\sqrt{\gamma})n) > 3\sqrt{\gamma}n$, so $|I^{xy}_2| \ge (1-3\sqrt{\gamma})n$ as just shown. Let $C^1_1, \dots, C^t_1$ be the equivalence classes of $\sim$, and arbitrarily choose a representative $x_i$ of each equivalence class $C^i_1$. 

Since each $x_i$ lies in $V_1'$, the sets $N^c_2(x_i)$ each have size $n \pm 2 \mu' n$ by (j). Furthermore, any two such sets intersect in at most $3\sqrt{\gamma}n$ vertices, since the representatives $x_i$ each lie in different equivalence classes. We cannot have $t>p$, as then $(p+1)(n - 2 \mu' n) - \binom{p+1}{2} 3\sqrt{\gamma}n \le \bsize{\bigcup_{i=1}^{p+1} N^c_2(x_i)} \le |V_2| \le pn$ is a contradiction, so we must have $t \leq p$. Next, observe that any vertex in $C_1^i$ has at most $(n + 2 \mu' n) - (1-3\gamma)n \leq 4 \gamma n$ neighbours in $N^c_2(x_i)$. So $e(C_1^i, N^c_2(x_i)) \leq 4 \gamma n|C_1^i|$. By averaging, some vertex $v \in N^c_2(x_i)$ therefore has at most $4 \gamma n|C_1^i|/(n - 2\mu' n) \leq 5 \gamma |C_1^i|$ neighbours in~$C_1^i$. So $(pn - |C_1^i|) + 5 \gamma |C_1^i| \geq |N(v) \cap V_1|  \geq \delta^*(G) \geq (p-1)n - \alpha n$, which implies $|C_1^i| \leq n + 6 \gamma n$. Since the $t \leq p$ equivalence classes $C_1^i$ partition $V_1'$, we deduce that $t=p$ and $|C_1^i| = n \pm \gamma' n$ for any $i \in [t]$.

Now we show that any equivalence class $C_1^{i'}$ must be essentially contained in some  part $V^i_1$ of $V_1$; by symmetry it suffices to show that this is true of $C_1^1$. So observe that since $|N^c_2(x_1)| \geq n - 2\mu' n$, by (f) we must have $|N^c_2(x_1) \cap X_2^i| \geq n/2p$ for some $i \in [m]$. Fix such an $i$, and suppose for a contradiction that $|C_1^1 \cap V_1^{i'}| \geq \gamma n$ for some $i' \neq i$. We observed above that any $v \in C_1^1$ has at most $4 \gamma n$ neighbours in $N^c_2(x_1)$, so there are at least $n/2p - 4\gamma n \geq n/3p$ vertices of $X_2^i$ which are not neighbours of $v$. Then $e(V_1^{i'}, X_2^{i}) \leq |V_1^{i'}||X_2^i| - |C_1^1 \cap V_1^{i'}|n/3p < (1-c)|V_1^{i'}||X_2^i|$, contradicting (e). We conclude that all but at most $p\gamma n$ vertices of $C_1^1$ lie in $V_1^i$, and thus that all but at most $p\gamma n$ vertices of any equivalence class lie in the same part of $V_1$. So for any $i, i'$ we have either $|V_1^i \cap C_1^{i'}| \leq p\gamma n$ or $|V_1^i \cap C_1^{i'}| = |C_1^{i'}| \pm p\gamma n =  n \pm 2\gamma' n $; (k) follows immediately since the classes $C_1^1, \dots, C_1^t$ partition $V_1'$ and $|V_1 \sm V_1'| \leq \gamma n$ by (j). This completes the proof of Claim~\ref{splitc3}.

 \medskip

To complete the proof of Lemma~\ref{split}, note that there exist integers $p_1, \dots, p_m$ such that $|V_1^i| = p_in \pm \gamma'' n$ for each $i \in [m]$. Indeed, if $V_1$ has a part of size at least $(p-1)n - \gamma n$, then since each part of $\Part$ has size at least $n - \mu n$, we may assume that $V_1$ has two parts $V_1^1$ and $V_1^2$ with respective sizes $(p-1) n \pm \gamma n$ and $n \pm \gamma n$. On the other hand, if $V_1$ has no such part then the required integers $p_i$ exist by (k). We partition $V_1$ into sets $U^i_1$ with $|U^i_1| = p_in$ for $i \in [m]$ such that each $U^i_1$ either contains or is contained in some $V^i_1$. Then $U_1^i$ contains at least $p_in - \gamma '' n \geq |V_1^i| - 2\gamma'' n$ vertices of $V_1^i$ for any $i \in [m]$. Furthermore, by (d) and (f), for each $2 \leq j \leq r$ we may partition $V_j$ into sets $U^i_j$ with $|U^i_j| = p_i n$ for $i \in [m]$ such that each $U^i_j$ contains at least $|X^i_j| - 2p\gamma' n \geq p_in - 2\gamma'' n$ vertices from $X^i_j$.  
By (e) and (g) we deduce that $d(U^i_j,U^{i'}_{j'}) \ge 1-d$ whenever $i \ne i'$ and $j \ne j'$. In particular, $d(U^1_j,V_{j'} \sm U^1_{j'}) \ge 1-d$ for any $j \ne j'$, so $G$ is $d$-splittable.
\endproof

Combining Lemmas \ref{2splitorpc} and \ref{split}, if $G$ is neither $d$-splittable nor $d$-pair-complete (if $p=2$) then there is no divisibility barrier to a perfect matching in $J_p$. We saw in Lemma \ref{splittable} that there is also no space barrier to a perfect matching in $J_p$. So Theorem~\ref{partitematching} implies that $G$ contains a near-balanced perfect matching. The following corollary formalises this argument.

\begin{coro} \label{findmatching} 
Suppose that $1/n \ll \gamma \ll \alpha \ll d \ll 1/r$ and $2 \le p \le r$. Let $G$ be an $r$-partite graph on vertex classes $V_1, \dots, V_r$ each of size $pn$ with $\delta^*(G) \geq (p-1)n - \alpha n$. Suppose also that $G$ is neither $d$-splittable nor $d$-pair-complete. Let $J = J(G)$ be the clique $p$-complex of $G$. Then $J_p$ contains a $\gamma$-balanced perfect matching. 
\end{coro}

\proof
Introduce new constants with $1/n \ll \gamma \ll \alpha \ll \mu, \beta \ll d \ll 1/r$. As described in Section~\ref{sec:theory}, the condition $\delta^*(G) \geq (p-1)n - \alpha n$ implies
$$\delta^*(J) \geq \left(pn, (p-1)n - \alpha n, (p-2)n - 2\alpha n, \dots, n - (p-1)\alpha n\right).$$
Suppose that $J_p$ has no $\gamma$-balanced perfect matching. Then by Theorem~\ref{partitematching} (with $pn$ in place of $n$ and $p$ in place of $k$) we deduce that there is either a space barrier or divisibility barrier. Consider first a space barrier. This means that there exist $p' \in [p-1]$ and $S \sub V$ with $|S \cap V_i| = p'n$ for each $i \in [r]$ so that $J_p$ is $\beta$-contained in $J_r(S,p')_p$, that is, all but at most $\beta (rpn)^p$ edges of $J_p$ have at most $p'$ vertices in $S$. However, since $G$ is not $d$-splittable, Lemma~\ref{splittable}(i)  (with $2\beta (rp)^p$ in place of $\beta$) implies that more than $\beta (rp)^p n^p$ copies of $K_{p}$ in~$G$ have more than $p'$ vertices in $G[S]$. Since each copy of $K_p$ in $G$ is an edge of $J_p$, there cannot be a space barrier.

Now suppose that there is a divisibility barrier. This means that there is some partition $\Part$ of $V(J)$ into parts of size at least  $\delta^*_{p-1}(J) - \mu pn \geq n - 2p\mu n$ such that $\Part$ refines the partition $\Part'$ of $V(G)$ into $V_1, \dots, V_r$ and $L^\mu_\Part(J_p)$ is incomplete with respect to $\Part'$. But if $p \geq 3$ then Lemma~\ref{split} (with $2p\mu$ in place of $\mu$) implies that $G$ is $d$-splittable, contradicting our assumption. Similarly, if $p=2$ then Lemma~\ref{2splitorpc} (with $4\mu$ in place of $\mu$) implies that $G$ is $d$-splittable or $d$-pair-complete, again contradicting our assumption. We conclude that $J_p$ must contain a $\gamma$-balanced perfect matching. 
\endproof

\section{Finding packings within rows} \label{sec:packing}

Recall from the proof outline given in Section~\ref{sec:outline} that step (ii) in proving Theorem~\ref{partitehajnalszem} is to find a balanced perfect $K_{p_i}$-packing in each row $G[X^i]$. In this section we demonstrate how this may be achieved. We need to consider two cases. The first case is where $G[X^i]$ is neither $d$-splittable nor $d$-pair-complete. Then Corollary~\ref{findmatching} gives a $\gamma$-balanced perfect $K_{p_i}$-packing in $G[X^i]$. In Lemma~\ref{theoremmatching} we show how such a matching can be `corrected' to a balanced perfect $K_{p_i}$-packing in this case if $p_i \geq 3$, and also if $p_i=2$ provided that $G[X^i]$ contains many $4$-cycles of a given type. If $p_i=2$ and $G[X^i]$ does not contain such $4$-cycles then it may not be possible to find a balanced perfect matching in $G[X^i]$. Proposition~\ref{onepcrow} will allow us to handle this case by deleting further copies of $K_k$ from $G$ so that the remainder of row $i$ does contain a balanced perfect matching. The second case is where $G[X^i]$ is $d$-pair-complete. Then we prove Lemma~\ref{paircompletematching}, which shows that $G[X^i]$ contains a perfect matching provided a parity condition is satisfied. Both here and later we use the fact that, if we add or remove a small number of vertices to or from each block $X^i_j$ of a row $X^i$ of $G$ which is neither $d$-splittable nor $d$-pair-complete, then the new row obtained is neither $d'$-splittable nor $d'$-pair-complete for $d' \ll d$. This is established by the following proposition.
 
\begin{prop} \label{robustness}
Suppose that $1/n, 1/n' \ll \zeta \ll d' \ll d \ll 1/r$ and $r \geq p \geq 1$. Let $G$ be an $r$-partite graph on vertex classes $V_1, \dots, V_r$, and for each $j \in [r]$ let $X_j, X'_j \subseteq V_j$ be such that $|X_j| = pn$, $|X_j'| = pn'$ and $|X_j \triangle X'_j| \leq \zeta pn$. Let $X = \bigcup_{j \in [r]} X_j$ and $X' = \bigcup_{j \in [r]} X'_j$; then the following statements hold.
\begin{enumerate}[(i)]
\item If $G[X]$ is not $d$-splittable then $G[X']$ is not $d'$-splittable.
\item If $p=2$ and $G[X]$ is not $d$-pair-complete then $G[X']$ is not $d'$-pair-complete.
\end{enumerate}
\end{prop}

\proof
Note that $n' = (1 \pm \zeta) n$. For (i), suppose for a contradiction that $G[X']$ is $d'$-splittable. Then by definition we may choose $p' \in [p-1]$ and subsets $S'_j \subseteq X'_j$ with $|S'_j| = p'n'$ for $j \in [r]$ such that $e(S'_j, X'_{j'} \sm S'_{j'}) \geq (1-d')p'(p-p')n'{}^2$ for any $j' \neq j$. For each $j \in [r]$ we choose $S_j \sub X_j$ such that $|S_j| = p'n$ and $S_j$ either contains or is contained in $S'_j \cap X_j$. Note that $|S_j' \triangle S_j| \leq 2\zeta pn$ and $|(X_j' \sm S_j') \triangle (X_j \sm S_j)| \leq 2\zeta pn$. We deduce that for any $j' \neq j$ we have
\begin{align*}
e(S_j, X_{j'} \sm S_{j'}) &\geq e(S'_j, X'_{j'} \sm S'_{j'}) - |S'_j \sm S_j| pn - |(X'_{j'} \sm S'_{j'}) \sm (X_{j'} \sm S_{j'})| pn 
\\ &\geq (1-d' )p'(p-p')n'{}^2 - 4\zeta p^2n^2 \geq (1-d)p'(p-p')n{}^2.
\end{align*}
Then $G[X]$ is $d$-splittable with respect to the sets $S_j$, a contradiction, so this proves (i). For (ii), suppose for a contradiction that $G[X']$ is $d'$-pair-complete. Then by definition we may choose subsets $S'_j \subseteq X'_j$ with $|S'_j| = n'$ for $j \in [r]$ such that $e(S'_j, S'_{j'}) \geq (1-d')n'{}^2$, $e(X'_j \sm S'_j, X'_{j'} \sm S'_{j'}) \geq (1-d')n'{}^2$ and $e(S'_j, X'_{j'} \sm S'_{j'}) \leq d'n'{}^2$ for any $j' \neq j$. We take $p'=1$ and choose $S_j$ for $j \in [r]$ as in (i).
Similar calculations as in (i) show that $e(S_j, S_{j'}) \geq (1-d)n^2$ and $e(X_j \sm S_j, X_{j'} \sm S_{j'}) \geq (1-d)n^2$ for any $j' \neq j$. We deduce that
\begin{align*}
e(S_j, X_{j'} \sm S_{j'}) & \leq e(S'_j, X'_{j'} \sm S'_{j'}) + |S_j \sm S'_j| \cdot 2n + |(X_{j'} \sm S_{j'}) \sm (X'_{j'} \sm S'_{j'})| \cdot 2n 
\\ &\leq d'n'{}^2 + 16 \zeta n^2 \leq dn^2.
\end{align*} 
Then $G[X]$ is $d$-pair-complete with respect to the sets~$S_j$, another contradiction, so this proves (ii). 
\endproof

We can now prove the main lemma of this section, which allows us to find balanced perfect clique packings in graphs which are not $d$-splittable or $d$-pair-complete.

\begin{lemma} \label{theoremmatching}
Suppose that $1/n  \ll \alpha, \nu \ll d \ll 1/r$, $2 \leq p \leq r$ and $r! \mid n$. Let $G$ be an $r$-partite graph on vertex classes $V_1, \dots, V_r$ each of size $pn$, and let $J$ be the clique $p$-complex of $G$. Suppose that $G$ contains a spanning subgraph $G^*$ such that $G^*$ is not $d$-splittable and $\delta^*(G^*) \geq (p-1)n - \alpha n$. If $p \geq 3$, then $J_p$ contains a balanced perfect matching. If instead $p=2$ then $J_p$ contains a balanced perfect matching if 
\begin{enumerate}[(i)]
\item $G^*$ is not $d$-pair-complete, and
\item either $r < 4$ or for any distinct $i_1, i_2, i_3$ in $[r] \sm \{1\}$ there are at least $\nu n^4$ $4$-cycles $x_1x_{i_1}x_{i_2}x_{i_3}$ in $G$ with $x_1 \in V_1$ and $x_{i_j} \in V_{i_j}$ for $j \in [3]$.
\end{enumerate}
\end{lemma} 

\proof
Introduce new constants $\eps, \gamma$ and $d'$ with $1/n \ll \eps \ll \gamma \ll \alpha, \nu \ll d' \ll d \ll 1/r$. Let $\mc{I} := \binom{[r]}{p}$, so $\mc{I}$ is the family of possible indices of edges of $J_p$. For any perfect matching $M$ in $J_p$ and any index $A \in \mc{I}$, let $N_M(A)$ be the number of edges in $M$ with index $A$. Since any vertex of $J$ lies in precisely one edge of $M$, for any $i \in [r]$ we must have
\begin{equation} \label{eq:indices}
\sum_{A \in \mc{I}~:~i \in A} N_M(A) = pn.
\end{equation}
Let $N := rn / \binom{r}{p} = pn/ \binom{r-1}{p-1}$, and observe that $N$ is an integer. Our goal is to find a perfect matching $M$ in $J_p$ with $N_M(A) = N$ for every $A \in \mc{I}$. To do this, we will apply Theorem~\ref{partitematching} to find a perfect matching which is near-balanced, but first we need to put aside some configurations that can be used to correct the small differences in the number of edges of each index. This will be unnecessary if $p=r$ or $p=r-1$, as then any perfect matching in $J_p$ must be balanced. Indeed, for $p=r$ this is trivial, whilst for $p = r-1$ we note that by~(\ref{eq:indices}), for any $i \in [r]$, the number of edges of any perfect matching which do not contain a vertex of $V_i$ is $rn - pn = n$. So for the purpose of finding configurations we may suppose that $p \leq r-2$ (this is why we only require (ii) for $r \geq 4$).

Fix a set $S \sub [r]$ with $|S|=p-2$ and an ordered quadruple $T = (a,a',b,b')$ of distinct members of $[r] \sm S$. An \emph{$(S,T)$-configuration} consists of two vertex-disjoint copies $K$ and $K'$ of $K_{p-1}$, where $K$ has index $S \cup \{b\}$ and $K'$ has index $S \cup \{b'\}$, and vertices $v \in V_a$ and $v' \in V_{a'}$ such that $v$ and $v'$ are both adjacent to every vertex of $K \cup K'$. Given such an $(S,T)$-configuration, we can select two vertex-disjoint copies of $K_p$ in $G$ (that is, two disjoint edges of $J_p$) in two different ways. One way is to take $K \cup \{v\}$ of index $S \cup \{a,b\}$ and $K' \cup \{v'\}$ of index $S \cup \{a',b'\}$; we call this the \emph{unflipped} state. The other way is to take $K \cup \{v'\}$ of index $S \cup \{a',b\}$ and $K' \cup \{v\}$ of index $S \cup \{a,b'\}$; we call this the \emph{flipped} state. Let $\mc{W}$ be the set of all pairs $(S,T)$ as above. The first step in our proof is to find a collection $\CONFIG$ of pairwise vertex-disjoint configurations in $G$ which contains $p\gamma n$ $(S, T)$-configurations in $G$ for each $(S, T) \in \mc{W}$. 

Suppose first that $p \geq 3$. To choose an $(S,T)$-configuration in this case we first fix $c \in S$ and find vertices $v \in V_a$ and $v' \in V_{a'}$ with $|N(v) \cap N(v') \cap V_c| > (p-2 + 1/2p)n$. To see that this is possible, let $\mc{T}$ be the set of ordered triples $(v,v',w)$ with $v \in V_a$, $v' \in V_{a'}$, $w \in V_c$ and $vw,v'w \in G$. For each $w$ there are at least $\delta^*(G) \geq \delta^*(G^*)$ choices for each of $v$ and $v'$, so $|\mc{T}| \ge pn((p-1)n - \alpha n)^2$. Let $P$ be the set of ordered pairs $(v,v')$ that belong to at least $(p-2 + 1/2p)n$ triples of $\mc{T}$, i.e.\ have $|N(v) \cap N(v') \cap V_c| > (p-2 + 1/2p)n$. Then $|\mc{T}| \le |P|pn + (pn)^2(p-2 + 1/2p)n$, so $|P| \ge ((p-1)n - \alpha n)^2 - (p-2 + 1/2p)pn^2 > n^2/3$. Given such a pair $(v,v')$, we choose the remaining vertices of the configuration greedily, ending with the two vertices in $V_c$. For each vertex not in $V_c$, the number of choices is at least $pn - (p-1)(n-\delta^*(G)) > n/2$. For the two vertices in $V_c$, the number of choices for each is at least $|N(v) \cap N(v') \cap V_c| - (p-2)(n-\delta^*(G)) > n/3p$. Now we choose the collection $\CONFIG$ greedily. At each step, the configurations chosen so far cover at most $2p \cdot |\mc{W}| \cdot p\gamma n$ vertices. Since there are at least $n^2/3$ choices for the pair $(v,v')$ and at least $n/3p$ choices for any other vertex we are always able to choose an $(S, T)$-configuration which is vertex-disjoint from any configuration chosen so far, as required.

Now consider instead the case $p=2$, for which $J_p = G$. Recall that we can assume $r \ge 4$. Now an $(S,T)$-configuration consists of a $4$-cycle $xyzw$ with $x \in V_a, y \in V_b, z \in V_{a'}$ and $w \in V_{b'}$, where $T = (a, a', b, b')$ (we have $S=\es$). Note that the unflipped state of such a configuration has edges $xy$ and $zw$, and the flipped state has edges $xw$ and $yz$. If $a=1$ then by assumption there are at least $\nu n^4$ such $4$-cycles in $G$. For $a \neq 1$ we instead choose \emph{fake configurations}; for $S = \emptyset$ and $T = (a, a', b, b')$ a fake $(S,T)$-configuration consists simply of vertices $xyzw$ with $x \in V_a, y \in V_b, z \in V_{a'}$ and $w \in V_{b'}$ such that $xy$ and $zw$ are edges. A fake configuration should be thought of as a configuration which cannot be flipped. There are at least $(2n \delta^*(G))^2 \geq n^4$ fake configurations for each $T$, so similarly to before we may choose $\CONFIG$ greedily to consist of genuine $(S,T)$-configurations if $a = 1$, and fake $(S,T)$-configurations otherwise. Indeed, at any step the configurations chosen so far cover at most $8 |\mc{W}| \gamma n$ vertices, so for any $(S, T) \in \mc{W}$ at most $8|\mc{W}|\gamma n \cdot (2n)^3 < \nu n^4$  $(S, T)$-configurations share a vertex with a previously-chosen configuration.

Let $V' = \bigcup_{i \in [r]} V'_i$ be the set of all vertices not covered by configurations in $\CONFIG$. We now find a matching in $J_p$ covering $V'$. Note that the configurations in $\CONFIG$ cover $2p^2|\mc{W}|\gamma n$ vertices in total, equally many of which lie in each vertex class, so for any $i \in [r]$ we have $|V'_i| = pn'$, where $n' := n - 2p^2|\mc{W}|\gamma n/r$. Let $G'=G^*[V']$ and let $J'$ be the clique $p$-complex of $G'$. Then $\delta^*(G') \geq \delta^*(G^*) - \alpha n \geq (p-1)n' - 2\alpha n$. Furthermore, by Proposition~\ref{robustness} $G'$ is not $d'$-splittable, and if $p = 2$ then $G'$ is not $d'$-pair-complete. So we may apply Corollary~\ref{findmatching} with $d', n', 3\alpha$ and $\eps$ in place of $d, n, \alpha$ and $\gamma$ respectively to obtain that $J'_p$ must contain an $\eps$-balanced perfect matching. Extend this matching to a perfect matching $M^0$ in $J_p$ by adding the configurations in $\CONFIG$, all in their unflipped state. This adds equally many edges of each index, so $M^0$ is $\eps$-balanced, and so $N_{M^0}(A) = (1 \pm \eps)N$ for any index $A \in \mc{I}$.

It remains only to flip some configurations to correct these small imbalances in the number of edges of each index. To accomplish this, we shall proceed through the index sets $A \in \mc{I}$ in order. For each $A$ we will flip some configurations to obtain a perfect matching with precisely $N$ edges of each index set $A'$ considered in this order up to and including $A$. At the end of this process we will obtain a perfect matching with precisely $N$ edges of every index set. Let $\mc{A}_1 \sub \mc{I}$ consist of all members of~$\mc{I}$ of the form $[p-1] \cup \{i\}$ for some $p+2 \leq i \leq r$, and $\mc{A}_2 \sub \mc{I}$ consist of all members of $\mc{I}$ of the form $[p+1] \sm \{i\}$ for some $i \in [p+1]$. Let $\mc{A} = \mc{A}_1 \cup \mc{A}_2$ and $m = |\mc{I}| = \binom{r}{p}$, so $|\mc{A}| = r$. Note that if $p = 2$ then $\mc{A}$ contains all pairs $\{1,j\}$ with $2 \leq j \leq r$. Choose any linear ordering $A_1 \leq A_2 \leq \dots \leq A_{m}$ of the elements of $\mc{I}$ such that 
\begin{enumerate}[(i)]
\item for any $A \in \mc{I}$, $x \in A$ and $y \notin A$ with $y < x$ we have $\{y\} \cup A \sm \{x\} > A$, and
\item $\mc{A}$ is a terminal segment of $\mc{I}$.
\end{enumerate}
This is possible since for any $A \in \mc{A}$, $x \in A$ and $y \notin A$ with $y < x$ we have $\{y\} \cup A \sm \{x\} \in \mc{A}$. Note that for any $A \in \mc{I} \sm \mc{A}$ there exist $x, y \in A$ and $x', y' \notin A$ such that $x' < x$ and $y' < y$. Crucially, the sets $\{x'\} \cup A \sm \{x\}$, $\{y'\} \cup A \sm \{y\}$ and $\{x', y'\} \cup A \sm \{x, y\}$ each appear after $A$, by choice of the ordering. 

We now proceed through the sets $A_i$, $i \in [m-r]$ in turn (these are precisely the sets $A_i$ with $A_i \notin \mc{A}$). At each step $i$ we will flip at most $2^{i-1}\eps N$ configurations to obtain a perfect matching $M^i$ such that $N_{M^i}(A_j) = N$ for any $j \leq i$ and $|N_{M^i}(A_j) - N | \leq 2^i \eps N$ for any $j > i$. The matching $M^0$ satisfies these requirements for $i=0$, so suppose that we have achieved this for $A_1, \dots, A_{i-1}$, and that we wish to obtain $M^i$ from $M^{i-1}$. If $N_{M^{i-1}}(A_i) = N$ then we may simply take $M^i = M^{i-1}$, so we may assume $N_{M^{i-1}}(A_i) \ne N$. First suppose that $N_{M^{i-1}}(A_i) > N$. Since $A_i \notin \mc{A}$ we may choose $x, y \in A_i$ and $x', y' \notin A_i$ such that $x' < x$ and $y < y'$. Furthermore, if $p=2$ then we may also require that $x' = 1$. We let $T = (x', x, y', y)$, $S = A_i \sm \{x, y\}$, and choose a set of $N_{M^{i-1}}(A_i) - N$ unflipped $(S, T)$-configurations in $\CONFIG$ (we shall see later that this is possible). We flip all of these configurations (this is possible if $p=2$ since $x' = 1$, so these configurations are not fake-configurations). In doing so, we replace $N_{M^{i-1}}(A_i) - N$ edges of $M^{i-1}$ of index $A_i = S \cup \{x, y\} $ and $N_{M^{i-1}}(A_i) - N$ edges of $M^{i-1}$ of index $S \cup \{x', y'\}$ with $N_{M^{i-1}}(A_i) - N$ edges of $M^{i-1}$ of index $S \cup \{x', y\} $ and $N_{M^{i-1}}(A_i) - N$ edges of $M^{i-1}$ of index $S \cup \{x, y'\}$. The number of edges of each other index remains the same; note that this includes any index $A_j$ with $j < i$. Let $M^i$ be formed from $M^{i-1}$ by these flips; then $N_{M^i}(A_j) = N$ for any $j \leq i$ by construction. Also, for any $j > i$ we have
$$|N_{M^i}(A_j) - N| \leq |N_{M^{i-1}}(A_j) - N| + |N_{M^{i-1}}(A_i) - N| \leq 2 \cdot 2^{i-1} \eps N = 2^i \eps N,$$
as required. On the other hand, for $N_{M^{i-1}}(A_i) < N$ we obtain $M^i$ similarly by the same argument with $T = (x', x, y, y')$. To see that it is always possible to choose a set of $N_{M^{i-1}}(A_i) - N$ unflipped $(S, T)$-configurations from $\CONFIG$, note that there are $m-r = \binom{p}{r}-r$ steps of the process, and that at step~$i$ exactly $|N_{M^{i-1}}(A_i) - N | \leq 2^{i-1} \eps N$ members of $\CONFIG$ are flipped. So in total at most $2^{m-r} \eps N \leq p\gamma n$ members of $\CONFIG$ are flipped. Since $\CONFIG$ was chosen to contain at least this many pairwise vertex-disjoint $(S, T)$-configurations for any $(S, T) \in \mc{W}$, it will always be possible to make these choices.

At the end of this process, we obtain a perfect matching $M := M^{m-r}$ in $J_p$ such that $N_{M}(A) = N$ for every $A \notin \mc{A}$. It remains only to show that we also have $N_{M}(A) = N$ for any $A \in \mc{A}$. We first consider $A \in \mc{A}_1$, so $A = [p-1] \cup \{i\}$ for some $i \geq p+2$, and $A$ is the only index in $\mc{A}$ which contains $i$. Each of the $\binom{r-1}{p-1} - 1$ other sets $A'$ containing $i$ has $N_M(A')=N=pn/\binom{r-1}{p-1}$, so by (\ref{eq:indices}) we also have $N_M(A)=N$. Thus $N_M(A) = N$ for any $A \in \mc{A}_1$. Now consider $A \in \mc{A}_2$, so $A = [p+1] \sm \{i\}$ for some $i \in [p+1]$. Note that $A$ is the only member of $\mc{A}_2$ which does not contain $i$. Since $N_M(A')=N$ for any $i \in A' \notin \mc{A}_2$, by (\ref{eq:indices}) we have $\sum_{A' \in \mc{A}_2 \sm \{A\}} N_M(A') = pN$. This holds for all $A \in \mc{A}_2$, so $N_M(A) = N$ for all $A \in \mc{A}_2$. Thus $M$ is a balanced perfect matching in $J_p$, as required.
\endproof

We also need to be able to find balanced perfect matchings in pair-complete rows. Recall that an $r$-partite graph $G$ with vertex classes $V_1,\dots,V_r$ of size $2n$ is \emph{$d$-pair-complete} if there exist sets $S_j \sub V_j$, $j \in [r]$ each of size $n$ such that $d(S_i, S_j) \ge 1-d$, $d(V_i \sm S_i, V_j \sm S_j) \geq 1-d$ and $d(S_i, V_j \sm S_j) \le d$ for any $i,j \in [r]$ with $i \neq j$. This implies that almost all vertices in $S_i$ have few non-neighbours in $S_j$ and almost all vertices in $V_i \sm S_i$ have few non-neighbours in $V_j \sm S_j$. Lemma \ref{paircompletematching} will show that $G$ contains a balanced perfect matching under a similar condition, namely that there are sets $X_j \sub V_j$, $j \in [r]$ each of size approximately $n$ such that all vertices in $X_i$ have few non-neighbours in $X_j$ and all vertices in $V_i \sm X_i$ have few non-neighbours in $V_j \sm X_j$ for any $i \neq j$, provided that $X = \bigcup_{i \in [r]} X_i$ has even size. Note that we cannot omit the parity requirement on $|X|$, as there may be no edges between $X$ and $V(G) \sm X$. The proof uses the following characterisation of multigraphic degree sequences by Hakimi \cite{H}. A sequence $d=(d_1,\dots,d_n)$ with $d_1 \ge \cdots \ge d_n$ is \emph{multigraphic} if there is a loopless multigraph on $n$ vertices with degree sequence $d$. 

\begin{prop} \label{hakimi} (\cite{H})
A sequence $d=(d_1,\dots,d_n)$ with $d_1 \ge \cdots \ge d_n$ is multigraphic if and only if $\sum_{i=1}^n d_i$ is even and $d_1 \le \sum_{i=2}^n d_i$.
\end{prop}

\begin{lemma}\label{paircompletematching}
Suppose that $1/n  \ll \zeta \ll 1/r$ and $r-1 \mid 2n$. Let $G$ be an $r$-partite graph on vertex classes $V_1, \dots, V_r$ each of size $2n$. For each $i \in [r]$ let sets $X_i$ and $Y_i$ partition $V_i$ and be such that
\begin{enumerate}[(i)]
\item $|X_j|, |Y_j| = (1 \pm \zeta) n$ for any $j \in [r]$,
\item for any $i \neq j$, any $x \in X_i$ has at most $\zeta n$ non-neighbours in $X_j$ and any $y \in Y_i$ has at most $\zeta n$ non-neighbours in $Y_j$, and
\item $X := \bigcup_{i \in [r]} X_i$ has even size.
\end{enumerate}
Then $G$ contains a balanced perfect matching.
\end{lemma} 

\proof
Choose an integer $n'$ so that $(1-5\zeta)n \leq n' \leq (1-4\zeta)n$ and $r-1 \mid n'$. For each $j \in [r]$ let $a_j := |X_j| - n'$. So $3\zeta n \leq a_j \leq 6\zeta n$ for each $j$, and $a := \sum_{j \in [r]} a_j = |X| - rn'$ is even by (iii) and since $rn'$ is divisible by $r(r-1)$, which is even. By Proposition \ref{hakimi}, we can choose pairs $(i_\ell, j_\ell)$ with $i_\ell \neq j_\ell$ for $\ell \in [a/2]$ so that each $j \in [r]$ appears in precisely $a_j$ pairs. For each $\ell \in [a/2]$ choose a matching $M_\ell$ in $G$ which contains an edge of $G[X]$ of index $\{i_\ell, j_\ell\}$, and an edge of $G[Y]$ of each index $A \in \binom{[r]}{2} \sm \{\{i_\ell, j_\ell\}\}$. We also require that these matchings are pairwise vertex-disjoint. Such matchings may be chosen greedily using (ii), since together they will cover a total of $2 \cdot a/2 \cdot \binom{r}{2} \leq 3\zeta r^3 n \leq n/2$ vertices. Let $M = \bigcup_{\ell \in [a/2]} M_\ell$. Note that $M$ has $a/2$ edges of each index, and so covers $(r-1)a/2$ vertices from each $V_j$. For each $j \in [r]$ let $X'_j = X_j \sm V(M)$ and $Y'_j = Y_j \sm V(M)$, and let $X' = X \sm V(M)$ and $Y' = Y \sm V(M)$. Then $|X'_j| = |X_j| - a_j = n'$ and $|Y'_j| = 2n - |X'_j| - (r-1)a/2$ for any $j \in [r]$. Since $r-1$ divides $n'$ and $2n$, we conclude that $r-1$ divides $|X'_j|$ and $|Y'_j|$ for any $j \in [r]$. So we may partition $X'$ and $Y'$ into sets $X'_A$ and $Y'_A$ for each $A \in \binom{[r]}{2}$, where for each $A = \{i,j\}$, $X'_A$ consists of $n_X = n'/(r-1)$ vertices from each of $X'_i$ and $X'_j$, and $Y'_A$ consists of $n_Y = |Y_1'|/(r-1)$ vertices from each of $Y'_i$ and $Y'_j$. Now for any $A \in \binom{[r]}{2}$, the induced bipartite graph $G[X'_A]$ has minimum degree at least $n_X - \zeta n \geq n_X/2$ by (ii), and so contains a perfect matching $M_A$ of size $n_X$; by the same argument $G[Y'_A]$ contains a perfect matching $M'_A$ of size $n_Y$. Finally, $M \cup \bigcup_{A} (M_A \cup M'_A)$ is a perfect matching in $G$ with $a/2 + n_X + n_Y$ edges of each index, as required.
\endproof

However, there are some $r$-partite graphs $G$ on vertex classes $V_1, \dots, V_r$ of size $2n$ which satisfy $\delta^*(G) \geq n - \alpha n$ and are neither $d$-splittable nor $d$-pair-complete but do not contain a balanced perfect matching. For example, let $H$ be a graph with vertex set $\{x_i, y_i : i \in [r]\}$, where $x_ix_j$ and $y_iy_j$ are edges for any $i \neq j$ except when $i = 1, j = 2$, and $x_1y_2$ and $x_2y_1$ are edges. Form $G$ by `blowing up' $H$, that is, replace each $x_i$ and $y_i$ by a set of $n$ vertices, where $r-1 \mid 2n$, and each edge by a complete bipartite graph between the corresponding sets. Then $G$ contains $2rn$ vertices, and so any balanced perfect matching $M$ in $G$ contains $2n/(r-1)$ edges of each index.  Let $X$ be the vertices of $G$ which correspond to some vertex $x_i$ of $H$. Then $|X| = rn$ and any edge of $M$ covering a vertex of $X$ either covers two vertices of $X$ or has index $\{1,2\}$. There are exactly $2n/(r-1)$ edges in $M$ of the latter form, so we must have $rn = |X| \equiv 2n/(r-1)$ modulo $2$. We conclude that $G$ contains no balanced perfect matching if this congruence fails (e.g.\ if $r=5$ and $n$ is even but not divisible by $4$). However, $\delta^*(G) = n$, and it is easily checked that $G$ is neither $d$-splittable nor $d$-pair-complete if $d \ll 1$. Other examples can be obtained similarly. Note that $G$ does contain a perfect matching by Corollary~\ref{findmatching}, but this cannot be balanced. In later arguments, such graphs will only cause difficulties when the row-decomposition has one row of this form (with $p_i = 2$), and every other row has $p_i =1$. In such cases the following proposition will enable us to delete further copies of $K_k$ so that the subgraph remaining has a balanced perfect $K_{p_i}$-packing in every row $i$.

\begin{prop} \label{onepcrow}
Suppose that $1/n \ll \gamma, \alpha \ll 1/r$, $r!$ divides $n$ and $r \geq k \geq 2$. Let $G$ be an $r$-partite graph on vertex classes $X_1, \dots, X_r$ of size $kn$ which admits a $(k-1)$-row-decomposition into pairwise-disjoint blocks $X^i_j$ with $|X^i_j| = p_i n$ for $i \in [s]$ and $j \in [r]$, where $p_1 = 2$ and $p_2, \dots, p_{k-1} = 1$. Suppose that $G[X^1]$ contains a $\gamma$-balanced perfect matching $M'$, and that for any $i \neq i'$ and $j \neq j'$ any vertex $v \in X^i_j$ has at most $\alpha n$ non-neighbours in $X^{i'}_{j'}$. Then there exists an integer $D \leq 2 \gamma n$ and a $K_k$-packing $M$ in $G$ such that $r!$ divides $n - D$, $M$ covers $p_i D$ vertices in $X^i_j$ for any $i \in [s]$ and $j \in [r]$, and $G[X^1 \sm V(M)]$ contains a balanced perfect matching. 
\end{prop}

\proof
Note that $|M'|=rn$, and since $M'$ is $\gamma$-balanced, $|N_A(M')-N_B(M')| \le \gamma rn/\binom{r}{2}$ for any $A,B \in \binom{[r]}{2}$. So we may write $M' = M_0 \cup M_1$, where~$M_0$ is a balanced matching in $G[X^1]$, $r \cdot r!$ divides $|M_0|$, and $|M_1| \leq 2\gamma rn$. Note that $M_0$ covers $2|M_0|/r$ vertices in each $X^1_j$, $j \in [r]$, so $M_1$ covers the remaining $2D$ vertices in each $X^1_j$, where $D := n - |M_0|/r$. Note also that $|M_1| = Dr$, $D \leq 2 \gamma n$, and $r!$ divides $|M_0|/r = n-D$. We will construct a sequence $M_1,\dots,M_{k-1}$, where $M_i$ is a $K_{i+1}$-packing in $G[\bigcup_{i' \in [i]} X^{i'}]$ that covers $2D$ vertices in $X^1_j$ and $D$ vertices in $X^{i'}_j$ for each $2 \le i' \le i$ and $j \in [r]$. Each $M_i$ will have size $|M_i|=Dr$, and will be formed by adding a vertex of $X^i$ to each copy of $K_i$ in $M_{i-1}$.

Suppose that we have formed $M_i$ in this manner for some $i \geq 1$, and now we wish to form~$M_{i+1}$. Let $Z$ be the set of ordered pairs $(j,q)$ with $j \in [r]$ and $q \in [D]$. We form a bipartite graph $B$ whose vertex classes are $M_i$ and $Z$, where a copy $K'$ of $K_{i+1}$ in $M_i$ and a pair $(j,q)$ are connected if $j \notin i(K')$, that is, if $K'$ contains no vertex from $X_j$. Note that any $K'$ in $M_i$ has degree $(r-i-1)D$ in $B$, as there are $r-i-1$ choices for $j \in [r] \sm i(K')$ and $D$ choices for $q \in [D]$. The same is true of any pair $(j,q)$ in $Z$, as $|M_i|=Dr$, and $(i+1)D$ cliques in $M_i$ intersect $X_j$. So $B$ is a regular bipartite graph, and therefore contains a perfect matching. Let $f:M_i\to Z$ be such that $\{K'f(K'): K' \in M_i\}$ is a perfect matching in $B$. For each $K' \in M_i$ we extend $K'$ to a copy of $K_{i+2}$ in $G$ by adding a vertex from $X^{i+1}_j$, where $f(K')=(j,q)$ for some $q \in [D]$. Such extensions may be chosen distinctly, since any $K' \in M_i$ has at least $n - |K'|\alpha n \geq n/2 \geq D$ possible extensions to $X^{i+1}_j$ for any $j \notin i(K')$. Since every pair in $Z$ was matched to some $K' \in M_i$, the $K_{i+2}$-packing $M_{i+1}$ covers precisely $D$ vertices from $X^{i+1}_j$ for each $j \in [r]$.

At the end of this process we obtain a $K_k$-packing $M := M_{k-1}$ in $G$ which covers $2D$ vertices from $X^1_j$ and $D$ vertices from $X^i_j$ for $i \geq 2$ and $j \in [r]$. By construction $M_0$ is a balanced perfect $K_{p_i}$-packing in $G[X^1 \sm V(M)]$ and $r!$ divides $|M_0|/r = n - D$, as required. 
\endproof

\section{Covering bad vertices} \label{sec:proof2} 
The final ingredient of the proof is a method to cover `bad' vertices. The row-decomposition of $G$ will have high minimum diagonal density, which implies that most vertices have high minimum diagonal degree. However, we need to remove those `bad' vertices which do not have high minimum diagonal degree, so that we can accomplish step (iii) of the proof outline in Section~\ref{sec:outline}, which is to glue together the perfect clique packings in each row to form a perfect $K_k$-packing in $G$. Lemma~\ref{diagonalmindeg} will show that we can cover the bad vertices of $G$ by vertex-disjoint copies of $K_k$, whilst keeping the block sizes balanced and fixing the parity of any pair-complete row, so that each row contains a clique packing covering all of the undeleted vertices. First we need some standard definitions. An $s \times r$ \emph{rectangle} $R$ is a table of $rs$ cells arranged in $s$ rows and $r$ columns. We always assume that $s \le r$. A \emph{transversal} $T$ in $R$ is a collection of $s$ cells of the grid so that no two cells of $T$ lie in the same row or column. We need the following simple proposition. 

\begin{prop} \label{transversal}
Suppose that $R$ is an $s \times r$ rectangle, where $r \ge s \ge 0$ and $r \ge 1$. Suppose that at most $r$ cells of $R$ are coloured, such that at most one cell is coloured in each column, and at most $r-1$ cells are coloured in each row. Then $R$ contains a transversal of non-coloured cells.
\end{prop}

\begin{proof}
We proceed by induction on $s$. Note that the proposition is trivial for the cases $s=0,1$ and $r=s=2$. Now assume that $r \geq 3$ and $s \geq 2$. Choose a row with the most coloured cells, and select a non-coloured cell in this row (this is possible since at most $r-1$ of the $r$ cells in this row are coloured). Let $R'$ be the $(s-1) \times (r-1)$ subrectangle obtained by removing the row and column containing this cell. Then it suffices to find a transversal of non-coloured cells in $R'$. Note that $R'$ has at most $r-1$ coloured cells, as if $R$ had any coloured cells, then we deleted at least one coloured cell. Also, since $R$ had at most $r < 2(r-1)$ coloured cells, $R'$ must have at most $r-2$ cells coloured in any row, since we removed the row containing the most coloured cells. Since $R'$ contains at most one coloured cell in any column, the required transversal in $R'$ exists by the induction hypothesis.  
\end{proof} 

The remainder of this section is occupied by the proof of Lemma~\ref{diagonalmindeg}. This is long and technical, so to assist the reader we first give an overview. We start by imposing structure on the graph $G$, fixing a row-decomposition with high minimum diagonal density. Let $X^i_j$, $i \in [s]$, $j \in [r]$ be the blocks of this row-decomposition. Then we identify vertices that are `bad' for one of two reasons: (i) not having high minimum diagonal degree, or (ii) belonging to a pair-complete row and not having high minimum degree within its own half. We assign each bad vertex $v$ to a row in which it has as many blocks as possible that are `bad' for $v$, in that they have many non-neighbours of $v$, so that there are as few bad blocks for $v$ as possible in the other rows. We also refer to the resulting sets $W^i_j$, $i \in [s]$, $j \in [r]$ as blocks, as although they may not form a row-decomposition, since few vertices are moved they retain many characteristics of a row-decomposition. If the row-decomposition has type $p=(p_i:i \in [s])$ then $|W^i|$ is approximately proportional to $p_i$ for $i \in [s]$. We establish some properties of the $W^i_j$'s in Claim \ref{partprops}. Next, we show how to find various types of $K_k$ in $G$ that will form the building blocks for the deletions. These will be `properly distributed', in that they have $p_i$ vertices (the `correct number') in row $i$ for $i \in [s]$, or `$ij$-distributed' for some $i,j$, in that they have `one too many' vertices in row $i$ and `one too few' vertices in row $j$. They will also have an even number of vertices in each half of pair-complete rows so as to preserve parity conditions, except that sometimes we require a clique that is `properly-distributed outside of row $\ell$' to fix the parity of a pair-complete row $\ell$.  Claim~\ref{generalexpansion} analyses a general greedy algorithm for finding copies of $K_k$, and then Claim~\ref{choosekk} deduces five specific corollaries on finding building blocks of the above types. In Claim~\ref{balancingrows} we use $ij$-distributed cliques to balance the row sizes, so that the remainder of row $\ell$ has size proportional to $p_\ell$ for $\ell \in [s]$. There are two cases according to whether $G$ has the same row structure as the extremal example; if it does we also ensure in this step that the remainder of each half in pair-complete rows has even size. Next, in Claim~\ref{preparing} we put aside an extra $K_k$-packing that is needed to provide flexibility later in the case when there are at least two rows with $p_i \ge 2$. Then in Claim~\ref{coverdiv} we cover all remaining bad vertices and ensure that the number of remaining vertices is divisible by $rk \cdot r!$. Next, in Claim~\ref{balancingcolumns} we choose a $K_k$-packing so that equally many vertices are covered in each part $V_j$. Then in Claim~\ref{balancingblocks} we choose a final $K_k$-packing so that the remaining blocks all have size proportional to their row size. After deleting all these $K_k$-packings we obtain $G'$ with an $s$-row-decomposition $X'$ that satisfies conclusion (i) of Lemma \ref{diagonalmindeg}.  To complete the proof, we need to satisfy conclusion (ii), by finding a balanced perfect $K_{p_i}$-packing in row $i$ for each $i \in [s]$. We need to consider two cases according to whether or not there are multiple rows with $p_i \ge 2$; if there are, then we may need to make some alterations to the $K_k$-packing from Claim~\ref{preparing}. Finally, we apply the results of the previous section to find the required balanced perfect clique packings.

\begin{lemma} \label{diagonalmindeg}
Suppose that $1/n^+ \ll \alpha \ll 1/r$, $r \geq k \geq 3$ and $r > 3$. Let $\Part'$ partition a set $V$ into $r$ parts $V_1, \dots, V_r$ each of size $n^+$, where $rn^+/k$ is an integer. Suppose that $G$ is a $\Part'$-partite graph on $V$ with $\delta^*(G) \geq (k-1)n^+/k$. Suppose also that if $rn^+/k$ is odd and $k$ divides $n^+$ then $G$ is not isomorphic to the graph $\Gamma_{n^+, r, k}$ of Construction~\ref{fischereg}.
We delete the vertices of a collection of pairwise vertex-disjoint copies of $K_k$ from $G$ to obtain $V'_1, \dots, V'_r$, $V' := \bigcup_{j \in [r]} V'_j$ and $G' = G[V']$, such that $|V'_j| = kn'$ for $j \in [r]$, where $r! \mid n'$ and $n' \ge n^+/k - \alpha n$. We can perform this deletion so that $G'$ has an $s$-row-decomposition $X'$, with blocks $X'{}^i_j$ for $i \in [s]$, $j \in [r]$ of size $p_in'$, where $p_i \in [k]$ with $\sum_{i \in [s]} p_i = k$, with the following properties:
\begin{enumerate}[(i)]
\item For each $i, i' \in [s]$ with $i \neq i'$ and $j, j' \in [r]$ with $j \neq j'$, any vertex $v \in X'{}^i_j$ has at least $p_{i'}n' - \alpha n'$ neighbours in~$X'{}^{i'}_{j'}$.
\item For every $i \in [s]$ the row $G[X'{}^i]$ contains a balanced perfect $K_{p_i}$-packing.
\end{enumerate}
\end{lemma}

\proof
Introduce new constants with $1/n \ll d_0 \ll d_1 \ll \dots \ll d_k \ll \alpha$. Let $n$ be the integer such that $n^+-k+1 \le kn \le n^+$. Note that $kn-\delta^*(G^+) \leq \bfl{n^+/k} = n$, so 
\textno any vertex has at most $n$ non-neighbours in any other part. &(\dagger)

Let $X$ be formed from $V$ by arbitrarily deleting $n^+ - kn$ vertices in each part. We fix an $s$-row-decomposition of $G_1 := G[X]$. Recall that this consists of pairwise disjoint blocks $X^i_j \sub V_j$ with $|X^i_j| = p_i n$ for each $i \in [s]$, $j \in [r]$, for some $s \in [k]$ and positive integers $p_i$, $i \in [s]$ with $\sum_{i \in [s]} p_i = k$. By Proposition~\ref{iterate} we may fix this row-decomposition to have minimum diagonal density at least $1-k^2d_{s-1}$, and such that each row $G_1[X^i]$ is not $d_s$-splittable. Having fixed $s$ and the row-decomposition of $G_1$, introduce new constants with $d_{s-1} \ll d_0' \ll d_1' \ll \dots \ll d_{2s+2}' \ll d_{s}$. For each $i \in [s]$ with $p_i = 2$, let $d(i)$ denote the infimum of all $d$ such that $G_1[X^i]$ is $d$-pair-complete. This gives us at most $s$ values of $d(i)$, so we may choose $t \in [2s+2]$ such that there is no $i \in [s]$ with $d'_{t-2} < d(i) \leq d'_t$. We let $d := d'_{t-1}$, $d' := d'_t$, and introduce further new constants $\nu, \eta, \beta, \beta', \zeta, \gamma, \gamma', d''$ and~$\omega$ such that
$$1/n^+ \ll \nu \ll \eta \ll d \ll \beta \ll \beta' \ll \zeta \ll \gamma \ll \gamma' \ll d'' \ll d' \ll \omega, \alpha \ll 1/r \leq 1/k.$$ 
These are the only constants which we shall use from this point onwards. The purpose of these manipulations is that our fixed row-decomposition has three important properties. Firstly, it has minimum diagonal density at least $1-d$, since $d = d_{t-1}' \geq k^2d_{s-1}$. Secondly, any row $G_1[X^i]$ is not $d'$-splittable, since $d' = d'_t \leq d_s$. Thirdly, any row $G_1[X^i]$ which is $d'$-pair-complete is $d$-pair-complete, since $d(i) \leq d' = d'_t$ by definition of $d(i)$, and so $d(i) \leq d'_{t-2} < d'_{t-1} = d$ by choice of $t$. 

Suppose first that $s = 1$, so $G_1$ has only one row $X^1$, and $p_1 = k$. Fix $n - r! \leq n' \leq n$ such that $r! \mid n'$, and let $C = n^+ - kn'$. Then $0 \leq C \leq kr! + k$. Also, since $rn^+/k$ and $r(kn')/k$ are integers, $rC/k$ is also an integer. We choose $rC/k$ pairwise-disjoint copies of $K_k$ in $G$ which together cover $C$ vertices in each $V_j$. To see that this is possible note that, by $(\dagger)$, for any $A \in \binom{[r]}{k}$ we may greedily choose the vertices of a copy of $K_k$ in $G$ of index $A$; this gives at least $n$ choices for each vertex, of which at most $n/2$ (say) have been previously used, so some choice remains. Note that our use of $(\dagger)$ here is not tight, in the sense that the argument would still be valid if $n$ was replaced by $n + o(n)$ in the statement of $(\dagger)$. This will be true of all our applications of $(\dagger)$ except for that in Claim~\ref{balancingrows}. We delete all of these copies of $K_k$ from $G$, and let $G'$ be the resulting graph. We let $X'{}^1_j$ consist of the $kn'$ undeleted vertices of $V_j$ for each $j \in [r]$. Then the sets $X'{}^1_j$ for $j \in [r]$ form a $1$-row-decomposition of $G'$, which is not $d''$-splittable by Proposition~\ref{robustness}. Since $p_1 = k \geq 3$, $G' = G'[X'{}^1]$ contains a balanced perfect $K_k$-packing by Lemma~\ref{theoremmatching}.

We may therefore assume that $s \geq 2$. From each block $X^i_j$ we shall obtain a set $W^i_j$ by moving a small number of `bad' vertices between blocks, and reinstating the vertices deleted in forming $X$. As a consequence the sets $W^i_j$ will not form a proper row-decomposition (for example, blocks in the same row may have different sizes). However, since only a small number of vertices will be moved or reinstated, the partition of $V(G)$ into sets $W^i_j$ will retain many of the characteristics of the $s$-row-decomposition of $G_1$ into blocks $X^i_j$. We therefore keep the terminology, referring to the sets $W^i_j$ as `blocks', and the $W^i = \bigcup_j W^i_j$ and $W_j = \bigcup_i W_j^i$ as `rows' and `columns' respectively. Perhaps it is helpful to think of the sets $W^i_j$ as being containers which correspond to the blocks $X^i_j$, between which vertices may be transferred. It is important to note, however, that the blocks $X^i_j$ will remain unchanged throughout the proof. Furthermore, we shall sometimes refer to the row $G[W^i]$ simply as row~$i$, but we say that a row $i$ is \emph{pair-complete} if $G_1[X^i]$ is $d$-pair-complete. This means that the truth of the statement `row $i$ is pair-complete' depends only on our fixed row-decomposition of $G_1$, and not on the `blocks' $W^i_j$ or their subsets defined later. Note that if $p_i=2$ and row $i$ is not pair-complete then $G[X^i]$ is not $d'$-pair-complete.

We start by identifying the {\em bad} vertices, which may be moved to a different block. For each $i \in [s]$ and $j \in [r]$ let $B^i_j$ consist of all vertices $v \in X^i_j$ for which there exist $i' \neq i$ and $j' \neq j$ such that $|N(v) \cap X^{i'}_{j'}| \leq (1-\sqrt{d})p_{i'}n$.  We must have $|B^i_j| \leq rk\sqrt{d} p_in$, otherwise for some $i' \neq i$ and $j' \neq j$ there are more than $\sqrt{d} p_in$ vertices in $X^i_j$ with at most $(1 - \sqrt{d})p_{i'}n$ neighbours in $X^{i'}_{j'}$. Then  $d(X^i_j,X^{i'}_{j'}) < \sqrt{d} \cdot (1 - \sqrt{d}) + (1 - \sqrt{d}) \cdot 1 = 1 - d$ contradicts the minimum diagonal density of $G_1$.

Next, for each $i \in [s]$ for which row $i$ is pair-complete, by definition there are sets $T{}_j^i \subseteq X_j^i$ of size $n$ for each $j \in [r]$ such that $d(T{}_j^i, T{}_{j'}^i) \geq 1-d$ and $d(X_j^i \sm T{}_j^i, X_{j'}^i \sm T{}_{j'}^i) \geq 1-d$ for any $j \neq j'$. For each $j \in [r]$, we let $B'{}^i_j$ consist of all vertices $v \in T{}^i_j$ for which there exists $j' \neq j$ such that $|N(v) \cap T{}^i_{j'}| \leq (1-\sqrt{d})n$, and also all vertices $v \in X^i_j \sm T{}^i_j$ for which there exists $j' \neq j$ such that $|N(v) \cap (X^i_{j'} \sm T{}^i_{j'})| \leq (1-\sqrt{d})n$. We must have $|B'{}^i_j| \leq 2r \sqrt{d} n$, otherwise (without loss of generality) there exists some $j' \neq j$ for which more than $\sqrt{d}n$ vertices in $T{}^i_j$ have at most $(1 - \sqrt{d})n$ neighbours in $T{}^{i}_{j'}$. Then $d(T{}^i_j,T{}^{i}_{j'}) < \sqrt{d} \cdot (1 - \sqrt{d}) + (1 - \sqrt{d}) \cdot 1 = 1 - d$ contradicts the choice of the sets $T^i_j$. Thus we have bad sets $B^i_j$ and $B'{}^i_j$ for $i \in [s]$, $j \in [r]$. We also consider the $n^+ - kn$ deleted vertices in each part to be bad. Let $B$ be the set of all bad vertices. The remaining vertices $Y$ are {\em good}; let $Y^i_j = X^i_j \sm B$, $Y^i = \bigcup_{j \in [r]} Y^i_j$ and $Y_j = \bigcup_{i \in [s]} Y^i_j$ for each $i$ and $j$, so $Y = \bigcup_{i \in [s]} Y^i$. 

Let $v$ be any vertex of $G$. We say that a block $X^i_j$ is \emph{bad with respect to $v$} if $|N(v) \cap X^i_j| < p_i n - n/2$, that is, if $v$ has more than $n/2$ non-neighbours in $X^i_j$. So if $v$ is a good vertex, then the set of blocks which are bad with respect to $v$ is a subset of the set of blocks in the same row and column as $v$. Also, by $(\dagger)$ for any $v \in V(G)$ at most one block in each other column can be bad with respect to $v$. Similarly as with the notion of pair-completeness, this definition fixes permanently which blocks are bad with respect to a vertex $v$. We shall later sometimes refer to a `block' $W^i_j$ being bad with respect to $v$; this should always be taken to mean that $X^i_j$ is bad which respect to $v$. We say that a block is \emph{good with respect to $v$} if it is not bad with respect to $v$.

We now define the sets $W^i_j$ for each $i \in [s]$ and $j \in [r]$ as follows. Any vertex in $Y^i_j$ is assigned to $W^i_j$. It remains only to assign the bad vertices; each bad vertex $v \in V_j$ is assigned to $W^i_j$, where $i$ is a row containing the most blocks $X^i_j$ which are bad with respect to $v$ (if more than one row has the most bad blocks then choose one of these rows arbitrarily). For each pair-complete row $i$, we also modify the sets $T{}^i_j$ to form sets $S^i_j$. Indeed, $S^i_j$ is defined to consist of all vertices of $T{}^i_j \cap Y^i_j$, plus any vertex in $W^i_j \sm Y^i_j$ which has at least $n/2$ neighbours in $T{}^i_{j'}$ for some $j' \neq j$. We let $S^i := \bigcup_{j \in [r]} S^i_j$ for any such $i$. This completes the phase of the proof in which we impose structure on $G$. The next claim establishes some properties of the decomposition into `blocks' $W^i_j$.

\begin{claim} \label{partprops} \textbf{(Structural properties)}
\begin{enumerate}[({A}1)]
\item At most $\beta n/2$ vertices of $G$ are bad.
\item We have $Y^i_j \subseteq W^i_j$ and $(p_i-\beta/2)n \leq |Y^i_j| \leq |W^i_j| \leq (p_i+\beta/2)n$ for any $i \in [s]$ and $j \in [r]$. Furthermore, if row $i$ is pair-complete then $|Y^i_j \cap S^i_j|, |Y^i_j \sm S^i_j| \geq n - \beta n/2$ for any $j \in [r]$.
\item Let $v \in Y^i_j$. Then $v$ has at most $\beta n$ non-neighbours in $W^{i'}_{j'}$ for any $i' \neq i$ and $j' \neq j$. Furthermore, if row $i$ is pair-complete and $j' \neq j$ then $v$ has at most $\beta n$ non-neighbours in $S^i_{j'}$ if $v \in S^i_j$, and at most $\beta n$ non-neighbours in $W^i_{j'} \sm S^i_{j'}$ if $v \notin S^i_j$.
\item Let $v \in W^i_j$. Then $v$ has at most $2n/3$ non-neighbours in any block $W^{i'}_{j'}$ which is good with respect to $v$. Furthermore, if row $i$ is pair-complete then there is some $j' \neq j$ such that $v$ has at most $2n/3$ non-neighbours in $S^i_{j'}$ if $v \in S^i_j$, and at most $2n/3$ non-neighbours in $W^i_{j'} \sm S^i_{j'}$ if $v \notin S^i_j$.
\item For any $i$ with $p_i \geq 2$, there are at least $\gamma' n^{p_i+1}$ copies of $K_{p_i+1}$ in $G[Y^i]$. Furthermore, if row $i$ is pair-complete then there are at least $\gamma' n^3$ copies of $K_3$ in $G[Y^i \cap S^i]$ and at least $\gamma' n^3$ copies of $K_3$ in $G[Y^i \sm S^i]$.
\end{enumerate}
\end{claim}

\proof For (A1), note that since there were at most $rk\sqrt{d}p_in + 2r\sqrt{d}n$ bad vertices in each $X^i_j$, the total number of bad vertices is at most $rk (rk^2\sqrt{d}n + 2r\sqrt{d}n) + rk \leq \beta n/2$. For (A2), note that $Y^i_j \sub W^i_j$, and any vertex of $X^i_j \sm Y^i_j$ or $W^i_j \sm X^i_j$ is bad. Since $|X^i_j| = p_in$, and there are at most $\beta n/2$ bad vertices by (A1), we conclude that (A2) holds. For (A3), note that since~$v$ is good we have $|X^{i'}_{j'} \sm N(v)| \leq \sqrt{d}p_{i'}n$ for any $i' \neq i$ and $j' \neq j$. Since $|W^{i'}_{j'} \sm X^{i'}_{j'}| \leq \beta n/2$ by (A1), we conclude that $|W^{i'}_{j'} \sm N(v)| \leq \beta n$, as required. Similarly, if row $i$ is pair-complete and $v \in S^i_j$, then $v \in T{}^i_j$, so $v$ being good implies that $|T{}^i_{j'} \sm N(v)| \leq \sqrt{d}n$ for any $j' \neq j$. On the other hand, if $v \in W^i_j \sm S^i_j$, then $v \in X^i_j \sm T{}^i_j$, so $v$ being good implies that $|(X^i_{j'} \sm T{}^i_{j'}) \sm N(v)| \leq \sqrt{d}n$ for any $j' \neq j$. Any non-neighbour of $v$ in $S^i_{j'} \triangle T{}^i_{j'}$ or $W^i_{j'} \sm X^i_{j'}$ must be a bad vertex; by (A1) this completes the proof of (A3).

Next, for (A4) suppose that $W^{i'}_{j'}$ is good with respect to $v$. Recall that this means $|X^{i'}_{j'} \sm N(v)| \leq n/2$. Since any vertex in $W^{i'}_{j'} \sm X^{i'}_{j'}$ is bad, we find that $|W^{i'}_{j'} \sm N(v)| \leq n/2 + \beta n/2 \leq 2n/3$ by (A1). So suppose now that row~$i$ is pair-complete. If $v$ is good, then the `furthermore' statement holds by (A3), so we may suppose that $v \in W^i_j \sm Y^i_j$. If $v \in S^i_j$ then by definition $|N(v) \cap T{}^i_{j'}| \geq n/2$ for some $j' \neq j$, so $|T{}^i_{j'} \sm N(v)| \leq n/2$; then $|S{}^i_{j'} \sm N(v)| \leq 2n/3$ by (A1), since any vertex in $S{}^i_{j'} \triangle T{}^i_{j'}$ is bad. On the other hand, if $v \notin S^i_j$ then by definition $|N(v) \cap T{}^i_{j'}| < n/2$, so $|T{}^i_{j'} \sm N(v)| > n/2$, for any $j' \neq j$. By $(\dagger)$ this implies $|(X^i_{j'} \sm T{}^i_{j'}) \sm N(v)| < n/2$, and so $|(W^i_{j'} \sm S{}^i_{j'}) \sm N(v)| \leq 2n/3$ by (A1).

Finally, for (A5) suppose first that row $i$ is pair-complete, so $p_i = 2$. Then by (A2) we have $|Y^i_1 \cap S^i_1|, |Y^i_2 \cap S^i_2|, |Y^i_3 \cap S^i_3| \geq n - \beta n/2$. Furthermore, by (A3) any vertex in one of these three sets has at most $\beta n$ non-neighbours in each of the other two sets. So we may choose vertices $v_1 \in Y^i_1 \cap S^i_1$, $v_2 \in Y^i_2 \cap S^i_2 \cap N(v_1)$ and $v_3 \in Y^i_3 \cap S^i_3 \cap N(v_1) \cap N(v_2)$ in turn with at least $n - 3\beta n$ choices for each vertex. We conclude that there are at least $n^3/2$ copies of $K_3$ in $G[Y^i \cap S^i]$. The same argument applied to the sets $Y^i_1 \sm S^i_1$ shows that there are at least $n^3/2$ copies of $K_3$ in $G[Y^i \sm S^i]$. On the other hand, if row $i$ is not pair-complete, then we simply wish to find at least $\gamma' n^{p_i+1}$ copies of $K_{p_i+1}$ in $G[Y^i]$. Since $G_1[X^i]$ is not $d'$-splittable, by Lemma~\ref{splittable}(ii) (with $2 \gamma'$ in place of $\beta$) there are at least $2\gamma' n^{p_i + 1}$ copies of $K_{p_i+1}$ in $G_1[X^i]$. By (A1) at most $\beta (p_in)^{p_i + 1} \leq \gamma' n^{p_i + 1}$ such copies contain a bad vertex; this leaves at least $\gamma' n^{p_i +1}$ copies of $K_{p_i +1}$ in $G[Y^i]$.
\endproof 

In the next claim we analyse a general greedy algorithm that takes some fixed clique $K''$ in which all but at most one vertex is good, and extends it to a copy of $K_k$ with prescribed intersections with the blocks, described by the sets $A_1,\dots,A_s$.

\begin{claim} \label{generalexpansion} \textbf{(Extending cliques)}
Let $K''$ be a clique in $G$ on vertices $v_1, \dots, v_m$, where $v_2, v_3, \dots, v_m$ are good. Suppose that $A_1, \dots, A_s \sub [r]$ are pairwise-disjoint sets such that for each $q \in [m]$ there is some $i \in [s]$ and $j \in A_i$ such that $v_q \in W^i_j$. Suppose also that for each $i \in [s]$ one of the following five conditions holds:
\begin{enumerate} [(a)]
\item for every $j \in A_i$, $W^i_j$ contains some $v_q$ with $q \in [m]$, 
\item $|A_i| \leq p_i$ and $V(K'') = \emptyset$, 
\item $|A_i| \leq p_i$ and there is some block $W^i_j$ with $j \in A_i$ which is good with respect to $v_1$ and does not contain a vertex of $K''$, 
\item $|A_i| \leq p_i$ and $v_1 \in W^i$, 
\item $|A_i| < p_i$.
\end{enumerate}
Let $a := \sum_{i \in [s]} |A_i|$. Then for any $b_1, \dots, b_s \in \{0,1\}$ there are at least $\omega n^{a-m}$ copies $K'$ of $K_a$ in $G[Y \cup \{v_1\}]$ which extend $K''$ and satisfy the following properties.
\begin{enumerate}[(i)]
\item $K'$ intersects precisely those $W^i_j$ with $i \in [s]$ and $j \in A_i$. 
\item For any pair-complete row $i$ such that $|A_i| = 2$ and $V(K'') \cap W^i =\emptyset$ we have that $|V(K') \cap S^i|$ is even.
\item Consider any pair-complete row $i$ such that $|A_i| = 1$ and write $\{j\} = A_i$. If $W^i_j$ is good with respect to $v_1$ and does not contain a vertex of $K''$ then $|V(K') \cap S^i| = b_i$.
\end{enumerate}
\end{claim}

\proof
If $V(K'') \ne \es$, then by relabelling the columns $W_j$ if necessary, we may assume that for any $i \in [s]$ for which (c) holds, the block $W^i_{\max A_i}$ is good with respect to $v_1$ and does not contain a vertex of $K''$. We also note for future reference that the only properties of good vertices used in the proof of this claim will be those in (A3).

First we define the greedy algorithm for extending $K''$ to $K'$, and then we will show that we have many choices for the vertex at each step of the algorithm. We proceed through each column $V_j$, $j \in [r]$ in turn. If $j$ is not in $A_i$ for any $i \in [s]$, then we take no action, since $K'$ will not have a vertex in this column. Similarly, if $v_q \in W_j$ for some $q \in [m]$, then we again take no action, since we already have a vertex of $K'$ in this column, namely $v_q$; note that $v_q \in W^i_j$ for the unique $i$ such that $j \in A_i$, since the sets $A_\ell$ are pairwise-disjoint. Now suppose that $j \in A_i$ for some $i \in [s]$, and $V(K'') \cap W_j = \emptyset$. Let $v_1', \dots, v'_{t-1}$ be the vertices previously chosen by the algorithm (so not including $v_1, \dots, v_m$). We choose a vertex $v'_{t} \in Y^i_j \cap \bigcap_{\ell \in [m]} N(v_\ell) \cap \bigcap_{\ell \in [t-1]} N(v'_\ell)$, so $\{v_1, \dots, v_m, v'_1, \dots, v'_t\}$ induces a clique in $G$. If row $i$ is pair-complete, $|A_i| = 2$, $V(K'') \cap W^i =\emptyset$ and we have previously selected a vertex $v'_\ell$ in $W^i$, then we also add the requirement that $v'_t \in S^i$ if and only if $v'_\ell \in S^i$. If instead row $i$ is pair-complete and meets the conditions of (iii) then we instead add the requirement that $v'_t \in S^i$ if and only if $b=1$. After proceeding through every $j \in [r]$ we have a vertex of $W^i_j$ for every $i \in [s]$ and $j \in A_i$ (some of which are the vertices of $K''$). We let $K'$ be the subgraph of $G$ induced by these vertices. Then $K'$ is a clique of size $a$ in $G$ which extends $K''$ and satisfies~(i). 
The additional requirements on the choice of vertices from any pair-complete row $i$ imply that $K'$ must satisfy (ii) and (iii) also. 

Having defined the greedy algorithm, we will now show that there are at least $n/4$ choices at each step. First we consider the number of choices for some $v'_t$ in $Y^i_j$, where row $i$ is not a pair-complete row satisfying the conditions in (ii) or (iii). 
Note that since we are making this choice, $W_j$ does not contain a vertex of $K''$, so (a) does not apply to row $i$.
Let \[ P := \{v_1, \dots, v_m, v_1', \dots, v_{t-1}' \} \ \text{ and } \ P^i := (P \cap W^i) \sm v_1.\]
Then $|P^i| \leq |A_i| - 1 \leq p_i - 1$. We need to estimate $|Y^i_j \cap \bigcap_{v \in P} N(v)|$. Note that each $v \in P^i$ has $|Y^i_j \sm N(v)| \leq n$ by $(\dagger)$. If $V(K'') \ne \es$, then write $N' := Y^i_j \sm N(v_1)$, so $|N'| \leq n$ also; if $V(K'') = \es$ we let $N' = \es$. Observe that any vertex of $P \sm \{P^i \cup v_1\}$ is good, either by assumption (for $v_2, \dots, v_m$) or by selection (since the greedy algorithm only selects vertices from some $Y^{i'}_{j'}$). Then any vertex of $P \sm \{P^i \cup v_1\} = P \sm W^i$ lies in $Y^{i'}$ for some $i' \ne i$, and therefore has at most $\beta n$ non-neighbours in $Y^i_j$ by (A3). So $|Y^i_j \cap \bigcap_{v \in P} N(v)|$ is at least
\begin{equation}\label{eq:vertexchoices}
|Y^i_j| - |P^i|n - |N'| - r \beta n \stackrel{(A2)}{\geq} p_in - (p_i - 1) n - n - (r+1) \beta n = - (r+1) \beta n.
\end{equation}
Whilst this crude bound does not imply that we have even one possible choice for $v'_t$, we will now show that any of the assumptions (b)--(e) improves some part of the bound by at least $n/3$, which implies that there are at least $n/3 - (r+1) \beta n \geq n/4$ choices for $v'_t$. If (b) pertains to row $i$ (so $V(K'') = \es$), then we have $|N'| = 0$ instead of $|N'| \leq n$. If (d) or (e) pertains to row $i$, then we have the bound $|P^i| \leq p_i - 2$ in place of $|P^i| \leq p_i - 1$. The same is true if (c) pertains to row $i$, unless we are choosing the final vertex in $W^i$, that is $j = \max A_i$. Then our initial relabelling implies that $W^i_j$ is good with respect to $v_1$, so (A4) gives the bound $|N'| =  |Y^i_j \sm N(v_1)| \leq 2n/3$ in place of $|N'| \leq n$. In all cases the improvement of at least $n/3$ to (\ref{eq:vertexchoices}) yields at least $n/3 - (r+1)\beta n \geq n/4$ choices for $v'_t$.

It remains to consider the number of choices for some $v'_t$ in $Y^i_j$, where row $i$ is a pair-complete row satisfying the conditions in (ii) or (iii). Suppose first that row $i$ has the conditions of (ii), namely $|A_i| = 2 = p_i$ and $V(K'') \cap W^i = \emptyset$. Suppose also that we have previously selected a vertex $v'_\ell \in W^i$. Then we must ensure that $v'_t \in S^i_j$ if and only if $v'_\ell \in S^i$. Note that these conditions imply that either (b) or (c) pertains to row $i$. We will show that $|N'| \le 2n/3$ in either case. In case of (b) this holds because $N'$ is empty. In case of (c), $v'_t$ is the final vertex to be selected in row $i$, and so our initial relabelling implies that $W^i_j$ is good with respect to~$v_1$; then $|N'| \le 2n/3$ by definition. Since $v'_\ell$ is a good vertex (by choice), if $v'_\ell \in S^i$ then by (A3) we have $|S^i_j \sm N(v'_\ell)| \leq \beta n$, or if $v'_\ell \notin S^i$ then by (A3) we have $|(Y^i_j \sm S^i_j) \sm N(v'_\ell)| \leq \beta n$. In the former case we have $$\left|Y^i_j \cap S^i_j \cap \bigcap_{v \in P} N(v)\right| \geq |Y^i_j| - |N'| - r \beta n \geq n/4;$$
similarly, in the latter case we obtain $|(Y^i_j \sm S^i_j) \cap \bigcap_{v \in P} N(v)| \geq n/4$. So in either case there are at least $n/4$ possible choices for $v'_t$. Finally, suppose that row $i$ has the conditions of (iii), namely $|A_i| = 1$ and $W^i_j$ is good with respect to~$v_1$ (and $W^i_j$ does not contain a vertex of $K''$ since we are choosing $v'_t$). Then $|N'| \leq 2n/3$ by (A4), so $|Y^i_j \cap S^i_j \cap \bigcap_{v \in P} N(v)| \geq n/4$ and $|(Y^i_j \sm S^i_j) \cap \bigcap_{v \in P} N(v)| \geq n/4$, giving us at least $n/4$ choices for $v'_t$, regardless of the value of $b$.

In conclusion, there are at least $n/4$ suitable choices for each of the $a - m$ vertices chosen by the greedy algorithm in extending $K''$ to $K'$, giving at least $\omega n^{a-m}$ choices for $K'$, as required.\endproof

Now we describe the various types of $K_k$ that will form the building blocks for the deletions. Recall that the number of vertices in each row $W^i$ is approximately proportional to $p_i$. We say that a copy $K'$ of $K_k$ in $G$ is \emph{properly-distributed} if 
\begin{enumerate}[(i)]
\item $|V(K') \cap W^i| = p_i$ for each $i \in [s]$, and 
\item $|V(K') \cap S^i|$ is even for any pair-complete row $i \in [s]$.
\end{enumerate}
Also, for any $i, j \in [s]$ with $i \neq j$ we say that a copy $K'$ of $K_k$ in $G$ is $ij$-\emph{distributed} if
\begin{enumerate}[(i)]
\item $|V(K') \cap W^i| = p_i+1$, $|V(K') \cap W^j| = p_j-1$, and $|V(K') \cap W^\ell| = p_\ell$ for each $\ell \in [s] \sm \{i,j\}$, and 
\item $|V(K') \cap S^\ell|$ is even for any pair-complete row $\ell \neq i,j$.
\end{enumerate}
Note that an $ij$-distributed clique $K'$ has `one too many' vertices in row $i$, and `one too few' in row $j$. By deleting such cliques we can arrange that the size of each row $W^i$ is \emph{exactly} proportional to $p_i$. Thereafter we will only delete properly-distributed copies of $K_k$, so that this property is preserved. Also, condition (ii) in both definitions ensures that we preserve the correct parity of the halves in pair-complete rows. Finally, we say that a copy $K'$ of $K_k$ in $G$ is \emph{properly-distributed outside row $\ell$} if
\begin{enumerate}[(i)]
\item $|V(K') \cap W^i| = p_i$ for each $i \in [s]$, and 
\item $|V(K') \cap S^i|$ is even for any pair-complete row $i \neq \ell$.
\end{enumerate}
Thus $K'$ almost satisfies the definition of `properly-distributed', except that if row $\ell$ is pair-complete it may fail the parity condition for the halves. In the next claim we apply Claim~\ref{generalexpansion} to finding the building blocks just described.

\begin{claim} \label{choosekk} \textbf{(Building blocks)}
We can find copies of $K_k$ in $G$ as follows. 
\begin{enumerate}[(i)] 
\item Let $A_1, \dots, A_s \subseteq [r]$ be pairwise-disjoint with $|A_i| = p_i$ for each $i \in [s]$. Then there are at least $\omega n^k$ properly-distributed copies $K'$ of $K_k$ in $G$ such that for any $i \in [s]$, $K'$ intersects $W^i$ in precisely those $W^i_j$ with $j \in A_i$.
\item Any vertex $v \in V(G)$ lies in at least $\omega n^{k-1}/4$ properly-distributed copies of $K_k$ in $G$.
\item Let $i, j \in [s]$ be such that $p_i \geq 2$ and $i \neq j$. Then there are at least $\gamma n^k$ $ij$-distributed copies $K'$ of $K_k$ in $G[Y]$. Furthermore, if row $i$ is pair-complete then for any $b \in \{0,3\}$ there are at least $\gamma n^k$ such copies $K'$ of $K_k$ with $|V(K') \cap S^i| = b$.
\item Let $i, j \in [s]$ be such that $p_i = 1$ and $i \neq j$. Suppose $uv$ is an edge in $G[W^i]$ such that $u$ is a good vertex. Then there are at least $\omega n^{k-2}$ $ij$-distributed copies $K'$ of $K_k$ in $G$ that contain $u$ and $v$. Furthermore, if row $j$ is pair-complete then for any $b \in \{0,1\}$ there are at least $\omega n^{k-2}$ such copies $K'$ of $K_k$ with $|V(K') \cap S^j| = b$.
\item Let $i \in [s]$ be such that row $i$ is pair-complete, and $uv$ be an edge in $G[W^i]$ such that $u$ is good. Then there are at least $\omega n^{k-2}$ copies of $K_k$ in $G$ which contain both $u$ and $v$ and are properly-distributed outside row $i$.
\end{enumerate}
\end{claim}

\proof
For (i) we apply Claim~\ref{generalexpansion} to $A_1, \dots, A_s$ with $V(K'') = \emptyset$, so condition (b) pertains to all rows. This gives at least $\omega n^k$ copies $K'$ of $K_k$ in $G$ such that $K'$ intersects precisely those $W^i_j$ with $i \in [s]$ and $j \in A_i$, and $|V(K') \cap S^i|$ is even for any pair-complete row $i$. Since $|A_i| = p_i$ for each $i \in [s]$, each such $K'$ is properly-distributed. Next we consider (iii), as this is also a simple application of Claim~\ref{generalexpansion}; we will come back to (ii). We begin by choosing a copy $K''$ of $K_{p_i+1}$ in $G[Y^i]$. By (A5), there are at least $\gamma' n^{p_i+1}$ such copies, and if row $i$ is pair-complete, there are at least $\gamma' n^{p_i+1}$ such copies with precisely $b$ vertices in $S^i$. Fix any such $K''$ and let $A_i$ be the set of $q \in [r]$ such that $K''$ has a vertex in column $V_q$, so $|A_i| = p_i + 1$. Now choose pairwise-disjoint subsets $A_\ell \subseteq [r] \sm A_i$ with $A_j = p_j - 1$ and $|A_\ell| = p_\ell$ for every $\ell \in [s] \sm \{i,j\}$. We may apply Claim~\ref{generalexpansion} with $K''$ and the sets $A_\ell$ for $\ell \in [s]$, as condition (a) of the claim applies to row $i$, condition (e) applies to row $j$, and condition (c) applies to all other rows (since every vertex of $K''$ is good). We deduce that there are at least $\omega n^{k - p_i - 1}$ copies $K'$ of $K_k$ which extend $K''$ such that $|V(K') \cap W^\ell| = |A_\ell|$ for any $\ell \in [s]$ and $|V(K') \cap S^\ell|$ is even for any pair-complete row $\ell \neq i, j$. Each such $K'$ is $ij$-distributed,  
so in total we have at least $\gamma' \omega n^k \geq \gamma n^k$ copies $K'$ of $K_k$ with the required properties.

For (ii), (iv) and (v) we proceed similarly, but in each of these cases we have the possibility that the vertex $v$ might be bad,
so satisfying requirement (c) in Claim~\ref{generalexpansion} is no longer trivial. We consider the blocks as an $s \times r$
rectangle $R$, and colour those blocks which are bad with respect to $v$.  Our strategy will be to delete some appropriate rows
and columns from $R$, apply Proposition~\ref{transversal} to find a transversal $T$ of non-coloured blocks in the remaining
subrectangle $R'$, and then use $T$ to choose sets $A_i$ for $i \in [s]$ which meet the conditions of
Claim~\ref{generalexpansion}. (The notation $T$ is not intended to suggest any relationship with the sets $T^i_j$.) 
We note the following properties of $R$:
\begin{enumerate}[({R}1)]
\item any column not containing~$v$ has at most one coloured block,
\item at most $r-2$ coloured blocks lie outside the row containing $v$,
\item any row not containing $v$ has at most $r-3$ coloured blocks.
\end{enumerate}
To see this, recall that we observed (R1) earlier, and (R2) follows because $v$ lies in a row in which it has the most bad blocks. This also implies (R3), as any row not containing $v$ has at most $(r-1)/2 \le r-3$ coloured blocks, using the assumption that $r > 3$. 

For (ii), we let $\ell$ and $j_\ell$ be such that $v \in W^\ell_{j_\ell}$, and suppose first that row $\ell$ is not pair-complete (we postpone the case where $\ell$ is pair-complete to the end of the proof). We remove the row and column containing $v$ to obtain an $(s-1) \times (r-1)$ subrectangle $R'$. Since $R'$ contains at most $r-2$ coloured blocks, with at most one in each column, it contains a transversal $T$ of non-coloured blocks by Proposition~\ref{transversal}. For each $i \in [s] \sm \{\ell\}$ let $j_i$ be such that $T$ includes $W^i_{j_i}$. Choose any pairwise-disjoint subsets $A_i \subseteq [r]$ with $j_i \in A_i$ and $|A_i| = p_i$ for each $i \in [s]$. We may apply Claim~\ref{generalexpansion} to $A_i$, $i \in [s]$ with $V(K'') =\{v_1\} = \{v\}$, since condition~(d) pertains to row $\ell$, and condition~(c) pertains to all other rows. This yields at least $\omega n^{k-1}$ copies $K'$ of $K_k$ which extend $K''$ such that $|V(K') \cap S^i|$ is even for any pair-complete row $i$ and $|V(K') \cap W^i| = |A_i| = p_i$ for any $i \in [s]$. Each such $K'$ is a properly-distributed copy of $K_k$ in $G$ containing $v$.

For (iv), let $j_u$ and $j_v$ be the columns with $u \in V_{j_u}$ and $v \in V_{j_v}$, and let $A_i = \{j_u, j_v\}$. So $|A_i| = 2 = p_i + 1$. Suppose first that $p_j = 1$. In this case we delete rows $i$ and $j$ and columns $j_u$ and $j_v$ from $R$, leaving an $(s-2) \times (r-2)$ subrectangle $R'$. Then $R'$ has at most one coloured block in each column, at most $r-3$ coloured blocks in each row, and at most $r-2$ coloured blocks in total, so contains a transversal $T$ of non-coloured blocks by Proposition~\ref{transversal}. For each $\ell \in [s] \sm \{i,j\}$ let $q_\ell$ be the column such that $T$ includes $W^\ell_{q_\ell}$, and choose pairwise-disjoint subsets $A_\ell \subseteq [r] \sm A_i$ for $\ell \in [s] \sm \{i\}$ such that $q_\ell \in A_\ell$ and $|A_\ell| = p_\ell$ for each $\ell \in [s] \sm \{i,j\}$, and $A_j = \emptyset$. 
Now suppose instead that $p_j \geq 2$. In this case we delete row $i$ and columns $j_u$ and $j_v$ from $R$, leaving an $(s-1)
\times (r-2)$ subrectangle $R'$. Then $R'$ has at most one coloured block in each column, at most $r-3$ coloured blocks in each
row, and at most $r-2$ coloured blocks in total. Since $p_j \geq 2$ we have $r-2 \geq k-2 \geq s-1$, so we may again apply
Proposition~\ref{transversal} to find a transversal $T$ of non-coloured blocks in $R'$. For each $\ell \in [s] \sm \{i\}$
let $q_\ell$ be the column such that $T$ includes $W^\ell_{q_\ell}$, and choose pairwise-disjoint subsets $A_\ell \subseteq [r]
\sm A_i$ for $\ell \in [s] \sm \{i\}$ such that $q_\ell \in A_\ell$ for any $\ell \in [s] \sm \{i\}$, $|A_\ell| = p_\ell$ for each
$\ell \in [s] \sm \{i,j\}$, and $|A_j| = p_j - 1$. In either case we may apply Claim~\ref{generalexpansion} with $V(K'') =
\{u,v\}$, since conditions (a), (e) and (c) of Claim~\ref{generalexpansion} pertain to rows $i$, $j$ and all other rows
respectively. We deduce that there are at least $\omega n^{k-2}$ copies $K'$ of $K_k$ in $G$ which extend $K''$ such that $|V(K')
\cap W^\ell| = |A_\ell|$ for any $\ell \in [s]$ and $|V(K') \cap S^\ell|$ is even for any pair-complete row $\ell \neq i, j$. Each
such $K'$ is $ij$-distributed. Furthermore, if row $j$ is pair-complete, then we can apply Claim~\ref{generalexpansion}(iii),
as $W^j_{q_j} \in T$ is good with respect to $v$ by definition of $T$, to see that there are at least $\omega n^{k-2}$ such copies $K'$ with $|V(K') \cap S^j| = b$, as required.

For (v), we again let $j_u$ and $j_v$ be the columns with $u \in V_{j_u}$ and $v \in V_{j_v}$ and $A_i = \{j_u, j_v\}$. So $|A_i| = 2 = p_i$. We delete row $i$ and columns $j_u$ and $j_v$ from $R$, leaving an $(s-1) \times (r-2)$ subrectangle $R'$. Then $R'$ has at most one coloured block in each column, at most $r-3$ coloured blocks in each row, and at most $r-2$ coloured blocks in total. Since $p_i=2$ we have $r-2 \geq k-2 \geq s-1$, so as before Proposition~\ref{transversal} yields a transversal $T$ of non-coloured blocks in $R'$. We let $W^\ell_{q_\ell}$, $\ell \in [s] \sm \{i\}$ be the blocks of $T$, and choose pairwise-disjoint subsets $A_\ell \subseteq [r] \sm A_i$ such that $q_\ell \in A_\ell$ and $|A_\ell| = p_\ell$ for any $\ell \in [s] \sm \{i\}$. We may apply Claim~\ref{generalexpansion} with $V(K'') = \{u,v\}$, since conditions (a) and (c) of Claim~\ref{generalexpansion} pertain to row $i$ and all other rows respectively. So there are at least $\omega n^{k-2}$ copies $K'$ of $K_k$ in $G$ which extend $K''$ such that $|V(K') \cap W^\ell| = |A_\ell| = p_\ell$ for any $\ell \in [s]$ and $|V(K') \cap S^\ell|$ is even for any pair-complete row $\ell \neq i$. Each such $K_k$ is properly-distributed outside row $i$.

Finally we consider the case of (ii) where row~$\ell$ (containing $v$) is pair-complete. By (A4) and (A2) there are at least $n/4$ vertices $u \in Y^\ell \cap N(v)$ such that $|\{u,v\} \cap S^\ell|$ is even. For any such $u$, by (v) there are at least $\omega n^{k-2}$ copies $K'$ of $K_k$ which extend $uv$ and are properly-distributed outside row $\ell$. Each such $K'$ is properly-distributed since $V(K') \cap S^\ell = \{u,v\}$, so $v$ lies in at least $\omega n^{k-1}/4$ properly-distributed copies of $K_k$.
\endproof

We now use Claim~\ref{choosekk} to select several $K_k$-packings in $G$ whose deletion leaves a subgraph $G'$ which satisfies the conclusions of the lemma. The choice of these packings will vary somewhat according to the row structure of $G$. We say that $G$ has the \emph{extremal row structure} if it has the same row structure as the graphs $\Gamma_{n, r, k}$ of Construction~\ref{fischereg}, that is, if there is some row $i \in [s]$ with $p_i \geq 2$ which is pair-complete, and $p_{i'} = 1$ for any $i' \neq i$; this case requires special attention. The first step is the following claim, which balances the row sizes, so that the remainder of row $i$ has size proportional to $p_i$ for $i \in [s]$. This is the only step of the proof that requires the exact minimum degree condition (whether or not $G$ has the extremal row structure).

\begin{claim} \label{balancingrows} \textbf{(Balancing rows)}
There is a $K_k$-packing $M_1$ in $G$ such that $|M_1| \leq \beta rk^2n$ and $|W^i \sm V(M_1)|  = p_i (r n^+/k - |M_1|)$ is a constant multiple of $p_i$ for each $i \in [s]$. If $G$ has the extremal row structure we can also require that $|S^i \sm V(M_1)|$ is even for the unique $i \in [s]$ with~$p_i = 2$. 
\end{claim}

\proof
We write $a_i := |W^i| - p_irn^+/k$ for each $i \in [s]$, where we recall that $rn^+/k$ is an integer, so $a_i$ is an integer. Then $\sum_{i \in [s]} a_i = |V(G^+)| - rn^+ = 0$. Let $I^+ = \{i : a_i > 0\}$, $I^- = \{i : a_i < 0\}$ and $a := \sum_{a_i > 0} a_i = -\sum_{a_i < 0} a_i$. Recall that $n^+-k+1 \le kn \le n^+$ and $(1+\beta/2) p_i rn \geq |W^i| = \sum_{j \in [r]} |W^i_j| \geq (1-\beta/2)p_i rn$ for each $i$ by (A2). This gives $|a_i| \leq \beta rkn$ for each $i$ and $a \leq \beta rk^2n$. We choose sequences $(i_\ell: \ell \in [a])$ and $(i'_\ell: \ell \in [a])$ so that each $i \in I^+$ occurs $a_i$ times as some $i_\ell$, and each $i \in I^-$ occurs $-a_i$ times as some $i'_\ell$. We divide the remainder of the proof into two cases according to whether or not $G$ has the extremal row structure.

\medskip

\nib{Case 1: $G$ does not have the extremal row structure.}
We start by showing that for any $i \in I^+$ with $p_i = 1$ there is a matching $E^i$ in $G[W^i]$ of size $a_i$, each of whose edges contains a good vertex. We let $a_{ij} := |W^i_j| - n^+/k$ for each $j \in [r]$, so $|a_{ij}| \leq \beta n$ for each $j$ by (A2); note that the $a_{ij}$ are not necessarily integers. Then $|W^i_j| = n^+/k + a_{ij} \geq n + a_{ij}$. By $(\dagger)$, any $v \in W^i_{j'}$ for any $j' \neq j$ has at least $a_{ij}$ neighbours in $W^i_j$. We fix a matching $M$ in the bipartite graph $G[W^i_1 \cup W^i_2]$ which is maximal with the property that every edge contains a good vertex. We will show that $|M| \ge a_{i1}+a_{i2}$. For suppose otherwise. Then we may fix $u \in Y^i_1$ and $v \in Y^i_2$ which are not covered by $M$. As noted above, $u$ has at least $a_{i2}$ neighbours in $W^i_2$, and $v$ has at least $a_{i1}$ neighbours in $W^i_1$. Since $u$ and $v$ are good vertices, each of these neighbours must be covered by $M$, by maximality of $M$. Thus there must be an edge $w_1w_2$ in $M$ with $w_1 \in N(v)$ and $w_2 \in N(u)$. Now removing $w_1w_2$ from $M$ and adding the edges $w_1v$ and $w_2u$ contradicts the maximality of $M$. We conclude that $|M| \geq a_{i1}+a_{i2}$. Now we greedily extend $M$ to a maximal matching $E^i$ in $G[W^i]$ for which every edge of $E^i \sm M$ contains a vertex of $Y^i_1$. Then $E^i$ covers at least $a_{ij}$ vertices in $W^i_j$ for each $j \geq 3$, and so has total size at least $\sum_{j \in [r]} a_{ij} = a_i$; by removing edges we may take $E^i$ to have size exactly $a_i$, as claimed.

We now proceed through each $\ell \in [a]$ in turn. For each $\ell$, if $p_{i_\ell} = 1$ then we choose an edge $e \in E^{i_\ell}$ which has not been used for any $\ell' < \ell$ (this is possible as $i_\ell$ occurs $a_{i_\ell} = |E^{i_\ell}|$ times as some $i_{\ell'}$). We extend $e$ to an $i_\ell i'_\ell$-distributed copy $K^\ell$ of $K_k$ in $G$, so that $K^\ell$ is vertex-disjoint from any $K^{\ell'}$ previously selected for $\ell' < \ell$ and from any of the edges of $\bigcup_{i \in I^+} E^i$ other than $e$ itself. Then the extension of $e$ to $K^\ell$ must avoid a set of at most $ka + 2a \leq \beta' n$ `forbidden vertices'. By Claim~\ref{choosekk}(iv), there are at least $\omega n^{k-2} > \beta' n(rn^+)^{k-3}$ extensions of $e$ to an $i_\ell i'_\ell$-distributed copy of $K_k$, so we may choose $K'$ as desired. Similarly, if $p_{i_\ell} \geq 2$ then we choose an $i_\ell i'_\ell$-distributed copy $K^\ell$ of $K_k$ in $G$ which is vertex-disjoint from any previously chosen copy of $K_k$ and from the matchings $E^i$. As before this means we must avoid at most $\beta'n$ forbidden vertices. By Claim~\ref{choosekk}(iii), there are at least $\gamma n^k \geq \beta' n(rn^+)^{k-1}$ $i_\ell i'_\ell$-distributed copies of $K_k$ in $G$, so we can choose $K^\ell$ as desired. At the end of this process we have a $K_k$-packing $M_1 := \{K^\ell : \ell \in [a]\}$ with $|M_1| = a \le \beta rk^2n$.

Now consider any $i \in [s]$. If $i \in I^+$ (so $a_i > 0$) then each copy of $K_k$ in $M_1$ had $p_i$ vertices in $W^i$, except for $a_i$ copies which had $p_i +1$ vertices in $W^i$ (these are the $K^\ell$ with $i_\ell = i$). On the other hand, if $i \notin I^+$ (so $a_i \leq 0$) then each copy of $K_k$ in $M_1$ had $p_i$ vertices in $W^i$, except for $-a_i$ copies which had $p_i - 1$ vertices in $W^i$ (these are the $K^\ell$ with $i'_\ell = i$). In any case $|W^i \cap V(M_1)| = p_i |M_1| + a_i = p_i a + |W^i| - p_i r n^+/k$. Then $|W^i \sm V(M_1)| = p_i (rn^+/k - a)$ is a constant multiple of $p_i$ for each $i \in [s]$, so the proof of Case 1 is complete.

\medskip

\nib{Case 2: $G$ has the extremal row structure.}
Let $i^*$ be the unique row such that $p_{i^*} = 2$, so row~$i^*$ is pair-complete and $p_{i} = 1$ for any $i \neq i^*$. Recall that in this case we must satisfy the additional condition that $|S^{i^*} \sm V(M_1)|$ is even. We divide the proof into two further cases according to whether or not $a=0$.

\medskip

\nib{Case 2.1: $a > 0$.} We start by choosing copies $K^\ell$ of $K_k$ for $\ell<a$ exactly as in Case 1. For $\ell = a$ we consider three further cases, according to the value of $i^*$.

\medskip

\nib{Case 2.1.1: $i^* = i_a$.} As in Case 1, we use Claim~\ref{choosekk}(iii) to choose $K^a$ to be an $i_a i'_a$-distributed copy of $K_k$ in $G$ that does not include a forbidden vertex. However, we also choose $K^a$ so that $|S^{i^*} \sm \bigcup_{\ell \in [a]} V(K^{\ell})|$ is even. This is possible as by Claim~\ref{choosekk}(iii) there are at least $\gamma n^k$ $i_a i'_a$-distributed copies $K'$ of $K_k$ in $G$ for each choice of parity of $|V(K') \cap S^{i_a}|$. 

\medskip

\nib{Case 2.1.2: $i^* = i'_a$.} As in Case 1, we use Claim~\ref{choosekk}(iv) to extend an edge $e \in E^{i_a}$ to an $i_a i'_a$-distributed copy $K^a$ of $K_k$ in $G$ avoiding all forbidden vertices. Here we again impose the additional requirement that $|S^{i^*} \sm \bigcup_{\ell \in [a]} V(K^{\ell})|$ is even. This is possible as by Claim~\ref{choosekk}(iv) there are at least $\omega n^{k-2}$ extensions of $e$ to an $i_a i'_a$-distributed copy $K'$ of $K_k$ in $G$ for each choice of parity of $|V(K') \cap S^{i'_a}|$. 

\medskip

\nib{Case 2.1.3: $i^* \neq i_a, i'_a$.} In this case, instead of choosing the extension $K^a$ of $e \in E^{i_a}$ to be $i_a i'_a$-distributed, we extend $e$ to an $i_a i^*$-distributed copy $K^a$ of $K_k$ in $G$ which avoids any forbidden vertices (using Claim~\ref{choosekk}(iv) as before). The $K_k$-packing $M_1$ then covers `one too many' vertices in $W^{i'_a}$ and `one too few' in $W^{i^*}$. To correct this imbalance, we apply Claim~\ref{choosekk}(iii) to choose an $i^*i'_a$-distributed copy $K'$ of $K_k$ in $G$ such that $K'$ is vertex-disjoint from $M_1$ and $|S^{i^*} \sm (V(K') \cup V(M_1))|$ is even. We add $K'$ to $M_1$.

\medskip

In each of the cases 2.1.1, 2.1.2, 2.1.3 we have $|W^i \sm V(M_1)| = p_i (rn^+/k - a)$ for each $i \in [s]$ as in Case 1. We also have $|S^{i^*} \sm V(M_1)|$ even by construction, and $|M_1| \le \beta rk^2n$, noting in Case 2.1.3 that we could have improved the bound on $a$ to $a \leq \beta rk^2n-1$. This completes the proof in Case 2.1.

\medskip

\nib{Case 2.2: $a = 0$.} Note that $a_i = 0$ for each $i \in [s]$, so $|W^i|/p_i = rn^+/k$ is constant initially. If $|S^{i^*}|$ is even then $M_0 = \emptyset$ has the required properties, so we may assume $|S^{i^*}|$ is odd. Now suppose that for some $i \neq i^*$ there is an edge $uv$ in $G[W^i]$ such that $u$ is good. By Claim~\ref{choosekk}(iv) we may extend $uv$ to an $ii^*$-distributed copy $K^1$ of $K_k$ in $G$. Now apply Claim~\ref{choosekk}(iii) to choose an $i^*i$-distributed copy $K^2$ of $K_k$ in $G$ which does not intersect $K^1$ so that $|S^{i^*} \sm (V(K^1) \cup V(K^2)|$ is even. We then have $M_1 = \{K^1, K^2\}$ with the required properties. So we may assume that no such edge exists. Now suppose instead that $G[W^{i^*}]$ contains an edge $uv$ such that $u$ is good and $|\{u,v\} \cap S^{i^*}| = 1$. By Claim~\ref{choosekk}(v) we may extend $uv$ to a copy $K'$ of $K_k$ which is properly-distributed outside row $i^*$. We may take $M_1 = \{K'\}$, as then $|V(M_1) \cap S^i| = p_i$ for any $i \in [s]$ and $|S^{i^*} \sm V(M_1))|$ is even. So we may assume that no such edge exists. It follows that $|W^i_j| \leq n$ for any $i \neq i^*$ and $j \in [r]$. Indeed, if $|W^i_j| > n$ then choose $j' \neq j$ and a good vertex $u \in Y^i_{j'}$. Then $u$ has a neighbour $v \in W^i_j$, since $u$ has at most $n$ non-neighbours in any part, but we assumed no such edge $uv$ exists. Likewise, it follows that $|S^{i^*}_j|, |W^{i^*}_j \sm S^{i^*}_j| \leq n$ for any $j \in [r]$. For if (say) $|S^{i^*}_j| > n$, then choose $j' \neq j$ and $u \in Y^{i^*}_{j'} \sm S^{i^*}_{j'}$; then $u$ has a neighbour $v \in S^{i^*}_j$, giving an edge $uv$ that we assumed did not exist.

Now for any $j \in [r]$ we have $|W^{i^*}_j| \leq 2n$, $|S^{i^*}_j| \leq n$, $|W^{i^*}_j \sm S^{i^*}_j| \leq n$, and $|W^i_j| \leq n$ for any $i \neq i^*$. Since $|W_j| = |V_j| = n^+ \geq kn$, we have $n^+ = kn$, and equality holds in each of these inequalities. Then $k$ divides~$n^+$ and $rn^+/k = rn = |S^{i^*}|$ is odd. Now, any good vertex $u \in Y^i_j$ for $i \in [s]$ and $j \in [r]$ has $\delta^*(G) \geq (k-1)n$ neighbours in $V_{j'}$ for any $j' \neq j$. If $i \neq i^*$ then none of these can lie in $W^i_{j'}$, so we conclude that $N(u) \cap V_{j'} = V_{j'} \sm W^i_{j'}$. Similarly, we deduce that any $u \in Y^{i^*}_{j} \cap S^{i^*}_j$ has $N(u) \cap V_{j'} = V_{j'} \sm (W^{i^*}_{j'} \sm S^{i^*}_{j'})$ and any $u \in Y^{i^*}_{j} \sm S^{i^*}_j$ has $N(u) \cap V_{j'} = V_{j'} \sm S^{i^*}_{j'}$ for any $j' \neq j$. It follows that for any $i \neq i'$ and $j \neq j'$ any vertex $v \in W^i_j$ has $|N(v) \cap W^{i'}_{j'}| \geq |Y^{i'}_{j'}| \geq (1-\beta)p_{i'}n$. Furthermore, any vertex $v \in S^{i^*}_j$ has $|N(v) \cap S^{i^*}_{j'}| \geq |Y^{i^*}_{j'} \cap S^{i^*}_{j'}| \geq (1-\beta)n$, and any vertex $v \in W^i_{j'} \sm S^i_j$ has $|N(v) \cap (W^{i}_{j'} \sm S^i_{j'})| \geq |Y^{i}_{j'} \sm S^i_{j'}| \geq (1-\beta)n$. So every vertex of $G$ satisfies property (A3), which was shown to hold for all good vertices of $G$ in Claim~\ref{partprops}. This was the only property of good vertices used in the proof of Claim~\ref{generalexpansion}. So we may consider every vertex of $G$ to be good when applying Claim~\ref{generalexpansion}, and therefore also when applying Claim~\ref{choosekk}. The argument above then shows that we may choose $M_1$ as desired if there exists any edge $uv$ in $W^i$ for any $i \neq i^*$, or any edge $uv$ in $W^{i^*}$ such that $|\{u,v\} \cap S^{i^*}| = 1$. Since $G$ is not isomorphic to $\Gamma_{n^+, r, k}$, there must be some such edge. This completes the proof of Case 2.2, and so of the claim.
\endproof

In the next claim we put aside an extra $K_k$-packing $M_2$ that is needed to provide flexibility later in the case when there are at least two rows with $p_i \ge 2$. If instead $G$ has at most one row $i$ with $p_i \geq 2$ then we take $M_2=\es$. The proof of the claim is immediate from Claim~\ref{choosekk}(iii) so we omit it.

\begin{claim} \label{preparing} \textbf{(Preparing multiple rows with $p_i \geq 2$)}
Suppose that at least two rows of $G$ have $p_i \geq 2$. Then there is a $K_k$-packing $M_2$ vertex-disjoint from $M_1$ such that 
\begin{enumerate}[(i)]
\item every vertex covered by $M_2$ is good, and
\item $M_2$ consists of $\bcl{\eta n}$ $ij$-distributed copies of $K_k$ in $G[Y]$ for each ordered pair $(i, j)$ with $i, j \in [s]$, $i \neq j$ and $p_i, p_j \geq 2$.
\end{enumerate}
\end{claim}

\medskip

Note that (ii) implies that $|M_2| \leq \beta n$ and $|W^i \sm V(M_1 \cup M_2)|/p_i$ is constant for $i \in [s]$ by choice of $M_1$. The latter property will be preserved when we remove all subsequent packings, as they will only consist of properly-distributed cliques. Next we cover all remaining bad vertices and ensure that the number of remaining vertices is divisible by $rk \cdot r!$. For this, let $B'$ denote the set of all bad vertices in $G$ which are not covered by $M_1 \cup M_2$, and note that by (A1) we have $|B'| \leq \beta n/2$.

\begin{claim} \label{coverdiv} \textbf{(Covering bad vertices and ensuring divisibility)}
There is a $K_k$-packing $M_3$ vertex-disjoint from $M_1 \cup M_2$, consisting of at most $\beta n$ properly-distributed copies of $K_k$, such that
\begin{enumerate}[(i)]
\item $\bigcup_{i \in [3]} M_i$ covers every vertex of $B'$, and
\item the number of vertices not covered by $\bigcup_{i \in [3]} M_i$ is divisible by $rk \cdot r!$.
\end{enumerate}
\end{claim}

\proof
Fix $\beta n/2 \leq C \leq \beta n$ so that $C \equiv rn^+/k - |M_1 \cup M_2|$ modulo $r \cdot r!$ (recall that $rn^+/k$ is an integer). We will greedily choose $M_3$ to consist of $C$ copies of $K_k$. As long as some vertex $v \in B'$ remains uncovered, we choose some such $v$ and select a properly-distributed copy of $K_k$ containing $v$. Once every vertex in $B'$ is covered we repeatedly choose a properly-distributed copy of $K_k$ until we have~$C$ copies of $K_k$ in total. At any step there are at most $k(|M_1 \cup M_2| + C) \leq \beta' n$ forbidden vertices covered by $M_1 \cup M_2$ or a previously-chosen member of $M_3$, so there will always be a suitable choice available for the next member of $M$ by Claim~\ref{choosekk}(i) or~(ii). Fix such an~$M_3$. Then $|V(G) \sm \bigcup_{i \in [3]} V(M_i)|/k = rn^+/k - |M_1 \cup M_2| - C \equiv 0$ modulo $r \cdot r!$, so $|V(G) \sm \bigcup_{i \in [3]} V(M_i)|$ is divisible by $kr \cdot r!$. 
\endproof

The penultimate $K_k$-packing is chosen so that equally many vertices are covered in each part $V_j$.

\begin{claim} \label{balancingcolumns} \textbf{(Balancing columns)}
There is a $K_k$-packing $M_4$ vertex-disjoint from $\bigcup_{i \in [3]} M_i$, consisting of properly-distributed cliques, such that
\begin{enumerate}[(i)]
\item $|V_j \cap \bigcup_{i \in [4]} V(M_i)| = k|\bigcup_{i \in [4]} M_i|/r$ for any $j \in [r]$,
\item $|\bigcup_{i \in [4]} V(M_i)| \leq \beta' n/2$, and
\item $rk \cdot r!$ divides $|M_4|$.
\end{enumerate}
\end{claim}

\proof
We let $a'_j := |\bigcup_{i \in [3]} V(M_i) \cap V_j| - k|\bigcup_{i \in [3]} M_i|/r$ for each $j \in [r]$. So $\sum_{j \in [r]} |V(M_i)| = 0$, and $|a'_j| \leq |\bigcup_{i \in [3]} V(M_i)| \leq 2 \beta rk^3n$ for each $j \in [r]$. Note also that each $a_j'$ must be an integer since $r$ divides both $|V(G) \sm \bigcup_{i \in [3]} V(M_i)|$ (by choice of $M_3$) and $|V(G)|$ (by assumption). Similarly to the proof of Claim \ref{balancingrows}, we let $J^+ := \{j \in [r] : a'_j > 0\}$, $J^- := \{j \in [r] : a'_j < 0 \}$, and $a' := \sum_{j \in J^+} a'_j = - \sum_{j \in J^-} a'_j$. We form sequences $j_1, \dots, j_{a'}$ and $j'_1, \dots, j'_{a'}$ such that each $j \in J^+$ occurs $a'_j$ times as some $j_q$, and each $j \in J^-$ occurs $-a'_j$ times as some $j'_q$. Now, for each $q \in [a']$ choose sets $A_q, A'_q \in \binom{[r]}{k}$ such that $j_q \in A_q$ and $A'_q = (A_q \sm \{j_q\}) \cup \{j_q'\}$. For each $A \in \binom{[r]}{k}$ let $N_A$ and $N'_A$ be the number of times that $A$ appears as some $A_q$ or $A'_q$ respectively. Note that $N_A \leq a' \leq 2 \beta r^2k^3n$, so $ 2 \beta r^2k^3n + N'_A - N_A \geq 0$ for any $A \in \binom{[r]}{k}$. Fix an integer $C'$ such that $kr \cdot r!$ divides $C'$ and $ 2 \beta r^2k^3n \leq C' \leq  3 \beta r^2k^3n$. We choose $M_4$ to consist of $C' + N'_A - N_A$ properly-distributed copies of $K_k$ in $G$ with index $A$ for each $A \in \binom{[r]}{k}$. Since $\sum_{A \in \binom{[r]}{k}} (N'_A - N_A) = 0$, this will give us $|M_4| = C' \binom{r}{k}$; note that $rk \cdot r!$ then divides $|M_4|$. We also require that these copies are pairwise vertex-disjoint, and vertex-disjoint from $\bigcup_{i \in [3]} M_i$. We may choose $M_4$ greedily, since by Claim~\ref{choosekk}(i), for any $A \in \binom{[r]}{k}$ there are $\omega n^k$ copies of $K_k$ in $G$ with index $A$, and when choosing each copy we only need to avoid the at most $2 \beta rk^3n + \binom{r}{k} C'$ vertices covered by $\bigcup_{i \in [3]} M_i$ or previously chosen members of $M_4$. Now, since for any $q \in [a']$ we have $A_q \sm A'_q = \{j_q\}$ and $A'_q \sm A_q = \{j'_q\}$, for any $j \in [r]$ we have
$$|V(M_4) \cap W_j| = \sum_{A \in \binom{[r]}{k} : j \in A} (C' + N'_A - N_A) = C' \binom{r-1}{k-1} - a'_j = \frac{k|M_4|}{r} - a'_j.$$
Then $|V_j \cap \bigcup_{i \in [4]} V(M_i)| = k|M_4|/r - a'_j + |\bigcup_{i \in [3]} V(M_i) \cap V_j| = k|\bigcup_{i \in [4]} M_i|/r$ for any $j \in [r]$. Finally, $|\bigcup_{i \in [4]} V(M_i)| \leq C' \binom{[r]}{k} + 2\beta rk^3 n \leq \beta' n/2$. 
\endproof

The final $K_k$-packing is chosen so that the remaining blocks all have size proportional to their row size.

\begin{claim} \label{balancingblocks} \textbf{(Balancing blocks)}
There is a $K_k$-packing $M_5$ vertex-disjoint from $\bigcup_{i \in [4]} M_i$, consisting of properly-distributed cliques, such that defining $M = \bigcup_{i \in [5]} M_i$, $X'{}^i_j := W^i_j \sm V(M)$, $X'{}^i = \bigcup_{j \in [r]} X'{}^i_j$ and $X' = \bigcup_{i \in [s]} X'{}^i = V(G) \sm V(M)$, we have $|X'{}^i_j| = p_i n'$ for any $i \in [s]$ and $j \in [r]$, where $n' = |X'|/kr$ is an integer divisible by $r!$ with $n' \ge n-\zeta n/2$. Thus $X'{}^i_j$ forms a row-decomposition of $G':=G[X']$ of type $p$. Furthermore, any vertex $v \in X'{}^i_j$ has at most $\beta n \leq \alpha n'$ non-neighbours in any $X'{}^{i'}_{j'}$ with $i' \neq i$.
\end{claim}

\proof
For each $i \in [s]$ and $j \in [r]$ we let $W'{}^i_j = W^i_j \sm V(\bigcup_{i \in [4]} M_i)$, $W'{}^i := \bigcup_{j \in [r]} W'{}^i_j$, $W'_j := \bigcup_{i \in [s]} W'^i_j$, and $W' = \bigcup_{j \in [r]} W'_j$. We may fix an integer $D$ so that $|W'{}^i| = p_iD$ for any $i \in [s]$, since $|W^i \sm V(M_1 \cup M_2)|/p_i$ is constant for $i \in [s]$ and each clique in $M_3 \cup M_4$ is properly-distributed. Recall also that $|W'_j|$ is constant for each $j \in [r]$ by choice of $M_4$. Let $Q = (q_{ij})$ be the $s$ by $r$ integer matrix whose $(i,j)$ entry is $q_{ij} = |W'{}^i_j| - p_iD/r$. Then each row of $Q$ sums to zero. Furthermore, since $|W'_j|$ is constant for each $j \in [r]$ we have 
$$\sum_{i \in [s]} q_{ij} = \sum_{i \in [s]} (|W'{}^i_j| - p_iD/r) = |W'_j| - |W'|/r =0,$$
i.e.\ each column of $Q$ also sums to zero. We also have $\sum_{i,j} |q_{ij}| \leq \beta' n$ using (A2) and Claim
\ref{balancingcolumns}(ii).

We write $Q = \sum_{A \in \mc{Q}} A$, where $\mc{Q}$ is a multiset of matrices. Each matrix in $\mc{Q}$ is of the form $Q^{abcd}$, for some $a,c \in [s]$ and $b,d \in [r]$, defined to have $(i,j)$ entry equal to $1$ if $(i,j) = (a,b)$ or $(i,j) = (c,d)$, equal to $-1$ if $(i,j) = (a,d)$ or $(i,j) = (c,b)$, and equal to zero otherwise. To see that such a representation is possible, we repeat the following step. Suppose $q_{ab} > 0$ for some $a,b$. Since each row and column of $Q$ sum to zero, we may choose $c,d$ such that $q_{ad} < 0$ and $q_{cb} < 0$. Then $Q' := Q - Q^{abcd}$ is an $s$ by $k$ integer matrix in which the entries of each row and column sum to zero. Also, writing $Q' = (q'_{ij})$, we have $\sum_{i,j} |q'_{ij}| \leq \sum_{i,j} |q_{ij}| - 2$. By iterating this process at most $\sum_{i,j} |q_{ij}|/2$ times we obtain the all-zero matrix, whereupon we have expressed $Q$ in the required form with $|\mc{Q}| \leq \beta'n$, counting with multiplicity.

Let $\mc{P}$ denote the set of all families $\mc{A}$ of pairwise vertex-disjoint subsets $A_i \subseteq [r]$ with $|A_i| = p_i$ for $i \in [s]$. To implement a matrix $Q^{abcd} \in \mc{Q}$ we fix any two families $\mc{A}, \mc{A'}$ such that $b \in A_a$ and $d \in A_c$, and $\mc{A'}$ is formed from $\mc{A}$ by swapping $b$ and $d$. That is, $A_i'=A_i$ if $i \in [s] \sm \{a,c\}$, $A'_a = (A_a \sm \{b\}) \cup \{d\}$ and $A'_c = (A_c \sm \{d\}) \cup \{b\}$. For each $\mc{A} \in \mc{P}$ let $Q_\mc{A}$ be the number of times it is chosen as $\mc{A}$ for some $Q^{abcd}$, and $Q'_\mc{A}$ the number of times it is chosen as $\mc{A'}$ for some $Q^{abcd}$. Fix an integer $C''$ such that $kr \cdot r!$ divides $C''$ and $rk\beta' n \leq C'' \leq 2rk\beta' n$. For each $\mc{A} \in \mc{P}$ let $N_\mc{A} = C'' + Q_\mc{A} - Q'_\mc{A}$, so $N_\mc{A} \geq 0$. 

Now we greedily choose $M_5$ to consist of $N_\mc{A}$ copies of $K'$ for each $\mc{A} \in \mc{P}$, each of which will intersect each $W^i$ in precisely those $W^i_j$ such that $j \in A_i$. Then we will have $|M_5| = \sum_{\mc{A} \in \mc{P}} N_\mc{A} = C''|\mc{P}| \leq \zeta n/2$. When choosing any copy of $K_k$ we must avoid the vertices of copies of $K_k$ which were previously chosen for $M_5$, or which lie in $\bigcup_{i \in [4]} M_i$; there are at most $k \zeta n$ such vertices. By Claim~\ref{choosekk}(i), for any $\mc{A} \in \mc{P}$ there are at least $\omega n^k$ properly-distributed copies of $K_k$ in $G$ which intersect each $W^i$ in precisely those $W^i_j$ such that $j \in A_i$, so we can indeed choose $M_5$ greedily. This defines $M$, $X'{}^i_j$, $X'{}^i$, $X'_j$, $X'$, $n'$, $G'$ as in the statement of the claim. Note that $|M| \leq \zeta n$, and $kr \cdot r!$ divides $|X'|$ by Claims \ref{coverdiv}(ii) and \ref{balancingcolumns}(iii) and the choice of $C''$.

Finally, consider the number of vertices used in $W^a_b$, where $a \in [s]$, $b \in [r]$. If $\mc{A} \in \mc{P}$ is chosen uniformly at random, then $A_a$ is a uniformly random subset of size $p_a$, so contains $b$ with probability $p_a/r$. So if we chose $C''$ copies of each $\mc{A} \in \mc{P}$ we would choose $p_a N$ vertices in $W^a_b$, where $N := C'' |\mc{P}|/r$. However, since we choose $N_\mc{A}$ copies of each $\mc{A} \in \mc{P}$, we need to adjust by $Q_\mc{A}-Q'_\mc{A}$. These are chosen so that for each matrix $Q^{abcd} \in \mc{Q}$ we choose one more vertex in each of $W^a_b$ and $W^c_d$, and one fewer vertex in each of $W^a_d$ and $W^c_b$. Since $Q = \sum_{A \in \mc{Q}} A$ we thus use $p_a N + q_{ab}$ vertices in $W^a_b$. This gives $|X'{}^a_b| = |W'{}^a_b| - p_aN - q_{ab} = p_a(D/r - N)$. Writing $n' = D/r - N$, we have $n' = |X'|/kr$, so $n'$ is an integer divisible by $r!$. Note also that $n - n' \leq |V(M)|/kr \leq \zeta n/2$. Lastly, by choice of $M_3$ every bad vertex is covered by $M$, so any vertex $v \in X'{}^i_j$ has at most $\beta n \leq \alpha n'$ non-neighbours in any $X'{}^{i'}_{j'}$ with $i' \neq i$ and $j' \neq j$ by (A3).
\endproof

After deleting $M$ as in Claim \ref{balancingblocks}, we obtain $G'$ with an $s$-row-decomposition $X'$ that satisfies conclusion (i) of Lemma \ref{diagonalmindeg}. To complete the proof, we need to satisfy conclusion (ii), by finding a balanced perfect $K_{p_i}$-packing in row $i$ for each $i \in [s]$. 

Observe first that it is straightforward to find such a packing in row $i$ if $p_i \neq 2$. Indeed, for each $i \in [s]$ there is a trivial balanced perfect $K_1$-packing in $G[X'{}^i]$, namely $\{\{v\} : v \in X'{}^i\}$; this gives the desired packing if $p_i = 1$. Now suppose that $p_i \geq 3$ for some $i \in [s]$, and recall that $G_1[X^i]$ was not $d'$-splittable (this is true of any row of $G_1$). For any $j \in [r]$ we have $|X^i_j \triangle W{}^i_j| \leq 2\beta n$ by (A2) and $|W^i_j \triangle X'{}^i_j| = p_i n - p_i n' \leq (\zeta/2) p_i n$ by choice of $n'$, so $|X^i_j \triangle X'{}^i_j| \leq \zeta p_i n$. Proposition~\ref{robustness} then implies that $G[X'{}^i]$ is not $d''$-splittable, so $G[X'{}^i]$ contains a balanced perfect $K_{p_i}$-packing by Lemma~\ref{theoremmatching}.

It follows that if there is no row $i$ with $p_i = 2$, then we have conclusion (ii), so the proof is complete in this case. We may therefore assume that at least one row satisfies $p_i = 2$, and consider two cases according to whether or not any other row has $p_i \ge 2$.

\medskip

\nib{Case 1: There is exactly one row $i$ with $p_i \ge 2$.}

\medskip

\nib{Case 1.1: $G$ has the extremal row structure.} This means that $G$ has one pair-complete row $i^*$, and $p_i = 1$ for any $i \neq i^*$. There is a trivial balanced perfect $K_1$-packing in $G[X'{}^i]$ for any $i \neq i^*$, so it remains only to find a balanced perfect matching in $G[X'{}^{i^*}]$. Since row $i^*$ is pair-complete, we chose sets $S^{i^*}_j$ for $j \in [r]$ when forming the sets $W^i_j$. Let $S'{}^{i^*}_j := S^{i^*}_j \cap X'{}^i_j = S^{i^*}_j \sm V(M)$ for each $j \in [r]$. Then 
$$ (1+\beta) n \stackrel{(A2)}{\geq} |W^{i^*}_j| - |Y^{i^*}_j \sm S^{i^*}_j| \geq |S^{i^*}_j| \geq |S'{}^{i^*}_j| \geq |S^{i^*}_j| - |M| \stackrel{(A2)}{\geq} (1 - \beta/2)n - \zeta n,$$ 
and so $|S'{}^{i^*}_j| = (1 \pm 2\zeta) n'$ for each $j \in [r]$. Recall also that in this case we required that $|S^{i^*} \sm V(M_1)|$ was even when choosing $M_1$, and $M_2$ was empty. Furthermore, any copy $K'$ of $K_k$ in $M_3 \cup M_4 \cup M_5$ was chosen to be properly-distributed, so in particular $|K' \cap S^{i^*}|$ is even. We conclude that $|S^{i^*} \sm V(M)|$ is even. Since every bad vertex of $G$ was covered by $M_3 \subseteq M$, by (A3) for any $j' \neq j$ any vertex in $S'{}^{i^*}_j$ has at most $\beta n \leq 2\zeta n'$ non-neighbours in $S'{}^{i^*}_{j'}$, and any vertex in $X'{}^{i^*}_{j} \sm S'{}^{i^*}_j$ has at most $\beta n \leq 2\zeta n'$ non-neighbours in $X'{}^{i^*}_{j'} \sm S'{}^{i^*}_{j'}$ by (A3). By Lemma~\ref{paircompletematching} (with $2\zeta$ in place of $\zeta$) we conclude that $G[X'{}^{i^*}]$ contains a balanced perfect matching, completing the proof in this case.

\medskip

\nib{Case 1.2: $p_\ell = 2$ for some $\ell$, $p_{i'} =1$ for any $i' \neq \ell$, but row $\ell$ is not pair-complete.} Recall that this means that $G_1[X^\ell]$ was not $d'$-pair-complete. Moreover, just as for rows with $p_i \geq 3$, we have that $G_1[X^\ell]$ was not $d'$-splittable and that $|X^\ell_j \triangle X'{}^\ell_j| \leq \zeta (2n)$ for any $j \in [r]$. Proposition~\ref{robustness} therefore implies that $G[X'{}^\ell]$ is neither $d''$-splittable nor $d''$-pair-complete, so $G[X'{}^\ell]$ contains a $\nu$-balanced perfect matching by Corollary~\ref{findmatching}. Then Proposition~\ref{onepcrow} implies that there exists an integer $D \leq 2 \nu n'$ and a $K_k$-packing $M^*$ in $G$ such that $r!$ divides $n'' := n' - D$, $M^*$ covers $p_i D$ vertices in $X'{}^i_j$ for any $i \in [s]$ and $j \in [r]$, and $G[X'{}^\ell \sm V(M^*)]$ contains a balanced perfect matching. Note that since $n' \ge n - \zeta n/2$ we have $n'' \ge n^+/k - \zeta n$. We add the copies of $K_k$ in $M^*$ to $M$, and let $X''{}^i_j$, $X''{}^i$, $X''$, $G''$ be obtained from $X'{}^i_j$, $X'{}^i$, $X'$, $G'$ by deleting the vertices covered by $M^*$. This leaves an $s$-row-decomposition of $G''$ into blocks $X''{}^i_j$ of size $p_i n''$, in which $G[X''{}^\ell]$ contains a balanced perfect matching. As before $G[X''{}^i]$ contains a trivial balanced $K_1$-packing for every $i \neq \ell$. Finally, any vertex $v \in X'{}^i_j$ has lost at most $k|M^*| \leq \beta n$ neighbours in $X'{}^{i'}_{j'}$ for any $i \neq i'$ and $j \neq j'$, so still has at least $p_{i'}n' - 2\beta n \geq p_{i'}n'' - \alpha n''$. So the enlarged matching $M$, restricted blocks $X''{}^i_j$ and $G''$ satisfy (i) and (ii) with $n''$ in place of $n'$, which completes the proof in this case.

\medskip

\nib{Case 2: There are at least two rows $i$ with $p_i \ge 2$.}
In this case we modify the cliques in $M_2$ and the blocks $X'{}^i_j$ so that after these modifications $G[X'{}^i]$ contains a balanced perfect $K_{p_i}$-packing for each $i \in [s]$. We proceed through each $i \in [s]$ in turn. When considering any $i \in [s]$ we leave all blocks $X'{}^{i'}_j$ with $i' \neq i$ unaltered. We observed previously that if $p_i \neq 2$, then $G[X'{}^i]$ contains a balanced perfect $K_{p_i}$-packing, so no modifications are required to achieve the desired packings in such rows. 

This leaves only those rows $i \in [s]$ with $p_i = 2$ to consider. Suppose first that row $i$ is pair-complete. As in Case~1.1 we let $S'{}^i_j = S^i_j \cap X'^i_j$ for each $j \in [r]$, which gives $|S'{}^i_j| = (1\pm 2\zeta)n'$ for each $j \in [r]$, and for any $j' \neq j$, any vertex in $S'{}^i_j$ has at most $2 \zeta n'$ non-neighbours in $S'{}^i_{j'}$, and any vertex in $X'{}^i_{j} \sm S'{}^i_j$ has at most $2 \zeta n'$ non-neighbours in $X'{}^i_{j'} \sm S'{}^i_{j'}$. If $|S'{}^i|$ is even then $G[X'{}^i_j]$ contains a balanced perfect matching by Lemma~\ref{paircompletematching}. So we may suppose that $|S'{}^i|$ is odd. Fix any $i'$  with $i' \ne i$ and $p_{i'} \ge 2$. We choose any $i'i$-distributed copy $K'$ of $K_k$ in $M_2$, let $x$ be the vertex of $K'$ in $X'{}^i$, and let $j$ be such that $x \in X'{}^i_j$. By Claim~\ref{preparing}(i), every vertex in $K'$ is good, so at least $2n' - k\beta n$ vertices $y \in X'{}^i_j$ are adjacent to every member of $V(K') \sm \{x\}$. So we may choose a vertex $y \in X'{}^i_j \cap \bigcap_{v \in V(K') \sm \{x\}} N(v)$ such that $|\{x,y\} \cap S{}^i_j|$ is odd. We replace $K'$ in $M$ by the copy of $K_k$ in $G$ induced by $\{y\} \cup V(K') \sm \{x\}$, and replace $y$ by $x$ in $X'{}^i_j$ and $G'$. We also delete $y$ from $S'{}^i_j$ if $y \in S'{}^i_j$, and add $x$ to $S'{}^i_j$ if $x \in S{}^i_j$. So $|S'{}^i_j|$ is now even. Since $x$ is good, we may now apply Lemma~\ref{paircompletematching} as in Case~1.1 to find a balanced perfect matching in $G[X'{}^i]$.

Finally suppose that $p_i = 2$ and row $i$ is not pair-complete, that is $G_1[X^i]$ was not $d'$-pair-complete. Since $G_1[X^i]$ was also not $d'$-splittable, as before Proposition~\ref{robustness} implies that $G[X'{}^i]$ is neither $d''$-splittable nor $d''$-pair-complete.  Fix any $i'$  with $i' \ne i$ and $p_{i'} \ge 2$. Recall that $M_2$ contains $\bcl{\eta n}$ copies $K'$ of $K_k$ in $G[Y]$ which are $i'i$-distributed. We can fix $q \in [r]$ so that at least $\eta n/r$ such $K'$ have exactly one vertex in $Y^i_q$. We assign arbitrarily $\eta n/r$ such $K'$ to each ordered triple $(i_1, i_2, i_3)$ of distinct elements of $[r] \sm \{q\}$, so that at least $\eta n/r^4$ of the $K'$ are assigned to each triple. Now, fix any triple $(j_1, j_2, j_3)$ and any $K'$ which was assigned to it. 
Let $x$ be the vertex of $K'$ in $Y^i_q$, and consider paths $xx_1x_2x_3y$ of length $4$ in $G$ with $x_\ell \in X'{}^i_{j_\ell}$ for $\ell \in [3]$ and $y \in X'{}^i_q \cap \bigcap_{v \in V(K') \sm \{x\}} N(v)$. Since every vertex of $K' \sm \{x\}$ is good and does not lie in $W^i$, at most $k \beta n$ vertices $y \in X'{}^i_q$ fail to be adjacent to all of $V(K') \sm \{x\}$. Choosing $x_1, x_2, x_3$ and $y$ in turn, recalling $(\dagger)$ and $n' \ge n-\zeta n/2$, there are at least $2n' - n \geq (1-\zeta)n'$ choices for each $x_\ell$, and at least $2n' - n - (k-1)\beta n \geq (1-\zeta)n'$ choices for $y$. We obtain at least $n^4/2$ such paths, and so we may fix some $y = y(x)$ which lies in at least $n^3/5$ such paths. For each of these $n^3/5$ paths $xx_1x_2x_3y$ we add a `fake' edge between $y$ and $x_1$. Then allowing the use of fake edges, $y$ lies in at least $n^3/5$ $4$-cycles $x_1x_2x_3y$ of length $4$ in $G$ with $x_\ell \in X'{}^i_{j_\ell}$ for $\ell \in [3]$. We introduce fake edges in this manner for every $K'$ in $M_2$ which was assigned to the triple $(j_1, j_2, j_3)$, for every ordered triple $(j_1, j_2, j_3)$ of distinct elements of $[r] \sm \{q\}$. Let $G^*$ be the graph formed from $G[X{}'^i]$ by the addition of fake edges. Then by construction, for any triple $(j_1, j_2, j_3)$ there are at least $(\eta n/r^4)(n^3/5) \geq \nu n^4$ $4$-cycles $yx_1x_2x_3$ in $G^*$ with $y \in X'{}^i_q$ and $x_\ell \in X'{}^i_{j_\ell}$ for $\ell \in [3]$. Furthermore, $G^*$ contains a spanning subgraph $G[X{}'^i]$ which is neither $d''$-splittable nor $d''$-pair-complete, and has $\delta^*(G[X{}'^i]) \geq 2n' - n \geq n' - \zeta n$. Then $G^*$ contains a balanced perfect matching $M^*$ by Lemma~\ref{theoremmatching} (with $q$ in place of $1$). Of course, $M^*$ may contain fake edges. However, any fake edge in $M^*$ is of the form $y(x)x_1$, where $x_1$ is a neighbour of $x$, and $x$ lies in some $K'$ in $M_2$. Since $y(x)$ is uniquely determined by $x$ and $M^*$ is a matching, at most one edge in $M^*$ has the form $y(x)x_1$ for each $x$. Furthermore, by choice of $y = y(x)$, $\{y\} \cup V(K') \sm \{x\}$ induces a copy $K'_y$ of $K_k$ in $G$, and $xx_1$ is an edge. So we may replace $K'$ in $M_2$ by $K'_y$, replace $y$ in $X'{}^i_q$ by $x$ (note that $x$ is good), and the fake edge $yx_1$ in $M^*$ by the edge $xx_1$ of $G$. We carry out these substitutions for every fake edge in $M^*$, at the end of which $M^*$ is a perfect matching in $G[X'{}^i]$, which is balanced since each edge was replaced with another of the same index.

When considering row $i$ we only replace cliques of $M_2$ that are $i'i$-distributed for some $i' \neq i$. These cliques uniquely determine $i$, so do not affect the replacements for other rows. We may therefore proceed through every $i \in [s]$ in this manner. After doing so, the modified blocks $X'{}^i_j$ are such that $G[X'{}^i]$ contains a balanced perfect $K_{p_i}$-packing for any $i \in [s]$, i.e.\ this row-decomposition of the modified $G'$ satisfies condition (ii) of the lemma. Note that we still have $|X'{}^i_j| = p_in'$ for any $i \in [s]$. Since we only replaced good vertices with good vertices, every vertex in any modified block $X'{}^i_j$ is good, and so condition (i) of the lemma holds as in Claim \ref{balancingblocks}. This completes the proof of Lemma \ref{diagonalmindeg}.

\section{Completing the proof of Theorem~\ref{partitehajnalszem}} \label{proof}

In this section we combine the results of previous sections to prove Theorem~\ref{partitehajnalszem}. 
For this we also use a theorem of Daykin and H\"aggkvist~\cite{DH}, which gives a minimum vertex degree condition which is sufficient to ensure the existence of a perfect matching in a $k$-partite $k$-graph whose vertex classes each have size $n$.

\begin{theo}[\cite{DH}] \label{partitevertexdegree}
Suppose that $G$ is a $k$-partite $k$-graph whose vertex classes each have size $n$, in which every vertex lies in at least $\frac{k-1}{k}n^{k-1}$ edges. Then $G$ contains a perfect matching.
\end{theo}

We can now give the proof of Theorem~\ref{partitehajnalszem}, as outlined in Section~\ref{sec:outline}, which we first restate.

\medskip \noindent {\bf Theorem~\ref{partitehajnalszem}.} \emph{
For any $r \geq k$ there exists $n_0$ such that for any $n \geq n_0$ with $k \mid rn$ the following statement holds.
Let $G$ be an $r$-partite graph whose vertex classes each have size~$n$ such that  $\delta^*(G) \geq (k-1)n/k$. 
Then $G$ contains a perfect $K_k$-packing, unless $rn/k$ is odd, $k \mid n$, and $G \cong \Gamma_{n, r, k}$.}

\proof
First suppose that $k=2$, so a perfect $K_k$-packing is a perfect matching. If $r=2$ then $G$ is a bipartite graph with minimum degree at least $n/2$, so has a perfect matching. For $r \geq 3$ the result follows from Tutte's theorem, which states that a graph $G$ on the vertex set $V$ contains a perfect matching if and only if for any $U \subseteq V$ the number of odd components (i.e.\ connected components of odd size) in $G[V \sm U]$ is at most $|U|$. To see that this implies the theorem for $k=2$, suppose for a contradiction that there is some $U \subseteq V$ for which $G[V \sm U]$ has more than $|U|$ odd components. Clearly $|U| < |V|/2 = rn/2$. So by averaging $U$ has fewer than $n/2$ vertices in some $V_j$. Since $\delta^*(G) \geq n/2$, every $v \in V \sm V_j$ has a neighbour in $V_j \sm U$, so $G[V \sm U]$ has at most $|V_j \sm U| \leq n$ components. So we must have $|U| < n$. But then $U$ must have fewer than $n/r \leq n/3$ vertices in some $V_j$, so any $v \in V \sm V_j$ has more than $n/6$ neighbours in $V_j$. It follows that $G[V \sm U]$ has at most $5$ components, so $|U| < 5$. So any $v \in V \sm V_j$ actually has more than $n/2 - 5 > n/3$ neighbours in $V_j$, so $G[V \sm U]$ has at most $2$ components. So $|U| \leq 1$. If $|U| = 1$ then $|V \sm U|$ is odd, so $G[V \sm U]$ cannot have $2$ odd components. The only remaining possibility is that $|U| = 0$ and $G$ has $2$ odd components. Let $C^1$ and $C^2$ be these components, and for each $i \in [2]$ and $j \in [r]$ let $V^i_j$ be the vertices of $V_j$ covered by $C^i$. Then $|V^i_j| \geq \delta^*(G) \geq n/2$, so we deduce that $|V^i_j| = n/2$ and 
$G[V^i_j, V^i_{j'}]$ is a complete bipartite graph for any $i \in [2]$ and $j \neq j'$. So $|C^1| = rn/2$ is odd, $2$ divides $n$ and $G$ is isomorphic to $\Gamma_{n, r, 2}$, contradicting our assumption. 

We may therefore assume that $k \geq 3$. If $r = k = 3$ then Theorem~\ref{partitehajnalszem} holds by the result of \cite{MM}. So we may assume that $r > 3$. We introduce a new constant $d$ with $1/n \ll d \ll 1/r$. Since $r > 3$ and $r \geq k \geq 3$ we may apply Lemma~\ref{diagonalmindeg} (with $n$ and $d$ in place of $n^+$ and $\alpha$) to delete the vertices of a collection of pairwise vertex-disjoint copies of $K_k$ from $G$. Letting $V'$ be the set of undeleted vertices, we obtain $G' = G[V']$ and an $s$-row-decomposition of $G'$ into blocks $X'{}^i_j$ of size $p_in'$ for $i \in [s]$ and $j \in [r]$, for some $s \in [k]$ and $p_i \in [k]$ with $\sum_{i \in [s]} p_i = k$, such that $r! \mid n'$, $n' \ge n/k - dn$ and 
\begin{enumerate}[(i)]
\item for each $i, i' \in [s]$ with $i \neq i'$ and $j, j' \in [r]$ with $j \neq j'$, any vertex $v \in X'{}^i_j$ has at least $p_{i'}n' - dn'$ neighbours in~$X'{}^{i'}_{j'}$, and
\item for every $i \in [s]$ the row $G[X'{}^i]$ contains a balanced perfect $K_{p_i}$-packing $M^i$.
\end{enumerate} 
Note that we must have $|M^i| = rn'$ for any $i \in [s]$. 

Now we implement step (iii) of the proof outline, by constructing auxiliary hypergraphs, perfect matchings of which describe how to glue together the perfect $K_{p_i}$-packings in the rows into a perfect $K_k$-packing of $G$. We partition $[k]$ arbitrarily into sets $A_i$ with $|A_i| = p_i$ for $i \in [s]$. Let $\Sigma$ denote the set of all injective functions $\sigma : [k] \to [r]$. For each $i \in [s]$ we partition $M^i$ into sets $E_\sigma^i$ of size $N := \frac{rn'(r-k)!}{r!}$ for $\sigma \in \Sigma$, so that each member of $E_\sigma^i$ has index $\sigma(A_i)$. To see that this is possible, fix any $B \in \binom{[r]}{p_i}$. Since $M^i$ is balanced, $rn'/\binom{r}{p_i}$ members of $M^i$ have index $B$. Note that there are $\frac{p_i!(r-p_i)!}{(r-k)!}$ members of $\Sigma$ with $\sigma(A_i) = B$. Since $\frac{p_i!(r-p_i)!}{(r-k)!} \cdot N = r n'/\binom{r}{p_i}$ we may choose the sets $E_\sigma^i$ as required. 
For every $\sigma \in \Sigma$, we form an auxiliary $s$-partite $s$-graph $H_\sigma$ with vertex classes $E_\sigma^i$ for $i \in [s]$, where a set $\{e_1,e_2,\dots, e_s\}$ with $e_i \in E^i_\sigma$ for each $i \in [s]$ is an edge of $H_\sigma$ if and only if $xy \in G$ for any $i \neq j$, $x \in e_i$ and $y \in e_j$. Thus $H_\sigma$ has $N$ vertices in each vertex class, and $e_1e_2\dots e_s$ is an edge of $H_\sigma$ if and only if $\bigcup_{j \in [s]} V(e_j)$ induces a copy of $K_k$ in $G$. 

Next we show that each $H_\sigma$ has high minimum degree. Fix $\sigma \in \Sigma$ and $i \in [s]$. Then for any $e_i \in E^i_\sigma$, $e_i$ is a copy of $K_{p_i}$ in $G[X'{}^i]$ with index $\sigma(A_i)$, and so by (i) each vertex $x \in V(e_i)$ has at most $dn'$ non-neighbours in each $X'{}^{\ell}_{j}$ with $\ell \neq i$ and $j \notin \sigma(A_i)$. So at most $p_idn'$ vertices of $X'{}^{\ell}_j$ are not neighbours of some vertex of~$e_i$. Now we can estimate the number of $(s-1)$-tuples $(e_j \in E_\sigma^j: j \in [s] \sm \{i\})$ so that $\{e_1, \dots, e_s\}$ is not an edge of $H_\sigma$. There are fewer than $k$ choices for $j \in [s] \sm \{i\}$, at most $p_i dn'$ elements $e_j \in E_\sigma^j$ that contain a non-neighbour of some vertex of~$e_i$, and at most $N^{s-2}$ choices for $e_{j'} \in E_\sigma^{j'}$, $j' \in [s] \sm \{i,j\}$. So the number of edges of $H_\sigma$ containing $e_i$ is at least $N^{s-1} - kp_idn'N^{s-2} \geq \frac{k-1}{k} N^{s-1}$. Since $e_i$ was arbitrary, Lemma~\ref{partitevertexdegree} gives a perfect matching in each $H_\sigma$. This corresponds to a perfect $K_k$-packing in $G$ covering the vertices of $\bigcup_{i \in [s]} V(E_\sigma^i)$, where $V(E_\sigma^i)$ denotes the vertices covered by members of $E_\sigma^i$. Combining these perfect $K_k$-packings, and adding the pairwise vertex-disjoint copies of $K_k$ deleted in forming $G'$, we obtain a perfect $K_k$-packing in $G$.
\endproof
  
\section{Concluding remarks}

By examining the proof, one can obtain a partial stability result, i.e.\ some approximate structure for any $r$-partite graph $G$ with vertex classes each of size~$n$, where $k \mid rn$, such that $\delta^*(G) \geq (k-1)n/k - o(n)$, but $G$ does not contain a perfect $K_k$-packing. To do this, note that under this weaker minimum degree assumption, the $n$ in $(\dagger)$ must replaced by $n + o(n)$. We now say that a block $X^i_j$ is bad with respect to $v$ if $v$ has more than $n/2 + o(n)$ non-neighbours in $X^i_j$, so it is still true that at most one block in each column is bad with respect to a given vertex. Then each of our applications of $(\dagger)$ proceeds as before, except for in Claim~\ref{balancingrows}, where we used the exact statement of $(\dagger)$ (i.e. the exact minimum degree hypothesis). This was needed to choose a matching $E^i$ in $G[W^i]$ of size $a_i$, each of whose edges contains a good vertex, for each $i \in I^+$ with $p_i = 1$. If we can choose such matchings $E^i$ then the rest of the proof to give a perfect $K_k$-packing still works under the assumption $\delta^*(G) \geq (k-1)n/k - o(n)$. So we can assume that there is some $i \in I^+$ with $p_i = 1$ for which no such matching exists. Since the number of bad vertices and $a_i$ are $o(n)$, it follows that $W^i$ is a subset of size $rn/k + o(n)$ with $o(n^2)$ edges, i.e.\ we have a sparse set of about $1/k$-proportion of the vertices. On the other hand, this is essentially all that can be said about the structure of $G$, as any such $G$ with an independent set of size $rn/k+1$ cannot have a perfect $K_k$-packing.

\medskip

\textbf{Acknowledgement.} We thank the anonymous referees for their helpful comments.

\medskip

{\footnotesize \obeylines \parindent=0pt 

\begin{tabular}{lll}

Peter Keevash                       &\ &  Richard Mycroft \\
Mathematical Institute	    &\ &  School of Mathematics \\
University of Oxford	    &\ &  University of Birmingham \\
Oxford, UK                  &\ &  Birmingham, UK \\
\end{tabular}
}

{\footnotesize \parindent=0pt

\it{E-mail addresses}:
\tt{Peter.Keevash@maths.ox.ac.uk}, \tt{r.mycroft@bham.ac.uk}}

\end{document}